\documentclass{amsart}

\usepackage{amssymb}
\usepackage{mathtools}
\usepackage[utf8]{inputenc}
\usepackage[shortlabels]{enumitem}
\usepackage{url}
\usepackage{graphicx}
\usepackage{color}
\usepackage[colorlinks=true]{hyperref}
\hypersetup{colorlinks=true, citecolor=blue, linkcolor=blue, filecolor=magenta, urlcolor=cyan}
\numberwithin{equation}{section}

\setlist[enumerate,1]{label=(\alph*)}
\setlist[enumerate,2]{label=(\roman*)}

\newtheorem{thm}{Theorem}[section]
\newtheorem{prop}[thm]{Proposition}
\newtheorem{Rem}[thm]{Remark}
\newtheorem{lem}[thm]{Lemma}
\newtheorem{cor}[thm]{Corollary}
\theoremstyle{definition}
\newtheorem{defn}[thm]{Definition}

\begin{document}

\title[Explicit bounds for the prime number theorem]{Explicit bounds for the prime number theorem}

\author{Andrew Fiori}

\address{Department of Mathematics and Statistics, University of Lethbridge,
Lethbridge, Alberta, Canada}

\email{andrew.fiori@uleth.ca}

\thanks{Andrew Fiori thanks and acknowledges the University of Lethbridge for their financial support as well as the support of NSERC Discovery Grant RGPIN-2020-05316.}

\author{Mikko Jaskari}

\address{Department of Mathematics and Statistics, University of Turku, 20014 Turku, Finland}

\email{mikko.m.jaskari@utu.fi}

\thanks{Mikko Jaskari thanks and acknowledges the Research Council of Finland grants no. 346307, 364214 and 370133 and the University of Turku Graduate School UTUGS for their financial support.}



\begin{abstract}
We study the problem of finding tight explicit bounds on sums over zeros of $L$-functions near the $1$-line with a variety of zero-free regions and zero-density estimates.
As an application to demonstrate the techniques, we apply them to the case of the Riemann zeta function and the prime number theorem to obtain new best estimates on the error terms in the prime number theorem. We show that for $x\geqslant21$ we have
\begin{multline*}
|\psi(x)-x| < 0.239\, x \exp\bigl(-0.1982767\,(1+\log\log\log(x)/(15\log\log(x)))\\
{}\times\log(x)^{3/5}\log\log(x)^{-1/5}\bigr).
\end{multline*}
Combined with the zero-free half-plane $\Re(s)>7/8$ announced by OpenAI in October 2026, the same methods give $|\psi(x)-x| < 0.04934\,x^{3/4}\log(x)^2 + 0.02636425\,x^{7/8}$ for $x\geqslant 21$.
\end{abstract}
\maketitle

\section{Introduction}
The Chebyshev prime counting function
\begin{align*}
\psi(x) = \sum_{p^n \leqslant x} \log(p)
\end{align*}
and the study of its bounds have been important topics in number theory ever since the proof of the prime number theorem. The prime number theorem states that
$\psi(x) \sim x$
and in terms of the error term
\begin{align}
E_{\psi}(x) = \left| \frac{\psi(x)-x}{x} \right|,
\end{align}
this is equivalent to
\begin{align}
\lim_{x\to \infty}E_{\psi}(x) = 0.
\end{align}
The well-known von Mangoldt formula, valid for $x>1$ not a prime power,
\begin{align*}
\psi(x) = x - \sum_{\substack{\rho \\ \zeta(\rho)=0, \\ 0<\operatorname{Re}(\rho)<1}} \frac{x^{\rho}}{\rho} -\frac{1}{2}\log\left(1-x^{-2}\right)-\log(2\pi),
\end{align*}
where the sum is over the nontrivial zeros of the Riemann zeta function and is understood as $\lim_{T\to\infty}\sum_{|\operatorname{Im}(\rho)|<T}$ (it is not absolutely convergent). This formula expresses the relation between the Chebyshev function and the nontrivial zeros of the Riemann zeta function. Hence in order to understand the bounds of $E_{\psi}(x)$ we essentially have to bound the sum
\begin{align*}
\sum_{\substack{\rho \\ \zeta(\rho)=0, \\ 0<\operatorname{Re}(\rho)<1}} \frac{x^{\rho-1}}{\rho}.
\end{align*}
In 1899, de la Vallée Poussin proved that, for some $c'>0$,
\begin{align*}
E_{\psi}(x)=O(\exp(-c'\sqrt{\log(x)}))
\end{align*}
by using his zero-free region
\begin{align*}
\zeta(s) \not= 0 \textrm{ for } \operatorname{Re}(s) > 1-\frac{1}{R'\log(|\operatorname{Im}(s)|+2)}, \textrm{ for some } R'>0.
\end{align*}
The best current explicit version of this result is due to Fiori, Kadiri and Swidinsky \cite{FKS1}, where the bound
\begin{align}\label{eq:fksresult}
E_{\psi}(x) < 9.22022(\log(x))^{3/2}\exp(-0.8476836\sqrt{\log(x)}) \textrm{ for }  x >2
\end{align}
was obtained. 
We note that there have been subsequent improvements to both the zero-free region (see  \cite{MTYZFR} and \cite{BTY-ZFR}) and zero-density (see \cite{GF-Thesis}) results used.

More generally, in \cite{ING} it was shown that a zero-free region
\[  \zeta(s)\not=0 \textrm{ for } \operatorname{Re}(s) > 1 - v(\log(|\operatorname{Im}(s)|)) \]
leads, for every $\epsilon>0$, to an error term
\[E_{\psi}(x)=O(\exp((-1/2+\epsilon)\omega(x))) \]
where $\omega(x) = \min_{u\geqslant  w_0}(v(u)\log(x)+u)$, the minimum being taken over the range $u\geqslant w_0$ on which the zero-free region is valid (and $v$ is defined and positive).
For example, for the Korobov-Vinogradov zero-free region
\[ v(u) = cu^{-2/3}\log(u)^{-1/3} \]
one has
\[ \omega(x) \sim  \left(\frac{c^3 5^6}{2^2 3^4}\right)^{1/5} \frac{\log(x)^{3/5}}{\log\log(x)^{1/5}}. \]

In \cite{PIN} it was shown that incorporating zero-density estimates can improve the bound, again for every $\epsilon>0$, to
\[E_{\psi}(x)=O(\exp((-1+\epsilon)\omega(x))). \]
An explicit version of this bound for a Korobov-Vinogradov type zero-free region was established by Johnston and Yang in  \cite{JY-PNT}, where the authors obtain explicitly:
\begin{align} E_{\psi}(x) < 0.026 (\log(x))^{1.801}\exp(-0.1853 (\log(x))^{3/5}\log\log(x)^{-1/5} ) \textrm{ for } x \geqslant 23. \label{JYprevious} \end{align}
The problem of removing the $\epsilon$ from the error terms asymptotically has been examined very recently in 
\cite{JZDET} and \cite{BellottiNewZFR}. In particular they work to obtain results of the form 
\[E_{\psi}(x)=O(\exp(-\omega(x))). \]
The work of \cite{JZDET} explains how this relates to zero-density results, whereas \cite{BellottiNewZFR} establishes a zero-density result which would give this.

Recognizing that a bound of the form given in \cite{JY-PNT} is cleaner and likely easier to use in applications, in this article we take a slightly different approach from that of either \cite{JY-PNT} or \cite{BellottiNewZFR}. That is, in the context of Korobov-Vinogradov type zero-free regions we are specifically aiming for a fully explicit bound
\[ E_{\psi}(x)=O\left(\exp\left(-\left(\frac{c^3 5^6}{2^2 3^4}\right)^{1/5} \frac{\log(x)^{3/5}}{\log\log(x)^{1/5}}\right)\right).\]
Our proofs also establishes tighter, explicit bounds of the shape predicted by \cite{JZDET}, that is, for some explicit $q>0$, of the form
\[ E_{\psi}(x)=O\left(\omega(x)^{q}\exp(-\omega(x))\right). \]
We expect these to be used less in future applications since $\omega(x)$ in general has no closed form.

The main technical contribution of this paper is to provide explicit and effective bounds on sums over zeros of $L$-functions near the $1$-line under the hypothesis of having both a zero-free region and a zero-density estimate. These sums arise in many different applications of explicit formulas, for many classes of $L$-functions, and we aim for our results to be directly applicable in these applications even though we illustrate them primarily with bounds for $\psi(x)$.
In this article we treat fairly general forms of both zero-free regions and zero-density estimates including in particular the forms which have been established for large classes of functions.

The sums we consider are of the form
\begin{align} \sum_{\rho \in \mathcal{R}} \frac{x^{\beta-1}}{ |\rho|^m} \end{align}
for $m>0$, some region $\mathcal{R}$ as below, and where we write $\rho=\beta+it$ for the zeros and restrict to positive ordinates $t>0$ (the contribution of the conjugate zeros being identical in absolute value).

The regions that need to be considered have parameters depending on the precise sum being considered and are typically of the form:
\begin{enumerate}
\item  $\mathcal{R}_1=\{ \rho = \beta + it\;|\; \beta > \sigma_1(x),\; 0<t<T(x) \}$;
\item $\mathcal{R}_2= \{ \rho = \beta + it\;|\; \beta > \sigma_0,\;t > T(x) \}$;
\item $\mathcal{R}_3=  \{ \rho = \beta + it\;|\; \sigma_1(x) \geqslant \beta > \sigma_0,\;0<t < T(x) \}$;
\item $\mathcal{R}_4=  \{ \rho = \beta + it\;|\; \beta \leqslant \sigma_0,\; 0<t < T(x) \}$; and
\item $\mathcal{R}_5=  \{ \rho = \beta + it\;|\; \beta \leqslant \sigma_0,\; t > T(x) \}$.
\end{enumerate}
We assume that $\sigma_0 > 1/2$. The functions $T(x)$ and $\sigma_1(x)$ would typically be chosen so as to balance out the contribution of the sums from the various regions. Particularly $\sigma_1(x)$ balances the contributions from $\mathcal{R}_1$ and $\mathcal{R}_3$, and $T(x)$ either balances the contributions of $\mathcal{R}_1$ and $\mathcal{R}_2$, if $\mathcal{R}_2$ needs to be considered, or balances the contribution from $\mathcal{R}_1$ against an error term arising from either smoothing an explicit formula or Perron's formula depending on the method being used.

We show how to apply our work in computing actual bounds in the case of an explicit Korobov-Vinogradov zero-free region by improving the bound \eqref{JYprevious}; we prove the following theorem.
\begin{thm}\label{OurMainTheorem}
For $x\geqslant x_0 \geqslant e^3$ we have 
    \[ E_{\psi}(x) < L(x_0)\exp\Biggl(-D\biggl(1+\frac{\log\log\log(x)}{15\log\log(x)}\biggr)\log(x)^{3/5} (\log\log(x))^{-1/5}\Biggr), \]
where
\begin{equation}
    D = \frac{5}{2} \, \left(\frac{5}{3}\right)^{\frac{1}{5}} \left(\frac{2000}{161967}\right)^{\frac{3}{5}} \approx 0.1982767
    \label{eq:D1}.
\end{equation} 

\begin{center}
\textup{\textsc{Values for $L(x_0)$}}\\[2pt]
\setlength{\tabcolsep}{4pt}
\begin{tabular}{|c|c|l|c|c|l|c|c|}
\hline
$x_0$ & $L(x_0)$ & & $x_0$ & $L(x_0)$ & & $x_0$ & $L(x_0)$ \\
\hline
$\exp(3)$  & $0.2390$  & & $\exp(20)$    & $0.001013$              & & $\exp(20000)$  & $2.544 \times 10^{-24}$  \\
$\exp(4)$  & $0.2096$  & & $\exp(1500)$  & $0.0001630$             & & $\exp(30000)$  & $3.471 \times 10^{-30}$  \\
$\exp(5)$  & $0.1326$  & & $\exp(2000)$  & $0.00001100$            & & $\exp(40000)$  & $4.975 \times 10^{-35}$  \\
$\exp(6)$  & $0.08369$ & & $\exp(2500)$  & $1.012 \times 10^{-6}$  & & $\exp(100000)$ & $4.174 \times 10^{-55}$  \\
$\exp(7)$  & $0.05275$ & & $\exp(3000)$  & $1.166 \times 10^{-7}$  & & $\exp(500000)$ & $1.323 \times 10^{-114}$ \\
$\exp(8)$  & $0.03319$ & & $\exp(4000)$  & $2.517 \times 10^{-9}$  & & $\exp(10^6)$   & $1.416 \times 10^{-153}$ \\
$\exp(9)$  & $0.02085$ & & $\exp(5000)$  & $8.638 \times 10^{-11}$ & & $\exp(10^8)$   & $2.792 \times 10^{-157}$ \\
$\exp(10)$ & $0.01309$ & & $\exp(10000)$ & $1.805 \times 10^{-16}$ & & $\exp(10^9)$  &  $3.366 \times 10^{-558}$                  \\
\hline
\end{tabular}
\end{center}
and for $x_0 > \exp(10^9)$ we have $L(x_0)$ equal to the upper bound of $\mathcal{E}(x_0)$ given in Table \ref{tableBoundLarge}.
\end{thm}
The proof will be given in Section \ref{sec:prooftheorem1}.
\begin{Rem}
In comparison to \eqref{JYprevious}, we note that $-D = -0.1982767\dots$ is stronger than $-0.1853$ and we have an extra term $\frac{\log\log\log(x)}{15\log\log(x)}$ strengthening the result. We also do not have a logarithmic leading factor such as $\log(x)^{1.801}$ but rather a decreasing factor $L(x)$.

It is also fair to mention that the authors in \cite{JY-PNT} used an explicit Korobov-Vinogradov zero-free region with $c=1/57.54$ from \cite{FordZFR} to derive the bound \eqref{JYprevious} while we used stronger $c=1/53.989$ from \cite{BellottiZFR}. However, according to \eqref{MainConstantFormula}, even in the case of $c=1/57.54$ we would have obtained a constant around $-0.1908$, which would still have been better than $-0.1853$.

 We note that the first version of the preprint \cite{JZDET} and \cite[Eq.~(1.4)]{BellottiNewZFR} both express an exact formula
\[ \omega(x)=  \left(\frac{c^3 5^6}{2^2 3^4}\right)^{1/5} \frac{\log(x)^{3/5}}{\log\log(x)^{1/5}} \] in the case of the Korobov-Vinogradov  zero-free region when they should write
\[ \omega(x)\sim \left(\frac{c^3 5^6}{2^2 3^4}\right)^{1/5} \frac{\log(x)^{3/5}}{\log\log(x)^{1/5}}.
\]
This has since been corrected in the second version of \cite{JZDET}, which states the relation asymptotically. The error does not appear to impact the remaining analysis of either, as neither makes use of the claim. But we want to emphasize that $f(x) \sim g(x)$ does not imply $\exp(f(x))\sim \exp(g(x))$ as this difference is a core feature behind the shape of our theorem.
\end{Rem}
\begin{Rem}
For $x_0 \geqslant \exp(10^9)$ the bound in Theorem \ref{OurMainTheorem} is stronger than \eqref{eq:fksresult}.
\end{Rem}

In October 2026 OpenAI announced a proof \cite{OpenAIQRH} that $\zeta(s)$ has no zeros with $\Re(s)>7/8$, a quasi-Riemann hypothesis. Under such an assumption, and using the recent work of Chirre--Helfgott in \cite{ChirreHelfgott}, the same machinery gives bounds of a different shape:
\begin{thm}\label{OurQuasiTheorem}
Let $3/4\le\sigma\le 1$ and suppose that every zero $\rho=\beta+i\gamma$ of $\zeta$ with $\beta>0$ has $\beta\le\sigma$. Then for $x> x_0\geq e^3$
\[ E_{\psi}(x) < A(x_0)\log(x)^2\,x^{-1/4} + C(\sigma,x_0)\,x^{\sigma-1}, \]
where $A(x_0)$ and $C(\sigma,x_0)$ are given in the following table.
The values of $C(\sigma,x_0)$ provided depend on $\sigma$ only through which of the ranges $[3/4,13/16]$, $[13/16,7/8]$, $[7/8,1]$ contains it. 
\begin{center}
\begin{tabular}{|c|c|c|c|c|}
\hline
$x_0$ & $A(x_0)$ & $\sigma\in[3/4,13/16]$ & $\sigma\in[13/16,7/8]$ & $\sigma\in[7/8,1]$ \\
\hline
$\exp(3)$ & $0.04934$ & $1981.837$ & $7.396193$ & $0.02636425$ \\
$\exp(4)$ & $0.02162$ & $1981.837$ & $7.396193$ & $0.02636425$ \\
$\exp(5)$ & $0.0198944$ & $1981.837$ & $7.396193$ & $0.02636425$ \\
$\exp(150)$ & $0.0198944$ & $1759.13$ & $7.396193$ & $0.02636425$ \\
$\exp(200)$ & $0.0198944$ & $1255.613$ & $5.481296$ & $0.02506448$ \\
$\exp(300)$ & $0.0198944$ & $972.0567$ & $4.128316$ & $0.01840514$ \\
$\exp(500)$ & $0.0198944$ & $822.0193$ & $3.432807$ & $0.01505842$ \\
$\exp(1000)$ & $0.0198944$ & $736.3152$ & $3.042801$ & $0.01320991$ \\
$\exp(2000)$ & $0.0198944$ & $699.7371$ & $2.878095$ & $0.01243612$ \\
$\exp(5000)$ & $0.0198944$ & $679.4592$ & $2.787256$ & $0.01201123$ \\
$\exp(10^4)$ & $0.0198944$ & $672.9546$ & $2.75819$ & $0.01187557$ \\
$\exp(10^5)$ & $0.0198944$ & $667.2046$ & $2.732526$ & $0.0117559$ \\
\hline
\end{tabular}
\end{center}
\end{thm}
The proof is given in Section~\ref{sec:prooftheoremquasi}. In particular, the result of \cite{OpenAIQRH} allows $\sigma=7/8$ in Theorem~\ref{OurQuasiTheorem}, which then gives $E_\psi(x) < 0.04934\log(x)^2x^{-1/4} + 0.02636425\,x^{-1/8}$ for all $x>e^3$.

\subsection{Acknowledgments}

This project was initiated as part of the Inclusive Paths in Explicit Number Theory Summer School hosted at BIRS-UBCO and supported through the PIMS Collaborative Research Group on L-Functions in Analytic Number Theory. We are grateful to the organizers of that event, as well as PIMS, BIRS, and UBCO and the event's other sponsors.

We would also like to thank Nizar Bou Ezz for early discussions on the project  which helped direct us, including specifically his ideas for how to structure the argument of Proposition \ref{prop:fx}.

Finally, we would like to thank Sebastian Zuniga-Alterman for his input into the project.

\subsection{AI Use Declaration}

All of the ideas, theorems, and most of the proofs in this paper were developed by human authors.

AI systems, particularly Claude Opus 5, GPT-5.6 Sol and GPT-6 Astra were used to proofread the document, review the original code (Astra) and generate code to independently compute displayed numeric values (Opus 5) and to help in some derivations. AI assisted with deriving the bounds in Lemma \ref{lem:sigmabounds} and also the expression \eqref{H-inc} in Lemma \ref{hat-lemma} to help with the stage of the proof where that expression was used.

\section{Background and Notation}\label{sec:BN}

The non-trivial zeros of an $L$-function shall be denoted by $\rho=\beta+i\gamma$, and, when the ordinate is being used as a variable of integration, also by $\rho=\beta+it$. The non-triviality refers to the assumption that $0<\beta<1$. Zeros are always counted with multiplicity.

In order to consider sums over the zeros of an $L$-function we shall wish to have access to various explicit results concerning the number and location of these zeros. We describe each of these below.

\subsection{Zero-Free Regions}

A zero-free region refers to a function $v(w)$ of the logarithmic height $w=\log(t)$ such that any nontrivial zero $\rho=\beta+it$ with $|t|>e^{w_0}$ satisfies
\[  \beta < 1 - v(\log( |t|)). \]
We shall need to assume that the function $v(w)$ satisfies:
\begin{enumerate}
\item $v \in C^2([w_0, +\infty))$ and $0<v(w)<1/2$ for all $w>w_0$,
\item for all $w \ge w_0$, $v'(w)<0$ and $\lim\limits_{\substack{w\to \infty}} v'(w)=0$,
\item for all $w \ge w_0$, $v''(w)>0$, $v''$ is decreasing, tending to $0$, and
\[ \lim_{w\to \infty} v''(aw)\,\frac{w}{v'(w)}=c_a\neq 0 \]
for all $a\geqslant 1$ (note that necessarily $c_a<0$).
\end{enumerate}

There are several classical approaches to obtaining zero-free regions. These lead to zero-free regions of the following shapes, all of which satisfy the above provided $w_0$ is taken large enough:
\begin{itemize}
\item de la Vallée Poussin $v(w) = \frac{1}{Rw}$.
\item Littlewood $v(w) = \frac{\log(w)}{Rw}$.
\item Korobov-Vinogradov  $v(w) = \frac{1}{Rw^{2/3}\log(w)^{1/3}}$.
\end{itemize}

For the Riemann zeta function bounds of this form can be obtained as follows. For example, the article
\cite{MTYZFR}  gives $R=5.558691$ for de la Vallée Poussin following the argument of \cite{KadiriZFR} (with recent claimed improvements in \cite{BTY-ZFR}), the article
\cite{YangZFR} gives $R=21.233$ for the Littlewood form, and the article
\cite{BellottiZFR} gives $R=53.989$ for the Korobov-Vinogradov form building off \cite{MTYZFR} which follows the work of \cite{FordZFR}.

Work for many other types of $L$-functions has also been done. These typically provide results with similar asymptotic behaviours. For example for Dirichlet $L$-functions we have
\cite{McCurleyZFR-D},
\cite{HBZFR-D},
\cite{KadiriZFR-D}, and
\cite{KhaleZFR-D}.
For Dedekind zeta functions we have
\cite{KadiriZFR-D2} and
\cite{LeeZFR-D}.
For more general $L$-functions we have
\cite{LeungZFR} and
\cite{IPENT-ZFR-Hamieh}.

\subsection{Partial Riemann-Hypothesis Verification}

We shall denote by $H_0$ a value for which there is a partial verification of the (Generalized) Riemann-Hypothesis.
Such a verification asserts that all zeros with $| {\rm Im}(\rho) |\leqslant H_0$ satisfy ${\rm Re}(\rho)=1/2$.

The most recent partial verification of the Riemann-Hypothesis for zeta is due to Platt-Trudgian, \cite{PlattTrudgianZeros}, who verified this for $H_0= 3 000 175 332 800 > 3\cdot 10^{12}$.

In other contexts Platt has obtained partial verification of the Generalized Riemann Hypothesis for Dirichlet $L$-functions \cite{PlattDirichlet}.

We note that in the absence of such a verification in order to get explicit results on prime counting functions one may want to have explicit zero-free regions which are applicable for small $t$, see for example \cite[Chapter 3]{Kphd2002}.
 In this context one would also want zero-density bounds which are applicable for small $T$, see for example \cite[Theorem 22.1]{Kphd2002}.

Related to this one may wish to rule out exceptional zeros. Recent work of \cite{LZZ} rules these out for Dirichlet $L$-functions for $q<10^{10}$.

\subsection{Counts of Zeros}

Denote by $N(T)$ the number of zeros of $L$ (counted with multiplicity) with $ 0 < {\rm Im}(\rho) \leqslant T$\footnote{We note that there are different conventions for this quantity which differ by a factor of $2$.}.

Such counts are important when evaluating sums over zeros in regions where strong zero-density bounds are not available, that is, near the $1/2$ line.

Typical bounds for this are of the form
\[  |N(T) - AT\log(T) - BT - C|  = O( \log(T) ). \]
For the Riemann Zeta function bounds of this form can be obtained from \cite{HSW2021}, which has recently been updated by \cite{BWZerosZeta} to obtain
\[  |N(T) -  \frac{T}{2\pi}\log(T/2\pi e) |  \leqslant  0.10076 \log(T) + 0.24460\log\log T + 8.08344.  \]

In the context of  Dedekind zeta functions such results can be found in \cite{HSW2021b} whereas for Dirichlet $L$-functions they can be found in \cite{BMO2021}. Some very recent improvements to the Dedekind case have been announced in \cite{amberger2025estimating}, and \cite{BWArtin} gives an explicit zero-counting formula for Artin $L$-functions which, assuming Artin's holomorphy conjecture, improves on both \cite{amberger2025estimating} and \cite{BMO2021} for sufficiently large $T$; it is unconditional for Hecke $L$-functions over any number field.

We note that whether or not there is a $C$ term appears in a result is often a consequence of the method. For $\zeta$ the ``correct'' $C$ is $7/8$.

\subsection{Zero-Density Bounds}

Denote by $N(\sigma,T)$ the number of zeros of $L$ with ${\rm Re}(\rho) > \sigma$ and $ 0 < {\rm Im}(\rho) \leqslant T$.
We shall need to assume that there exist bounds of the form
\[ N(\sigma,T) \leqslant \widetilde{N}(\sigma,T)   \qquad \sigma > \sigma_0,\, T > T_0 \]
where 
\[ \widetilde{N}(\sigma,T) = \sum_{i}  c_i T^{p_i(\sigma)}\log(T)^{q_i(\sigma)}. \]
 We shall further wish to assume that the sum is finite with every $c_i>0$, and that $p_i$ and $q_i$ are continuous, non-increasing, with $0\leqslant q_i(\sigma)$, $0 \leqslant p_i(\sigma)<1$, and $p_i(1)=0$.
In practice, the value $T_0$ would often be the value $H_0$, see above.

For the Riemann Zeta function bounds of this form can be obtained from \cite{KLN}. These were updated in \cite{FKS1} to give:
\[  N(\sigma,T) < 17.4194 T^{\frac{8}{3}(1-\sigma)}(\log T)^{3+2(1-\sigma)}  + 5.2954 (\log T)^2  \]
uniformly for  $ \sigma \ge \frac{5}{8}$ and $T>3$.
Subsequent work of \cite[Corollary 4.50]{GF-Thesis} gives for $\sigma\in[0.75,1]$ and $T\geqslant 3$ the bound
\begin{equation}\label{eq:golnoush}
N(\sigma,T) \leqslant  12.45321 T ^{\frac{8}{3}(1-\sigma)}(\log T)^{3+2(1-\sigma)}+ 3.86894 (\log T)^2.
\end{equation}
More generally this type of result can be found in \cite{KNZDD} for Dedekind zeta functions.
For Dirichlet $L$-functions this is studied in \cite{RZDD}, but we caution that as noted in \cite{KLN} there are some errors which make the constants obtained in \cite{RZDD} potentially unreliable.

\section{Sums over Zeros}\label{sec:zeroso}

In this section we begin by giving a quick overview of how to use standard tools to find sums over the zeros before we provide the more significant contribution of tightly estimating sums near the one line in Section \ref{sec:zeros1}. For simplicity we assume $x>4$ so that for any non-trivial zero $\rho=\beta+i\gamma$ we can use the bound $\frac{x^{\beta-1}}{|\rho|^m} < \frac{x^{\beta-1}}{\gamma^m}$ and evaluate sums of the form $\sum \frac{x^{\beta-1}}{\gamma^m}$ over different regions.

\subsection{Zeros near the Half Line}

We first consider zeros near, or to the left of, the half-line. That is, we wish to consider the sum
\[   \sum_{\substack{0 <\gamma \leqslant T \\  \beta \leqslant  \sigma }} \frac{x^{\beta-1}}{\gamma^m}. \]
Note that if $m>1$ the value $T=\infty$ is admissible. 
 Throughout this subsection we take $1/2 \leqslant \sigma < 1$.

For $T>H_0$ it is very natural to split and bound it as
\[   x^{-1/2}\sum_{\substack{0 <\gamma \leqslant H_0  }} \frac{1}{\gamma^m} +   x^{\sigma-1}\sum_{\substack{H_0 <\gamma\leqslant T \\  \beta \leqslant  \sigma }} \frac{1}{\gamma^m}. \]
Following \cite[Proposition 3.6]{FKS1}, one can use instead (for $T_0\leqslant H_0$) the bound:
\[   x^{-1/2}\sum_{\substack{0 <\gamma \leqslant T_0 }} \frac{1}{\gamma^m}  +  x^{-1/2}\sum_{\substack{T_0 <\gamma \leqslant T }} \frac{1}{\gamma^m} +  x^{\sigma-1}\sum_{\substack{H_0 <\gamma \leqslant T \\ 1/2 < \beta \leqslant  \sigma }} \frac{1}{\gamma^m}. \]
An alternative slight refinement is provided by Proposition \ref{prop:lowlying} below.

Above, the first sum can be approximated using explicitly calculated lists of zeros, while the last sum, $\sum_{\substack{H_0 <\gamma\leqslant T \\ 1/2 < \beta \leqslant \sigma }} \frac{1}{\gamma^m}$, now runs over at most half of the zeros occurring in $\sum_{\substack{H_0 <\gamma \leqslant T \\  \beta \leqslant  \sigma }} \frac{1}{\gamma^m}$, provided the zeros off the critical line are symmetric about $\beta=1/2$ at each height (as is the case for $\zeta$, and for any $L$-function whose functional equation pairs $\rho$ with $1-\overline{\rho}$), since the zeros with $\beta\leqslant\sigma$ then consist of those on the critical line together with pairs $\beta, 1-\beta$ of which at most one lies in $1/2<\beta \leqslant \sigma$. With $\sigma>1/2$ the increase in size of the second term will be less significant than the decrease in the third.
\begin{Rem}
In most applications this optimization will not be significant.
\end{Rem}

The classical approach for such sums is to use a Riemann-Stieltjes integral
\[  \sum_{\substack{T_1 < \gamma \leqslant T_2 \\  \beta \leqslant  \sigma }} \frac{x^{\beta-1}}{\gamma^m} \leqslant x^{\sigma -1} \int_{T_1}^{T_2} \frac{1}{t^m} d N(t) \]
together with bounds for $N(T)$.
Using this method we have the following bound coming from \cite[Corollary 2.6]{FaberKadiri} and \cite[Lemma 7]{RosserShoenfeld}.
\begin{prop}\label{prop:K}
Suppose $m>0$, $2\pi < T_1 < T_2$ with $b_1,b_2\geqslant 0$, $b_3$ real, and assume that for all $T\geqslant T_1$ we have
\[ \left| N(T) - \left( \frac{T}{2\pi}\log \frac{T}{2\pi e} + \kappa \right) \right| \leqslant R_\kappa(T) := b_1\log T + b_2 \log\log T + b_3 \]
then
\[ \sum_{\substack{T_1 < \gamma \leqslant T_2  }} \frac{1}{\gamma^m}  \leqslant \left(\frac{1}{2\pi} + \frac{b_1 \log T_1 + b_2}{T_1 \log T_1 \log(T_1/2\pi)}\right) \left(\int_{T_1}^{T_2} \frac{1}{t^m}\log(t/2\pi) {\rm dt} \right) + \frac{2R_\kappa(T_1)}{T_1^m} \]
where we recall that for $m\neq 1$
\[ \int_{T_1}^{T_2} \frac{1}{t^m}\log(t/2\pi) {\rm dt} =  \left( \frac{1 + (m-1)\log(T_1/2\pi)}{(m-1)^2T_1^{m-1}} \right)-\left( \frac{1 + (m-1)\log(T_2/2\pi)}{(m-1)^2T_2^{m-1}} \right) \]
whereas for $m=1$
\[  \int_{T_1}^{T_2} \frac{1}{t^m}\log(t/2\pi) {\rm dt} = \log(T_2/T_1)\log(\sqrt{T_2T_1}/(2\pi)) \]
\end{prop}
The following sketch of proof also indicates which terms must be modified for other $L$-functions.
\begin{proof}[Sketch of Proof]
If $T_1$ or $T_2$ is the ordinate of zero the result follows from the general case by continuity, so simply assume they are not.
Then write
\[ \widetilde{N}(T) =  \frac{T}{2\pi}\log \frac{T}{2\pi e} +\kappa+ b_1\log T + b_2 \log\log T + b_3 \]
then we have
\begin{align*}
 \int_{T_1}^{T_2} \frac{1}{t^m} d N(t)
 &=   \frac{N(t)}{t^m}\Big|_{T_1}^{T_2} + \int_{T_1}^{T_2} \frac{m N(t)}{t^{m+1}}{\rm dt} \\
 &< \frac{2R_\kappa(T_1)}{T_1^m} + \frac{\widetilde{N}(t)}{t^m}\Big|_{T_1}^{T_2} + \int_{T_1}^{T_2} \frac{m \widetilde{N}(t)}{t^{m+1}}{\rm dt} \\
 &= \frac{2R_\kappa(T_1)}{T_1^m} + \int_{T_1}^{T_2} \frac{1}{t^m} d \widetilde{N}(t)
 \end{align*}
 One then simply takes a crude bound on
 \[ \int_{T_1}^{T_2}  \left( \frac{b_1}{t^{m+1}} + \frac{b_2}{t^{m+1}\log(t)}\right) {\rm dt} = \int_{T_1}^{T_2}  \left( \frac{b_1\log(t) + b_2}{t\log(t)\log(t/2\pi)} \right)  \frac{\log(t/2\pi) }{t^m}{\rm dt}  \]
 using that $\left( \frac{b_1\log(t) + b_2}{t\log(t)\log(t/2\pi)} \right) $ is decreasing for $t>2\pi$, since $b_1,b_2\geqslant 0$, to get the shape of the final answer.
\end{proof}
\begin{Rem}\label{rem:78}
Note that $\kappa$ in the above is an artifact of the bound on  $N(T)$ and its only role is in parameterizing the shape of what was being bounded.
\end{Rem}

A refinement of the above is possible by taking a less crude estimate of the error term; however, one can do better, as in \cite{BPT}. Part of the idea is that the mechanism by which bounds on quantities like $N(T)$ are obtained often involves bounding a quantity:
\[ S(T) = \pi^{-1} {\rm arg}(\zeta(1/2+iT)), \]
where the argument is obtained by continuous variation along the segments from $2$ to $2+iT$ and from $2+iT$ to $1/2+iT$ (with the usual averaging convention if $T$ is the ordinate of a zero).
Indeed, if we set $Q(T) = N(T) - \left(\frac{T}{2\pi}\log \frac{T}{2\pi e} +\frac{7}{8}\right)$ then for $T\ge 2\pi$ we have
\[ |Q(T)- S(T) | < \frac{A_2}{T} \]
where $A_2 = \frac{1}{150}$, see \cite[Lemma 2]{BPT}.

Consequently, if one defines $S_1(T) = \int_{0}^T S(t){\rm dt}$ and proves bounds of the form
\[ | S_1(T) -c_0 | < A_0 + A_1\log(T) \]
for $T>2\pi $, where $c_0 = S_1(T_*)$ for a suitable fixed $T_*$ (in what follows $T_*=168\pi$), then one can use integration by parts to refine the above estimates.
This is the strategy of \cite{BPT}, where they had $A_0=2.067$ and $A_1=0.059$, which ultimately yields:
\begin{prop}\label{prop:BPT}
 Suppose $m>0$ and $2\pi < T_1 < T_2$, and there are constants $A_0$, $A_1$, and $A_2$ such that for all $T\geqslant T_1$
\[  \left| S_1(T) - S_1(168\pi) \right| < A_0 + A_1 \log(T) \]
and
\[ \left| Q(T) - S(T) \right| < \frac{A_2}{T} \]
then we have the following estimate for sums over zeros of $\zeta$:
\begin{align*} \Biggl| \sum_{\substack{T_1 <\gamma \leqslant T_2  }} \frac{1}{\gamma^m}  - \frac{1}{2\pi}  \left(\int_{T_1}^{T_2} \frac{1}{t^m}\log(t/2\pi) {\rm dt} \right) -& 
\frac{Q(T_2)}{T_2^m} + \frac{Q(T_1)}{T_1^m}\Biggr| \\&<
\frac{2A_0m + 2A_1m\log(T_1) +A_1 + A_2}{T_1^{m+1}}  \end{align*}
\end{prop}
As above the case where $T_1$ or $T_2$ are ordinates of zeros follows from the general case.
\begin{Rem}
If one has an effective estimate for $N(T_1)$ then one can directly estimate $Q(T_1)$. For example, with $T_1=30\,610\,046\,000$ we know from \cite{PlattZeros} that $N(T_1)=103\,800\,788\,359$, and with $T_2=3\,000\,175\,332\,800$ we know from \cite{PlattTrudgianZeros} that $N(T_2) = 12\,363\,153\,437\,138$, so that
\begin{align*}
Q(T_1) &= 0.36876949943444761978\ldots, \\
Q(T_2) &= 0.90067212848167322498\ldots.
\end{align*}
From this we find by applying Proposition \ref{prop:BPT} that
\[ \sum_{T_1<\gamma \leqslant T_2} \frac{1}{\gamma} = 17.951092215437177382877787 \pm 7.53 \cdot 10^{-21}.\]

Otherwise, one can include one or both of the $\frac{Q(T_i)}{T_i^m}$ terms as an error term and bound it by $\frac{R_{7/8}(T)}{T^m}$ as in Proposition \ref{prop:K}.

For $m>1$ one may take $T_2$ to $\infty$ and we provide example calculations of the tails in Table \ref{Tab:Tails}. Note that one can obtain very precise bounds on the sum between the two  heights of that table, for $m\geqslant 2$, by taking differences.
\end{Rem}

\begin{table}
\caption{Precise values for reciprocal sums for zeros of $\zeta$ based on data from Platt \cite{PlattZeros}}\label{Tab:Platt}
\setlength{\tabcolsep}{2pt}
\begin{tabular}{|c|c|l|}
\hline
$X$ & $m$ & $\sum_{0<\gamma \leqslant X} \frac{1}{\gamma^m} $\\
    \hline
30 610 046 000 & 1 &  $39.57976478012384076139330428195283099 \pm 4.579 \cdot 10^{-33}$\\
30 610 046 000 & 2 &  $2.3104992994237211638265334892625777 \cdot 10^{-2}\pm 2.881 \cdot 10^{-34}$\\
30 610 046 000 & 3 &  $7.295482727097042139385394205871737\cdot 10^{-4} \pm 2.202 \cdot 10^{-35}$\\
30 610 046 000 & 4 &  $3.717259928526968616486626242986273 \cdot 10^{-5} \pm 1.763 \cdot 10^{-36}$\\
30 610 046 000 & 5 &  $2.231188699502103328640628691837191 \cdot 10^{-6} \pm 1.423 \cdot 10^{-37}$\\
30 610 046 000 & 6 &  $1.441739314009732796953815560948206 \cdot 10^{-7} \pm 1.148 \cdot 10^{-38}$\\
30 610 046 000 & 7 &  $9.675344542702350408719965627461515 \cdot 10^{-9} \pm 9.159 \cdot 10^{-40}$\\
30 610 046 000 & 8 &  $6.630316802529908698732720819613551 \cdot 10^{-10} \pm 7.261 \cdot 10^{-41}$\\
\hline
\end{tabular}
\end{table}

\begin{table}
\caption{Precise values for tails of sums for zeros of $\zeta$ from Proposition \ref{prop:BPT}} \label{Tab:Tails}
\begin{tabular}{|c|c|l|}
\hline
$X$ & $m$ & $\sum_{\gamma > X} \frac{1}{\gamma^m} $\\
    \hline
30 610 046 000 & 2 & $1.211817591506684755377 \cdot 10^{-10} \pm 5.4 \cdot 10^{-31}  $\\
30 610 046 000 & 3 & $1.936979148506797204374 \cdot 10^{-21} \pm 2.398\cdot  10^{-41} $\\
30 610 046 000 & 4 & $4.187784438181266359899 \cdot 10^{-32} \pm 1.046 \cdot 10^{-51} $\\
30 610 046 000 & 5 & $1.0223041003143534479121 \cdot 10^{-42} \pm 4.255\cdot  10^{-62} $\\
30 610 046 000 & 6 & $2.665890866401849763492 \cdot 10^{-53} \pm 1.669\cdot  10^{-72} $\\
30 610 046 000 & 7 & $7.246919788858296425405 \cdot 10^{-64} \pm 6.355 \cdot 10^{-83} $\\
30 610 046 000 & 8 & $2.027133372647964607815 \cdot 10^{-74} \pm 2.374 \cdot  10^{-93} $\\
3 000 175 332 800 & 2 & $1.479620256065536827148781172\cdot 10^{-12}\pm 5.600\cdot 10^{-37}$\\
3 000 175 332 800 & 3 & $2.42168510373753161368761323\cdot 10^{-25}\pm 2.796\cdot 10^{-49}$\\
3 000 175 332 800 & 4 & $5.3484657624351158869283997\cdot 10^{-38}\pm 1.241\cdot 10^{-61}$\\
3 000 175 332 800 & 5 & $1.33294576387575411409460504\cdot 10^{-50}\pm 5.172\cdot 10^{-74}$\\
3 000 175 332 800 & 6 & $3.5477666351663426850806775\cdot 10^{-63} \pm 2.070\cdot 10^{-86}$\\
3 000 175 332 800 & 7 & $9.8422067798255919866179924\cdot 10^{-76} \pm 8.035\cdot 10^{-99}$\\
3 000 175 332 800 & 8 & $2.8094204750410392718638777\cdot 10^{-88}\pm 3.064\cdot 10^{-111}$\\
\hline
\end{tabular}
\end{table}

In summary, for the Riemann-zeta function we have:
\begin{prop}\label{prop:lowlying}
Suppose $T> H_0 = 3 000 175 332 800$, $x>1$ and $1/2\leqslant \sigma<1$, then we have the following bounds on the sums over zeros of $\zeta$:
\[ \sum_{0<\gamma \leqslant H_0} \frac{1}{\gamma} =
57.530856995561018144271091 \pm 7.53\cdot 10^{-21}
\]
and
\begin{align*} \sum_{H_0<\gamma\leqslant T} \frac{1}{\gamma} =
\frac{1}{2\pi}\log\left(\frac{T}{H_0}\right)&\log\left(\frac{\sqrt{TH_0}}{2\pi}\right) - \frac{Q(H_0)}{H_0}
\\&\pm \left(\frac{ R_{7/8}(T) }{T}  +  \frac{2A_0 + 2A_1\log(H_0) + A_1+A_2}{H_0^2} \right)
\end{align*}
where all constants satisfy conditions as in Proposition \ref{prop:BPT} and \ref{prop:K}.
Consequently we have
\begin{align*}
\sum_{\substack{0 <\gamma \leqslant T \\  \beta \leqslant  \sigma }} \frac{x^{\beta-1}}{\gamma}
<&  57.5308569955611 x^{-1/2}  \\
&+\left( \frac{x^{\sigma-1} + x^{-\sigma}}{2} \right)\Biggl(\frac{1}{2\pi}\log\left(\frac{T}{H_0}\right)\log\left(\frac{\sqrt{TH_0}}{2\pi}\right)  - \frac{Q(H_0)}{H_0} \\&\qquad\qquad\qquad+ \frac{ R_{7/8}(T) }{T}  +  \frac{2A_0 + 2A_1\log(H_0) + A_1+A_2}{H_0^2} \Biggr).
\end{align*}
\end{prop}
\begin{proof}
    The proof is relatively straightforward by combining the $m=1$ row of Table \ref{Tab:Platt} together with the remark following Proposition \ref{prop:BPT}.
    We briefly explain where the term $\frac{x^{\sigma-1} + x^{-\sigma}}{2}$ comes from. Broadly speaking it arises from pairing each zero with the corresponding one under the functional equation. Several cases need to be considered: zeros to the left of $1-\sigma$, zeros on the half line, and the remaining zeros. In the latter two cases that $\frac{x^{\sigma-1} + x^{-\sigma}}{2}$ is increasing in $\sigma$ when $\sigma>1/2$ is used. In the  first  case one uses that the combined contribution from the two zeros is an overestimate.
\end{proof}

\subsection{Intermediate Zeros}

Here we wish to consider the sum
\[   \sum_{\substack{0 <\gamma \leqslant T \\  \sigma_0  < \beta \leqslant  \sigma_1 }} \frac{x^{\beta-1}}{\gamma^m}, \]
where $1/2 < \sigma_0 < \sigma_1 < 1$ and where $T=\infty$ is an admissible value.

\begin{prop}
Let $m>0$, $T_2>T_1 >1$ (with $T_2=\infty$ admissible, in which case the terms involving $T_2$ are $0$), $\sigma_0>1/2$, $c_i> 0$  and $p_i<m$, and suppose that $N(\sigma,T) \leqslant \sum_i c_i T^{p_i}\log(T)^{q_i}$ for $\sigma\geqslant \sigma_0$ and $T_1\leqslant T  \leqslant T_2$. Then for $\sigma\geqslant\sigma_0$ we have the following bound on the sum over zeros of $\zeta$:
\[ S_{m,T_1,T_2}(\sigma) := \sum_{\substack{T_1 <\gamma \leqslant T_2 \\  \sigma  < \beta  }} \frac{1}{\gamma^m}  \leqslant  B_0(m,\sigma, T_1, T_2),\]
where $B_0(m,\sigma, T_1, T_2)$ is given by:
\begin{equation}\label{def-B4}
\begin{split}
\sum_{i} c_i  \biggl( &\frac{ (\log T_2 )^{q_i }}{T_2^{m-p_i }} \\
&+ \frac{m}{(m-p_i)^{q_i+1} } \Bigl( \Gamma(q_i+1,(m-p_i)\log T_1) \\
&\qquad\qquad - \Gamma(q_i+1,(m-p_i)\log T_2)\Bigr) \biggr).
\end{split}
\end{equation}
\end{prop}
The quantities $c_i$, $p_i$, and $q_i$ are likely to depend on $\sigma$.

\begin{Rem}\label{Rem:Gammabounds}
We recall that the incomplete Gamma function is given by 
\begin{equation}\label{def-Gamma}
\Gamma(s,x) = \int_x^{\infty} t^{s-1} e^{-t} dt, \ \text{ for }\ \operatorname{Re}(s)>0 \text{ (and for all $s$ when $x>0$)}.
\end{equation}
and by integration by parts $\Gamma(s,x) = x^{s-1}e^{-x} + (s-1)\Gamma(s-1,x)$ so that:
\begin{equation}\label{prop-Gamma}
\begin{aligned}
\Gamma(1+a,x) &= x^ae^{-x} + a\Gamma(a,x),\\
\Gamma(2+a,x) &= \Bigl(x^{1+a} + (1+a)x^a\Bigr) e^{-x} + (1+a)a\Gamma(a,x) , \\
\Gamma(3+a,x) &= \Bigl(x^{2+a}+(2+a)x^{1+a} + (2+a)(1+a)x^a\Bigr)e^{-x}\\&\qquad\qquad + (2+a)(1+a)a\Gamma(a,x) , \\
\Gamma(4+a,x) &= \Bigl(x^{3+a}+(3+a)x^{2+a} + (3+a)(2+a)x^{1+a} \\
&\qquad\qquad + (3+a)(2+a)(1+a)x^{a}\Bigr)e^{-x} \\&\quad+ (3+a)(2+a)(1+a)a\Gamma(a,x) , \\
\Gamma(a,x) &\leqslant x^{a-1}e^{-x}  \text{ for $x>0$ and $a \leqslant 1 $ (with equality when $a=1$)}, \text{ and }\\
\Gamma(s,x) &\sim x^{s-1} e^{-x} \ \text{as}\ x\to\infty, \ \text{for all }\ s.\\
\end{aligned}
\end{equation}
leading to upper bounds for any integer $n\geqslant 0$, $x>0$ and $0\leqslant a\leqslant 1$
\[ \Gamma(n+a,x) \leqslant \frac{P_n(a,x)}{x^{1-a}}e^{-x} \]
where $P_n(a,x) = \sum_{k=0}^{n} \Bigl(\prod_{j=k}^{n-1}(j+a)\Bigr) x^{k}$ is a degree $n$ polynomial in $x$ whose coefficients are  the rising factorials (Pochhammer symbols) $(k+a)_{n-k}$ (the restriction $a\geqslant0$ ensures these coefficients, and hence the coefficient $(n-1+a)\cdots a$ of $\Gamma(a,x)$, are non-negative; for $a<0$ the inequality can fail).
\end{Rem}

The above leads to the following:
\begin{lem}\label{lem:Qbound}
    Suppose that $m>0$, $T_2>T_1>1$, $q(\sigma) = \widehat{q} + \widetilde{q}(1-\sigma)$ where $\widehat{q}$ is a non-negative integer, $\widetilde{q}\geqslant 0$ and $0\leqslant \widetilde{q}(1-\sigma)\leqslant 1$, and $p(\sigma) = \widetilde{p}(1-\sigma)$ with $\widetilde{p}\ge 0$ and $m-p(\sigma)>0$, then
    \[
        \frac{m\Gamma(q(\sigma)+1,(m-p(\sigma))\log T_1)}{(m-p(\sigma))^{q(\sigma)+1} }  \leqslant
        \frac{1}{T_1^m\log(T_1)} \frac{ Q(T_1,1-\sigma) }{ \left(\log(T_1)^{\widetilde{q}} T_1^{\widetilde{p}}\right)^{\sigma-1}}
    \]
    and
      \begin{align*}
       & \frac{m(\Gamma(q(\sigma)+1,(m-p(\sigma))\log T_1)-\Gamma(q(\sigma)+1,(m-p(\sigma))\log T_2))}{(m-p(\sigma))^{q(\sigma)+1} }\\ & \qquad\qquad\qquad\leqslant
        \frac{1}{T_1^m\log(T_1)} \frac{ Q(T_1,1-\sigma) }{ \left(\log(T_1)^{\widetilde{q}} T_1^{\widetilde{p}}\right)^{\sigma-1}}
        -\frac{1}{T_2^m\log(T_2)} \frac{ Q(T_2,1-\sigma) }{ \left(\log(T_2)^{\widetilde{q}} T_2^{\widetilde{p}}\right)^{\sigma-1}}
    \end{align*}
    where $ Q(T,a) = \frac{mP_{\widehat{q}+1}(\widetilde{q}a,\log(T)(m-\widetilde{p}a))}{
 (m-\widetilde{p}a)^{\widehat{q}+2}}$. Moreover, $Q(T,a)$ and its derivatives of every order with respect to $a$ are non-negative provided $T>1$, $m-\widetilde{p}a>0$, $a\geqslant 0$ and $\widetilde{q}\ge0$.
\end{lem}
For the final inequality above one should repeat the calculations of Remark \ref{Rem:Gammabounds} on the difference.

\begin{prop}
Let $x>1$ and suppose we have
$ S_{m,T_1,T_2}(\sigma) \leqslant \widetilde{S}_{m,T_1,T_2}(\sigma)$ for $\sigma_0\leqslant\sigma\leqslant\sigma_1$, where $\widetilde{S}_{m,T_1,T_2}$ is of bounded variation on $[\sigma_0,\sigma_1]$. Then we have the following equivalent bounds
\[
 \sum_{\substack{T_1 <\gamma \leqslant T_2 \\  \sigma_0  < \beta \leqslant \sigma_1  }} \frac{x^{\beta-1}}{\gamma^m}  \leqslant  x^{\sigma_1-1} \left( \widetilde{S}_{m,T_1,T_2}(\sigma_1) - \int_{\sigma_0}^{\sigma_1} x^{\sigma-\sigma_1}d\widetilde{S}_{m,T_1,T_2}(\sigma) \right)
 \]
 and
 \[
  \sum_{\substack{T_1 <\gamma \leqslant T_2 \\  \sigma_0  < \beta \leqslant \sigma_1  }} \frac{x^{\beta-1}}{\gamma^m}  \leqslant  \widetilde{S}_{m,T_1,T_2}(\sigma_0)x^{\sigma_0-1}   + \int_{\sigma_0}^{\sigma_1} \widetilde{S}_{m,T_1,T_2}(\sigma)\log(x)x^{\sigma-1}d\sigma.
 \]
 \end{prop}
\begin{proof}
First, we notice that for any $\sigma'<\sigma''$ that
\[ \sum_{\substack{T_1 <\gamma \leqslant T_2 \\  \sigma'  < \beta  \leqslant  \sigma'' }} \frac{x^{\beta-1}}{\gamma^m} \leqslant   x^{\sigma''-1} (S_{m,T_1,T_2}(\sigma') - S_{m,T_1,T_2}(\sigma'')).  \]
Consequently, by taking a finer and finer subdivision of  $[\sigma_0,\sigma_1]$ we obtain through Riemann-Stieltjes integration that:
\[  \sum_{\substack{T_1 <\gamma\leqslant T_2 \\  \sigma_0  < \beta  \leqslant  \sigma_1 }} \frac{x^{\beta-1}}{\gamma^m} \leqslant -\int_{\sigma_0}^{\sigma_1} x^{\sigma-1} d S_{m,T_1,T_2}(\sigma) \]
Integration by parts gives:
\begin{align*} &-\int_{\sigma_0}^{\sigma_1} x^{\sigma-1} d S_{m,T_1,T_2}(\sigma) \\&\qquad= -x^{\sigma_1-1} S_{m,T_1,T_2}(\sigma_1)  + x^{\sigma_0-1}S_{m,T_1,T_2}(\sigma_0) + \int_{\sigma_0}^{\sigma_1} S_{m,T_1,T_2}(\sigma) dx^{\sigma-1} \end{align*}
As $x^{\sigma-1}$ is increasing we see that we may replace various terms by their upper bounds to obtain
\[  \sum_{\substack{T_1 <\gamma\leqslant T_2 \\  \sigma_0  < \beta  \leqslant  \sigma_1 }}  \frac{x^{\beta-1}}{\gamma^m}\leqslant  x^{\sigma_0-1}\widetilde{S}_{m,T_1,T_2}(\sigma_0) + \int_{\sigma_0}^{\sigma_1} \widetilde{S}_{m,T_1,T_2}(\sigma) dx^{\sigma-1}. \]
  Reversing the integration by parts now yields
 \[
  \sum_{\substack{T_1 <\gamma\leqslant T_2 \\  \sigma_0  < \beta  \leqslant  \sigma_1 }}  \frac{x^{\beta-1}}{\gamma^m}\leqslant  x^{\sigma_1-1}\widetilde{S}_{m,T_1,T_2}(\sigma_1) - \int_{\sigma_0}^{\sigma_1} x^{\sigma-1}d\widetilde{S}_{m,T_1,T_2}(\sigma) . \qedhere\]
\end{proof}

\begin{Rem}
We can see that, provided $\widetilde{S}_{m,T_1,T_2}(\sigma)$ decreases with $\sigma$, we have
\begin{align*} - \int_{\sigma_0}^{\sigma_1} x^{\sigma-\sigma_1}d\widetilde{S}_{m,T_1,T_2}(\sigma)
&= -\widetilde{S}_{m,T_1,T_2}(\sigma_1) + x^{\sigma_0-\sigma_1}\widetilde{S}_{m,T_1,T_2}(\sigma_0) \\
&\qquad + \int_{\sigma_0}^{\sigma_1} \widetilde{S}_{m,T_1,T_2}(\sigma)\log(x)x^{\sigma-\sigma_1} d\sigma \\
&\leqslant -\widetilde{S}_{m,T_1,T_2}(\sigma_1) + x^{\sigma_0-\sigma_1}\widetilde{S}_{m,T_1,T_2}(\sigma_0) \\
&\qquad + \widetilde{S}_{m,T_1,T_2}(\sigma_0) \int_{\sigma_0}^{\sigma_1} \log(x)x^{\sigma-\sigma_1} d\sigma \\
&=\widetilde{S}_{m,T_1,T_2}(\sigma_0)  -\widetilde{S}_{m,T_1,T_2}(\sigma_1) 
\end{align*}
which recovers the trivial bound
\[ 
 \sum_{\substack{T_1 <\gamma\leqslant T_2 \\  \sigma_0  < \beta  \leqslant  \sigma_1 }} \frac{1}{\gamma^m} x^{\beta-1} \leqslant \widetilde{S}_{m,T_1,T_2}(\sigma_0)  x^{\sigma_1-1} .
 \]
 \end{Rem}
For any fixed $x$ numerical integration could improve the trivial bound; the following proposition shows how to improve it abstractly.

\begin{prop}\label{prop:tightintermediate}
Let $T_1 > e$, $1/2 < \sigma_0 <\sigma_1\leqslant 1$ and suppose
\[ N(\sigma,T) \leqslant \sum_i c_i T^{\widetilde{p}_i(1-\sigma)}\log(T)^{\widehat{q}_i + \widetilde{q}_i(1-\sigma)} \]
for $\sigma_1 \geqslant \sigma \geqslant \sigma_0$ and $T>T_1$ and with $m>0$, $c_i\geqslant 0$, $\widehat{q}_i$ a non-negative integer, $\widetilde{p}_i\geqslant 0$, $\widetilde{q}_i\ge 0$, $\widetilde{q}_i(1-\sigma_0)\leqslant 1$ and $m-\widetilde{p}_i(1-\sigma_0)>0$. Let $r\geqslant 0$ be an integer.
Then for $x\geqslant x_0 > 1$ with $\log(x_0)> \widetilde{p}_i\log(T_1) + \widetilde{q}_i\log\log(T_1)$ for all $i$,
\begin{align*}
\sum_{\substack{T_1 <\gamma \\  \sigma_0  < \beta \leqslant \sigma_1  }} \frac{x^{\beta-1}}{\gamma^m}
&\leqslant B_0(m,\sigma_0,T_1,\infty)\,x^{\sigma_0-1} \\
&\quad + \sum_i \frac{c_i}{T_1^m\log(T_1)} \Biggl(\overline{B}(m,\sigma_0,\sigma_1,T_1,\widetilde{p}_i,\widehat{q}_i,\widetilde{q}_i,r,x_0)\,x^{\sigma_1-1} \\
&\qquad\qquad\qquad\qquad - \underline{B}(m,\sigma_0,T_1,\widetilde{p}_i,\widehat{q}_i,\widetilde{q}_i,r,x)\,x^{\sigma_0-1}\Biggr)
\end{align*}
where $B_0$ is as in \eqref{def-B4} (with $p_i=\widetilde{p}_i(1-\sigma_0)$, $q_i=\widehat{q}_i+\widetilde{q}_i(1-\sigma_0)$ and $T_2=\infty$),
\begin{align*}
\overline{B}(m,\sigma_0,\sigma_1,T,\widetilde{p},\widehat{q},\widetilde{q},r,x) &=
(T^{\widetilde{p}}\log(T)^{\widetilde{q}})^{(1-\sigma_1)} \\
&\quad\times\biggl(\sum_{j=0}^{r-1} \frac{Q_{m,T,\widehat{q},\widetilde{q},\widetilde{p}}^{(j)}(1-\sigma_1)\log(x)}
{\left(\log(x) - \widetilde{p}\log(T) - \widetilde{q}\log\log(T)\right)^{j+1}}
\\& \qquad\qquad + \frac{Q_{m,T,\widehat{q},\widetilde{q},\widetilde{p}}^{(r)}(1-\sigma_0)\log(x)}
{\left(\log(x) - \widetilde{p}\log(T) - \widetilde{q}\log\log(T)\right)^{r+1}}\biggr)
\end{align*}
and
\begin{align*}
\underline{B}(m,\sigma_0,T,\widetilde{p},\widehat{q},\widetilde{q},r,x) &=
(T^{\widetilde{p}}\log(T)^{\widetilde{q}})^{(1-\sigma_0)} \\
&\quad\times\biggl(\sum_{j=0}^{r} \frac{Q_{m,T,\widehat{q},\widetilde{q},\widetilde{p}}^{(j)}(1-\sigma_0)\log(x)}
{\left(\log(x) - \widetilde{p}\log(T) - \widetilde{q}\log\log(T)\right)^{j+1}}\biggr) \ \geqslant 0
\end{align*}
with $Q^{(j)}$ denoting the $j$th derivative with respect to $a$ of
 \[ Q_{m,T,\widehat{q},\widetilde{q},\widetilde{p}}(a) = \frac{mP_{\widehat{q}+1}(\widetilde{q}a,\log(T)(m-\widetilde{p}a))}{
 (m-\widetilde{p}a)^{\widehat{q}+2}} .
 \]
\end{prop}
\begin{Rem}
    For $T_1 \ge e$ the function $\overline{B}(m,\sigma_0,\sigma_1,T_1,\widetilde{p}_i,\widehat{q}_i,\widetilde{q}_i,r,x_0)$ is non-increasing in $\sigma_0$, $\sigma_1$, and $x_0$, hence the value at some point can be used uniformly for larger values.
In contrast the function $\underline{B}(m,\sigma_0,T,\widetilde{p},\widehat{q},\widetilde{q},r,x)$ is positive and non-increasing in $x$, so the coefficient of $x^{\sigma_0-1}$ above is non-decreasing in $x$ and cannot be replaced by its value at $x_0$. For a bound uniform in $x\geqslant x_0$ one may replace $\underline{B}(\ldots,x)$ by its limit $(T^{\widetilde{p}}\log(T)^{\widetilde{q}})^{(1-\sigma_0)}Q_{m,T,\widehat{q},\widetilde{q},\widetilde{p}}(1-\sigma_0)$ as $x\to\infty$ (only the $j=0$ term survives); by Lemma~\ref{lem:Qbound} the resulting coefficient,
\[  B_0(m,\sigma_0,T_1,\infty) - \sum_i \frac{c_i}{T_1^m\log(T_1)}(T_1^{\widetilde{p}_i}\log(T_1)^{\widetilde{q}_i})^{(1-\sigma_0)}Q_{m,T_1,\widehat{q}_i,\widetilde{q}_i,\widetilde{p}_i}(1-\sigma_0), \]
 is non-positive, so the $x^{\sigma_0-1}$ term may also simply be discarded.
\end{Rem}
\begin{Rem}
  The condition $\widetilde{q}_i\ge 0$ can be relaxed at the expense of replacing the value of $Q^{(r)}(1-\sigma_0)$ in both formulas with its maximum on the interval, which  will no longer obviously be at an endpoint.
\end{Rem}
\begin{proof}
From 
\begin{align*}
  \sum_{\substack{T_1 <\gamma\\  \sigma_0  < \beta \leqslant \sigma_1  }} \frac{x^{\beta-1}}{\gamma^m}
  &\leqslant  \widetilde{S}_{m,T_1,\infty}(\sigma_0)x^{\sigma_0-1} \\
  &\qquad + \int_{\sigma_0}^{\sigma_1} \widetilde{S}_{m,T_1,\infty}(\sigma)\log(x)x^{\sigma-1}d\sigma.
\end{align*}
 and writing
 \begin{align*}
     \widetilde{S}_{m,T_1,\infty}(\sigma) &=
 \sum_i c_i\frac{m\Gamma(\widehat{q}_i+\widetilde{q}_i(1-\sigma)+1,(m-\widetilde{p}_i(1-\sigma))\log(T_1))}{(m-\widetilde{p}_i(1-\sigma))^{\widehat{q}_i +1 + \widetilde{q}_i(1-\sigma)}} \\&
 < \sum_i c_im\frac{P_{\widehat{q}_i+1}(\widetilde{q}_i(1-\sigma),\log(T_1)(m-\widetilde{p}_i(1-\sigma)))}{
 \log(T_1)^{1-\widetilde{q}_i(1-\sigma)}(m-\widetilde{p}_i(1-\sigma))^{\widehat{q}_i+2}T_1^{m-\widetilde{p}_i(1-\sigma)}
 }\\
 &= \frac{1}{T_1^m\log(T_1)} \sum_i c_iQ_{m,T_1,\widehat{q}_i,\widetilde{q}_i,\widetilde{p}_i}(1-\sigma) \left(\log(T_1)^{\widetilde{q}_i} T_1^{\widetilde{p}_i}\right)^{1-\sigma}
  \end{align*} 
  we thus have the bound
 \begin{align*}
   \sum_{\substack{T_1 <\gamma\\  \sigma_0  < \beta \leqslant \sigma_1  }} \frac{x^{\beta-1}}{\gamma^m}
   &\leqslant  \widetilde{S}_{m,T_1,\infty}(\sigma_0)x^{\sigma_0-1} \\
   &\qquad + \frac{\log(x)}{T_1^m\log(T_1)}\sum_i c_i\int_{\sigma_0}^{\sigma_1} Q_{m,T_1,\widehat{q}_i,\widetilde{q}_i,\widetilde{p}_i}(1-\sigma) \\
   &\qquad\qquad\qquad\qquad\times\left(\frac{x}{T_1^{\widetilde{p}_i}\log(T_1)^{\widetilde{q}_i}} \right)^{\sigma-1} d\sigma.
 \end{align*}
 Repeated integration by parts essentially gives the result, except that at the final stage one uses that the higher derivatives of $Q(a)$ are all positive, so that we may uniformly bound the relevant quantity on the interval.
\end{proof}

\begin{Rem}
We aim to compare the quality of the trivial bound and the strategy above.

To be concrete we can consider $m=1$,  $T_1=H_0=3\,000\,175\,332\,800$, $T_2=\infty$, $\sigma_0=0.9$, $\sigma_1\ge 0.9$, $q(\sigma)=3+2(1-\sigma)$, and $p(\sigma)=\frac{8}{3}(1-\sigma)$. For simplicity we take $c(\sigma)=1$, noting this value affects every result linearly. In this case we would have
\[ \sum_{\substack{T_1 <\gamma\leqslant T_2 \\  \sigma_0  < \beta  \leqslant  \sigma_1 }} \frac{1}{\gamma^m} x^{\beta-1} \leqslant 5.238\cdot10^{-5}x^{\sigma_1-1}.\]
Taking $x=e^{100}$ and $\sigma_1=0.99$ gives a bound of $1.927\cdot 10^{-5}$ whereas $x=e^{1000}$ gives $2.378\cdot10^{-9}$.
 
Performing integration to obtain upper bounds for  $-\int_{\sigma_0}^{\sigma_1} x^{\sigma-\sigma_1}d\widetilde{S}_{m,T_1,T_2}(\sigma)$ allows for improvements. Using the same example
\[  \sum_{\substack{T_1 <\gamma \leqslant T_2 \\  \sigma_0  < \beta  \leqslant  \sigma_1 }}  \frac{x^{\beta-1}}{\gamma^m}\leqslant  x^{\sigma_0-1}\widetilde{S}_{m,T_1,T_2}(\sigma_0) + \int_{\sigma_0}^{\sigma_1} \widetilde{S}_{m,T_1,T_2}(\sigma) dx^{\sigma-1} \]
where
\[ \widetilde{S}_{m,T_1,T_2}(\sigma)
= \frac{1}{(1-p(\sigma))^{q(\sigma)+1}}\Gamma(q(\sigma)+1,(1-p(\sigma))\log(H_0)). \]
Now again with $x=e^{100}$ and $\sigma_1=0.99$ this evaluates numerically to approximately $4.271\cdot 10^{-8}$, and at $x=e^{1000}$ we obtain $1.037\cdot10^{-12}$.

However, to obtain asymptotic results where $x$ varies we would use Proposition \ref{prop:tightintermediate} with the same examples for
$x=e^{100}$ with $r=0$ we obtain $5.207\cdot 10^{-8}$, for $r=1$ we obtain
$4.701\cdot 10^{-8}$, and $r=2$ we obtain  $4.568\cdot 10^{-8}$. Whereas for $x=e^{1000}$ with $r=0$ we obtain $1.436\cdot 10^{-12}$, for $r=1$ we obtain
$1.040\cdot 10^{-12}$, and $r=2$ we obtain $1.037\cdot 10^{-12}$.
Note that the value does not necessarily improve with $r$, as the higher derivatives of $Q$ eventually grow quickly (at $x=e^{100}$ the bound is minimal for $r=3$ and worsens from $r=4$ on).
\end{Rem}

We now provide an easy-to-use version of Proposition \ref{prop:tightintermediate} by evaluating with specific parameters and using that $H_0$ is the Riemann Hypothesis verification height.
\begin{cor}\label{cor:intermediate}
Suppose that we have
\[ N(\sigma,T) \leqslant \mathcal{U}T^{\frac{8}{3}(1-\sigma)}\log(T)^{3+2(1-\sigma)} + \mathcal{V}\log(T)^2 \]
for $\sigma\in [\sigma_0,\sigma_1]$ and $T\geqslant H_0$.
Suppose $0.99 \le \sigma_1 \le 1$. For $\sigma_0=0.9$ and $\log(x)>100$,
\[  \sum_{  \sigma_0  < \beta  \leqslant  \sigma_1 } \frac{1}{\gamma} x^{\beta-1} \leqslant
\left(
1.701359
\cdot 10^{-7} \mathcal{U}
+ 2.949343\cdot 10^{-10}\,\mathcal{V}
\right)x^{\sigma_1-1},
\]
while for $\sigma_0=0.75$ and $\log(x)>100$ the same bound holds with $1.701359\cdot 10^{-7}$ replaced by $4.594096\cdot 10^{-7}$.
We can reduce $\sigma_0$ without much harm to the bounds once $x$ is larger: for $\sigma_0=0.75$ and $\log(x)>200$,
\[ \sum_{  \sigma_0  < \beta  \leqslant  \sigma_1 } \frac{1}{\gamma} x^{\beta-1} \leqslant
\left(
3.677198\cdot 10^{-8}  \mathcal{U} +
2.949343\cdot 10^{-10}\,\mathcal{V}
\right) x^{\sigma_1-1},
\]
and if instead $\sigma_0=0.75$ and $\log(x)>1000$ then
\[ \sum_{  \sigma_0 < \beta  \leqslant  \sigma_1 } \frac{1}{\gamma} x^{\beta-1} \leqslant
\left(
2.282680 \cdot 10^{-8}  \mathcal{U} +
2.949343\cdot 10^{-10}\,\mathcal{V}
\right)x^{\sigma_1-1}.
\]
The coefficients are those of Proposition~\ref{prop:tightintermediate} with $r=1$ (for $\log(x)>100$), $r=12$ (for $\sigma_0=0.75$, $\log(x)>200$) and $r=16$ otherwise, rounded up in the last digit. The non-positive $x^{\sigma_0-1}$ term of the proposition has been discarded (for $\sigma_0=0.9$ the coefficient would be $-1.348797\cdot10^{-11}\,\mathcal{U}$).
\end{cor}

Corollary~\ref{cor:intermediate} requires $\sigma_1\ge 0.99$, which is what Theorem~\ref{thm:assymp} needs. Under a hypothesis of quasi-Riemann type, smaller values of $\sigma_1$ may be of interest; these are provided by:
\begin{cor}\label{cor:intermediatelow}
Let $3/4 \le \sigma_1 \le 1$ and $L_0\geq 84$, and suppose that
\[ N(\sigma,T) \leq \mathcal{U}T^{\frac{8}{3}(1-\sigma)}\log(T)^{3+2(1-\sigma)} + \mathcal{V}\log(T)^2 \]
for $\sigma\in [3/4,\sigma_1]$ and $T\geq H_0$. Then for $\log(x)>L_0$
\[  \sum_{  3/4  < \beta  \leq  \sigma_1 } \frac{1}{\gamma} x^{\beta-1} \leq
\left( C_{\mathcal{U}}(\sigma_1,L_0)\, \mathcal{U} + 2.949343\cdot 10^{-10}\,\mathcal{V} \right)x^{\sigma_1-1},\]
where
\begin{align}\label{eq:cUL0}
C_{\mathcal{U}}(\sigma_1,L_0) = \frac{\left(H_0^{8/3}\log(H_0)^2\right)^{1-\sigma_1}}{H_0\log(H_0)}\biggl(& \frac{Q_{1,H_0,3,2,8/3}(1-\sigma_1)\,L_0}{L_0-\frac{8}{3}\log(H_0)-2\log\log(H_0)} \\&+ \frac{Q_{1,H_0,3,2,8/3}'(1/4)\,L_0}{\left(L_0-\frac{8}{3}\log(H_0)-2\log\log(H_0)\right)^2}\biggr).\nonumber
\end{align}
$C_{\mathcal{U}}(\sigma_1,L_0)$ is non-increasing in $\sigma_1$ and in $L_0$; upper bounds at $\sigma_1=3/4,13/16,7/8$, which are therefore valid for all larger $\sigma_1\le 1$ and $L_0$, are:
\begin{center}
\begin{tabular}{|r|c|c|c|}
\hline
$L_0$ & $\sigma_1\ge3/4$ & $\sigma_1\ge13/16$ & $\sigma_1\ge7/8$ \\
\hline
$100$ & $396.1743$ & $1.625773$ & $7.679923\cdot10^{-3}$ \\
$120$ & $167.9272$ & $0.623139$ & $2.746144\cdot10^{-3}$ \\
$140$ & $116.0693$ & $0.4118469$ & $1.752566\cdot10^{-3}$ \\
$150$ & $103.0261$ & $0.3604065$ & $1.515809\cdot10^{-3}$ \\
$157$ & $96.22037$ & $0.3338988$ & $1.394842\cdot10^{-3}$ \\
$191$ & $76.67014$ & $0.259201$ & $1.058532\cdot10^{-3}$ \\
$200$ & $73.53686$ & $0.2474514$ & $1.006346\cdot10^{-3}$ \\
$300$ & $56.92998$ & $0.1863715$ & $7.389714\cdot10^{-4}$ \\
$500$ & $48.14281$ & $0.154973$ & $6.045997\cdot10^{-4}$ \\
$1000$ & $43.12342$ & $0.1373663$ & $5.303814\cdot10^{-4}$ \\
$2000$ & $40.98117$ & $0.1299307$ & $4.993136\cdot10^{-4}$ \\
$5000$ & $39.79356$ & $0.1258298$ & $4.82254\cdot10^{-4}$ \\
$10^4$ & $39.41261$ & $0.1245176$ & $4.768072\cdot10^{-4}$ \\
$10^5$ & $39.07585$ & $0.123359$ & $4.720024\cdot10^{-4}$ \\
\hline
\end{tabular}
\end{center}
\end{cor}
This is Proposition~\ref{prop:tightintermediate} with $r=1$ (and $m=1$, $T_1=H_0$, $\sigma_0=3/4$, $x_0=e^{L_0}$); the proof is otherwise immediate.

\section{Zeros near the 1-Line}\label{sec:zeros1}
In this section we aim to study the sum
\[ \sum_{\substack{0 <\gamma \leqslant T \\ \sigma < \beta <  1}} \frac{x^{\beta-1}}{\gamma^m} \]
 where $m$ is fixed (and typically small), $x>4$ is fixed (and often large), and where $T=\infty$ is an admissible value.  An additional complexity we wish to handle is to provide bounds that are uniform while allowing for $T=T(x)$ and $\sigma=\sigma(x)$ chosen in such a way as to balance contributions from other error terms. In particular, $\sigma(x) >3/4$ is chosen to balance the contribution from the intermediate zeros studied in the previous section.

Assuming $H_0 \ge e^{w_0}$ and $v(w)$ are as above, that $\sigma>1/2$ and $x>1$ we have
\[ \Sigma_{\sigma}^{m} := \sum_{\substack{0 <\gamma\leqslant T \\ \sigma < \beta <  1}} \frac{x^{\beta-1}}{\gamma^m}  \leqslant \sum_{\substack{H_0<\gamma \leqslant T \\ \beta > \sigma } } \frac{x^{-v(\log(\gamma))}}{\gamma^m}. \]
Writing the last sum as a  Riemann-Stieltjes integral, then integrating by parts we get:
\[\Sigma_{\sigma}^{m} \leqslant  \int_{H_0}^T  \frac{x^{-v(\log(t))}}{t^m} dN(\sigma,t)
=\frac{x^{-v(\log(T))}}{T^m}  N(\sigma,T) - \int_{H_0}^T N(\sigma,t) {d}\Bigl( \frac{x^{-v(\log(t))}}{t^m} \Bigr) \]
In the following proposition we will show that the function $\frac{x^{-v(\log(t))}}{t^m}$ has a unique critical point after which it is decreasing; this will allow us to bound the last integral.

\begin{prop}\label{prop:tm}
Let $G_x(t)=\frac{x^{-v(\log(t))}}{t^m}$ with $m>0$. Then for all $x>\exp{(\frac{-m}{v'(w_0)})}$, there is a unique  $t_m=t_m(x)>e^{w_0}$ such that $G_x'(t_m)=0$, and for $t>e^{w_0}$ one has $G_x'(t)<0$ if and only if $t>t_m$.
\end{prop}
\begin{proof}
We observe that $G_x'(t)$ equals
\begin{equation*}
-t^{-1-m}x^{-v(\log(t))}\left[m+\log(x)v'(\log(t))\right]
\end{equation*}
Therefore, $G_x'(t)=0$ if and only if $\log(x)=h(t)$, $t=t(x)$, where
\begin{equation*}
h(t)=\frac{-m}{v'(\log(t))}.
\end{equation*}
Observe that $h(t)$ is increasing for $t>e^{w_0}$ since $v'(w)$ is increasing for $w>w_0$, and $\lim_{t\to\infty}h(t)=\infty$, as $\lim_{w\to\infty}v'(w)=0$. Thus $h:(e^{w_0},\infty)\to(h(e^{w_0}),\infty)$ defines a bijection. This fact implies that for every $x>e^{h(e^{w_0})}$ there is a unique $t_m=t_m(x)$ satisfying $h(t_m)=\log(x)$. This $t_m$ satisfies the desired properties.
\end{proof}

\begin{defn}\label{defn:um}
We define $t_m(x)$ as in Proposition \ref{prop:tm}. Implicitly, $t_m$ depends on $v$. We shall also define $u_m=\log(t_m)$.
\end{defn}

The case $t_m \leqslant H_0$ can be reduced to the case $T_0 \leqslant H_0 \leqslant t_m \leqslant T$ by simply replacing $t_m$ by $H_0$ and noting that the difference is a positive integral. Consequently, for simplicity of exposition we shall assume that $T_0 \leqslant H_0 \leqslant t_m \leqslant T$. We have that $\int_{H_0}^{t_m} N(\sigma,t) \;{dG_x}\geqslant 0$. Combining this with the zero-density bound assumption and integrating by parts again we get
\[\Sigma_{\sigma}^{m} \leqslant\ \frac{x^{-v(\log(T))}}{T^m}  \widetilde{N}(\sigma,T) - \int_{t_m}^T \widetilde{N}(\sigma,t) dG_x = \frac{x^{-v(\log(t_m))}}{t_m^m}  \widetilde{N}(\sigma,t_m) + \int_{t_m}^T G_x(t)\;d\widetilde{N}. \]
Combining this with \[ \widetilde{N}(\sigma,T) = \sum_{i}  c_i T^{p_i(\sigma)}\log(T)^{q_i(\sigma)} \]
and \[ \int_{t_m}^T G_x(t)\;d\Bigl(t^p\log(t)^q\Bigr) = p\int_{t_m}^{T} \frac{x^{-v(\log(t))}}{t^{m+1-p}}\log(t)^q\;dt +  q\int_{t_m}^{T} \frac{x^{-v(\log(t))}}{t^{m+1-p}}\log(t)^{q-1}\;dt\]
we obtain \begin{align}\Sigma_{\sigma}^{m} \le \sum_{i}c_i (E_{i1}(x) + E_{i2}(x) + E_{i3}(x)), \label{eq:edge} \end{align}
where
\begin{align*}
E_{i1}(x) &:=  \frac{x^{-v(\log(t_m))}}{t_m^m} t_m^{p_i(\sigma(x))}\log(t_m)^{q_i(\sigma(x))} \\
E_{i2}(x) &:= p_i(\sigma(x))\int_{t_m}^{T} \frac{x^{-v(\log(t))}}{t^{m+1-p_i(\sigma(x))}}\log(t)^{q_i(\sigma(x))}dt \\
E_{i3}(x) &:= q_i(\sigma(x))\int_{t_m}^{T} \frac{x^{-v(\log(t))}}{t^{m+1-p_i(\sigma(x))}}\log(t)^{q_i(\sigma(x))-1}\;dt.
\end{align*}
We observe that essentially all we have to do is to bound the integral
\begin{align*}
I_{v,m,p,q}(x) := \int_{t_m}^{T} \frac{x^{-v(\log(t))}}{t^{m+1- p}}\log(t)^{q}dt.
\end{align*}
 Explicitly,
\begin{align*}
E_{i2}(x)&=p_i(\sigma(x))I_{v,m,p_i(\sigma(x)),q_i(\sigma(x))}(x), \\
E_{i3}(x)&=q_i(\sigma(x))I_{v,m,p_i(\sigma(x)),q_i(\sigma(x))-1}(x).
\end{align*}

\subsection[Bounds on the main integral for abstract v]{Bounds on $I_{v,m,p,q}(x) $ for Abstract $v$}\label{subsec:nearoneabstract}

We now focus our attention on bounding the integrals $I_{v,m,p,q}(x)$.  We first note that although the roles of $m$ and $p$ may seem redundant the term $t_m$ is defined in terms of $v$ and $m$, and does not depend on $p$. In this section we illustrate a very general approach to bounding these integrals. Throughout we assume $m>0$, $0\leqslant p<m$, $q\geqslant 0$ and $w_0>0$. For most results in this section the assumption $q \geqslant 0$ can be significantly weakened, often at the expense of needing to impose lower bounds on $x$ depending on $v$.
 As we do not need this in any of our applications, we do not pursue it here.

\begin{defn}\label{defn:Fxfx}
We define
\[  F_{x}(t) = F_{x,v,m,p,q}(t) :=\frac{x^{-v(\log(t))}}{t^{m-p}}\log(t)^q  \]
and
\[ f_{x}(u) = f_{x,v,m,p,q}(u) :=  -(m-p)u - \log(x)v(u) +  q\log(u) \]
so that $F_x(t) = \exp(f_x(\log(t)))$.
Note that by taking the change of variable $u=\log(t)$ in $ I_{v,m,p,q}(x)$, we get that:
\[   I_{v,m,p,q}(x) =\int_{u_m}^{\log(T)} \exp(f_x(u)) \;du.\]
\end{defn}

The assumptions on $v$ impose the following properties on $f_x$ slightly generalizing Proposition \ref{prop:tm}.

\begin{prop}\label{prop:fx}
Let $f_x$ be defined as above, then $f_x\in C^{2}([w_0,+\infty))$ and the following hold:
\begin{enumerate}
\item For all $x>\exp{(\frac{-m}{v'(w_0)})}$, there is a unique $u_{m-p,q}(x)\geqslant u_m(x)$ such that $f_x'(u_{m-p,q}(x))=0$ and, for $u>w_0$, $f_x'(u)<0$ if and only if $u>u_{m-p,q}(x)$; moreover $u_{m-p,q}(x)>u_m(x)$ unless $p=q=0$, in which case $u_{m,0}=u_m$.
\item $v'(u_{m-p,q}(x))\in O(\frac{1}{\log(x)})$, explicitly $\lim\limits_{\substack{x\to \infty}} \log(x)\, v'(u_{m-p,q}(x))=-m+p$.
\item  For $x>1$ and $u>w_0$, $f''_x(u)<0$, $f''_x$ is increasing and
\[ \lim\limits_{\substack{x\to \infty}} u_{m-p,q}(x)\, f_x''(au_{m-p,q}(x))=C_a:=(m-p)c_a< 0, \]
for all $a\geqslant 1$.
\end{enumerate}
\end{prop}
\begin{proof}
    The regularity of $f_x$ follows directly from the regularity of $v$, with first derivative:
    \[f'_x(u)=-m+p-\log(x)v'(u)+\frac{q}{u}\]
    and second derivative:
 \[f''_x(u)=-\log(x)v''(u)-\frac{q}{u^2}.\]
Recalling that for $u>w_0$, $v''(u)>0$ and decreases to $0$, we have that $f''_x<0$ and increases to $0$ (as $q\geqslant 0$), in particular $f'_x$ is strictly decreasing. For $x>\exp{(\frac{-m}{v'(w_0)})}$, $-\log(x)v'(\log(t_m))=m$, thus
\[f'_x(\log(t_{m}))=p+\frac{q}{\log(t_{m})}\geqslant 0\]
(with equality only when $p=q=0$) and
\[\lim_{u\to \infty} f_x'(u)=-m+p<0,\]
so there exists a unique $u_{m-p,q}=u_{m-p,q}(x)\geqslant\log(t_m)$, such that $f_x'(u_{m-p,q})=0$, and $f_x'(u)<0$ if and only if $u>u_{m-p,q}$. It follows directly that $u_{m-p,q} > u_m$ (unless $p=q=0$).
\par Note that $\lim\limits_{\substack{x\to \infty}} v'(\log(t_m))=\lim\limits_{\substack{x\to \infty}}\frac{-m}{\log(x)}=0$, so $\lim\limits_{\substack{x\to \infty}}t_m(x)=+\infty$, and so $\lim\limits_{\substack{x\to \infty}}u_{m-p,q}(x)=+\infty$. Now for $x$ big enough, $f'_x(u_{m-p,q}(x))=0$, so 
\[\log(x)v'(u_{m-p,q}(x))=-m+p+\frac{q}{u_{m-p,q}(x)},\]
consequently $\lim\limits_{\substack{x\to \infty}}\log(x)v'(u_{m-p,q}(x))=-m+p$.
\par Now for $a\geqslant 1$ we have:
\[u_{m-p,q}(x)f''_x(au_{m-p,q}(x))=-\log(x)u_{m-p,q} v''(au_{m-p,q})-\frac{q}{a^2u_{m-p,q}},\]
combining this with
\[ v''(au)=O_a(v'(u)/u), \qquad \log(x)v'(u_{m-p,q}(x))=O(1), \qquad u_{m-p,q}(x)\to\infty, \]
we get that
\begin{align*}
\lim\limits_{\substack{x\to \infty}}u_{m-p,q}(x)f''_x(au_{m-p,q}(x))
&=\lim\limits_{\substack{x\to \infty}}\Bigl(-\log(x)v'(u_{m-p,q}(x))\Bigr) \\
&\qquad\qquad\times\Bigl(\frac{u_{m-p,q}(x)v''(au_{m-p,q}(x))}{v'(u_{m-p,q})}\Bigr)\\
&=(m-p)c_a\\
&=C_a< 0.\qedhere
\end{align*}
\end{proof}

\begin{Rem}
The quantity $u_m(x)$ of Definition~\ref{defn:um} is simply the value $u_{m,0}$ as defined above. In what follows we shall continue to use both notations which should not cause confusion.
\end{Rem}

\begin{prop}\label{prop:tailintegralabstract}
Suppose that $x>\exp(-m/v'(w_0))$. If $T(x)=t_{m-p,q}(x)^{1+\eta}$, for $\eta>0$, then
\[  \frac{1}{F_x(t_{m-p,q}(x))} \int_{T(x)}^\infty \frac{x^{-v(\log(t))}}{t^{m+1-p}}\log(t)^q  \,dt  <   \frac{1}{-f'_x((1+\eta)u_{m-p,q})} \] hence is bounded as $x\to \infty$.
\end{prop}
\begin{proof}
We set $U(x) = \log(T(x))$ and we perform the $u$ substitution $u=\log(t)$ so that
\[ \int_{T(x)}^\infty \frac{x^{-v(\log(t))}}{t^{m+1-p}}\log(t)^q \,dt= \int_{U(x)}^{\infty} \exp(f_x(u)) du  \]
We apply Taylor's theorem and use that $f'_x(U(x))<0$ and $f''_x<0$
\[ f_x(u) = f_x(U) + f'_x(U)(u-U) + \frac{f''_x(u')}{2}(u-U)^2 <  f_x(U) + f'_x(U)(u-U) \]
to get
\begin{align*}
\int_{U(x)}^\infty \exp( f_x(u)) du
&< \int_{U(x)}^\infty \exp(f_x(U))\, \exp(f'_x(U)(u-U)) du \\
&=  \frac{1}{-f'_x(U(x))}F_x(T(x)).
\end{align*}
Now we prove that  $\frac{1}{-f'_x(U(x))}$ is bounded; since moreover $\frac{F_x(T(x))}{F_x(t_{m-p,q}(x))}\leqslant 1$ (as $f_x$ attains its maximum at $u_{m-p,q}(x)$ by Proposition \ref{prop:fx}), combining this with the above inequality we get the desired result.
\par By Taylor's theorem
\[ -f'_x(U(x)) = -f_x''(u')(U(x)-u_{m-p,q}(x)) \]
for some $u'\in(u_{m-p,q}(x),U(x))$, and $f_x''$ is increasing, so
\begin{align*}
-\eta u_{m-p,q}(x) f_x''((1+\eta)u_{m-p,q}(x))
&< -f_x''(u')(U(x)-u_{m-p,q}(x)) \\
&< -\eta u_{m-p,q}(x)f_x''(u_{m-p,q}(x)),
\end{align*}
and $\lim\limits_{\substack{x\to \infty}}u_{m-p,q}(x)\, f_x''((1+\eta)u_{m-p,q}(x))=C_{(1+\eta)}< 0$, so $\frac{1}{-f'_x(U(x))}$ is bounded.
\end{proof}
\begin{Rem}
For $\eta>0$ and  sufficiently large $x$ the ratio $\frac{F_x(T(x))}{F_x(t_{m-p,q}(x))}$ is decreasing in $x$; we do not establish this in full generality.
\end{Rem}

\begin{prop}\label{prop:mainintegralabstract}
Suppose that $x>\exp(-m/v'(w_0))$. Let  $\delta_{\rm low}(x) = (u_{m-p,q}(x)-u_{m}(x))\sqrt{-f_x''(u_{m-p,q}(x))/2}$ then
\[ \frac{\int_{t_m(x)}^{t_{m-p,q}(x)} \frac{x^{-v(\log(t))}}{t^{m+1-p}}\log(t)^q  \,dt}{ F_x( t_{m-p,q}(x))} \leqslant  \operatorname{erf}\left(\delta_{\rm low}(x) \right)\frac{\sqrt{\pi}}{\sqrt{-2f_x''(u_{m-p,q}(x))}} \]
and moreover if $T(x) > t_{m-p,q}(x)$ then
\[ \frac{\int_{t_{m-p,q}(x)}^{T(x)} \frac{x^{-v(\log(t))}}{t^{m+1-p}}\log(t)^q  \,dt}{ F_x( t_{m-p,q}(x))}<   \frac{\sqrt{\pi}}{\sqrt{-2f_x''(\log(T(x)))}}. \]
\end{prop}
\begin{proof}
We set $U(x) = \log(T(x))$ and we perform the $u$ substitution $u=\log(t)$ so that
\[ \int_{t_m(x)}^{T(x)} \frac{x^{-v(\log(t))}}{t^{m+1-p}}\log(t)^q \,dt= \int_{u_m(x)}^{U(x)} \exp(f_x(u)) du. \]
We apply Taylor's theorem, and use that $f_x''(u)$ is increasing, to find that for $u \in [u_m(x),u_{m-p,q}]$ we have
\[   f_x(u)  \leqslant   f_x(u_{m-p,q}(x)) + \frac{f''_x(u_{m-p,q}(x))}{2}(u-u_{m-p,q})^2 .\]
whereas for $u \in [u_{m-p,q},U(x)]$ we have
\[   f_x(u)  \leqslant   f_x(u_{m-p,q}) + \frac{f''_x(U(x))}{2}(u-u_{m-p,q})^2 .\]
Therefore, we have respectively
\[ \frac{\int_{u_m(x)}^{u_{m-p,q}(x)}\exp(f_x(u))du}{F_x(t_{m-p,q}(x))}  \leqslant  \int_{u_m(x)}^{u_{m-p,q}(x)}\exp\left(\frac{f''_x(u_{m-p,q})}{2}(u-u_{m-p,q})^2\right) du, \]
and
\[ \frac{\int_{u_{m-p,q}(x)}^{U(x)} \exp(f_x(u))du}{F_x(t_{m-p,q}(x))}  \leqslant  \int_{u_{m-p,q}(x)}^{U(x)} \exp\left(\frac{f''_x(U(x))}{2}(u-u_{m-p,q})^2\right) du, \]
and performing the substitutions
\[ \delta = \sqrt{\frac{-f''_x(u_{m-p,q})}{2}}(u-u_{m-p,q}) \qquad\text{and}\qquad \delta = \sqrt{\frac{-f''_x(U(x))}{2}}(u-u_{m-p,q}) \]
we obtain the bounds
\[  \sqrt{\frac{2}{-f''_x(u_{m-p,q}(x))}}\int_{-\delta_{\rm low} }^0 \exp(-\delta^2)d\delta = \operatorname{erf}\left(\delta_{\rm low} \right)\frac{\sqrt{\pi}}{\sqrt{-2f''_x(u_{m-p,q}(x))}}.\]
 and
\[  \sqrt{\frac{2}{-f''_x(U(x))}}\int_{0}^\infty \exp(-\delta^2)d\delta = \frac{\sqrt{\pi}}{\sqrt{-2f''_x(U(x))}}\]
which give the results.
\end{proof}
\begin{Rem}\label{rem:mainintegrallower}
A matching lower bound holds provided $T(x) = t_{m-p,q}(x)^{1+\eta(x)}$ where $1 > \eta(x) > \frac{1}{\sqrt{u_{m-p,q}(x)}}$ by applying Taylor's theorem to obtain lower bounds, though we do not need it anywhere. One easily obtains:
 \[  \int_{t_m(x)}^{T(x)} \frac{x^{-v(\log(t))}}{t^{m+1-p}}\log(t)^q\,dt  \asymp  \sqrt{\log( t_{m-p,q}(x))} F_x( t_{m-p,q}(x)). \]
A more careful analysis, see Proposition \ref{prop:Cvalue}, combining Propositions \ref{prop:tailintegralabstract} and \ref{prop:mainintegralabstract} with a carefully chosen splitting point $\eta$ will often allow one to 
obtain a more precise asymptotic. A value $\eta(x)$ of the form
\[ \frac{C\sqrt{2\log(u_{m-p,q})}}{\sqrt{(m-p)u_{m-p,q}}} \]
for some constant $C$ will tend to balance the terms coming from Propositions  \ref{prop:tailintegralabstract}  and \ref{prop:mainintegralabstract} with larger values reducing the weight of the tail.
\end{Rem}

\subsection[Bounds on the main integral for specific v]{Bounds on $I_{v,m,p,q}(x) $ for Specific $v$}
We assume a common shape of the usual zero-free regions we know by setting
\begin{align*}
v(u):=v_{a,b,c}(u)=cu^a\log(u)^b,
\end{align*}
where $-1 \leqslant a < 0$ and $c>0$. For example in the case of de la Vallée Poussin zero-free region $a=-1$ and $b=0$, and in the case of Littlewood $a=-1$ and $b=1$ and finally in Korobov-Vinogradov zero-free region $a=-2/3$ and $b=-1/3$.

Our strategy in this section proceeds in the following way:
\begin{itemize}
\item In section \ref{subsubsec:uapprox} we introduce the specific form of $u_{n,q}$ and its recursively defined component $\widehat{u}_{n,q}$. We study how to approximate $\widehat{u}_{n,q}$ in applications, especially in the case of $b<0$.
\item In section \ref{subsubsec:comparison} we introduce a proposition which creates a link between $u_{m-p,q}$ and $u_{m,0}$.
\item In section \ref{subsubsec:explicitprop} we introduce the necessary steps to apply Propositions \ref{prop:tailintegralabstract} and \ref{prop:mainintegralabstract} explicitly.
\item In section \ref{subsubsec:Ebounds} we show how to give computable upper bounds for $E_{i1}(x)$, $E_{i2}(x)$ and $E_{i3}(x)$ in \eqref{eq:edge}.
\item In section \ref{subsubsec:Eboundsparameters} we will finally insert all the parameters of the explicit zero-density bound and Korobov-Vinogradov zero-free region and compute some upper bounds for $E_{i1}(x)$, $E_{i2}(x)$ and $E_{i3}(x)$ and prove necessary decreasing properties by applying the results we have derived in the previous subsections, especially the bounding theory of $u_{n,q}$.
\end{itemize}
\subsubsection[Computing the saddle point]{Computing $u_{n,q}$}\label{subsubsec:uapprox}

We assume that $n>0$ and that $x$ is large enough so that by Proposition \ref{prop:fx} a unique critical point exists, in practice this means $x$ large enough so that we are past the zeros and poles of the expressions in \eqref{def:hatu}. 
In order to apply the results of Section \ref{subsec:nearoneabstract} in this setting we shall need to compute the critical value $u_{n,q}(x)$ which shall have the defining equation 
\begin{align*}
v'(u_{n,q}(x)) = \frac{-n}{\log(x)}\left( 1 - \frac{q}{nu_{n,q}(x)} \right)
\end{align*}
As we know the shape of $v(u)$ we can differentiate it:
\begin{align*}
v'(u) = c u^{a-1}\log(u)^{b}\left(a+\frac{b}{\log(u)}\right).
\end{align*}
Hence,
\begin{align*}
c u_{n,q}(x)^{a-1}\log(u_{n,q}(x))^{b}\left(a+\frac{b}{\log(u_{n,q}(x))}\right) = \frac{-n}{\log(x)}\left( 1 - \frac{q}{nu_{n,q}(x)} \right).
\end{align*}
There is no closed form expression for $u_{n,q}(x)$ if $b \not= 0$, but we can approximate it recursively in the following way. First we find a recursive equation for $u_{n,q}(x)$:
\begin{equation}\label{def:u}
\begin{split}
u_{n,q}(x) = \Biggl( &\frac{-n}{c\log(x)}\log(u_{n,q}(x))^{-b}\left(a+\frac{b}{\log(u_{n,q}(x))}\right)^{-1} \\
&\times\left( 1 - \frac{q}{nu_{n,q}(x)} \right)\Biggr)^{-1/(1-a)}.
\end{split}
\end{equation}
Now set
\begin{align}
\widetilde{u}_{n,q}(x) := \left(\frac{(1-a)^{-b}(-a)c}{n} \right)^{1/(1-a)}\log(x)^{1/(1-a)}\log\log(x)^{b/(1-a)} \label{def:tildeu}
\end{align}
and
\begin{align}
\begin{split}
\widehat{u}_{n,q}(x)&:= \left( 1+ \frac{\log(n^{-1}(1-a)^{-b}(-a)c)}{\log\log(x)} + b\frac{\log\log\log(x)}{\log\log(x)}-\frac{\log(\widehat{u}_{n,q}(x))}{\log\log(x)}\right)^{-b} \\ &\times\left( 1+\frac{b}{a\log(\widetilde{u}_{n,q}(x)\widehat{u}_{n,q}(x)^{-1/(1-a)})} \right)^{-1} \\ &\times\left( 1 - \frac{q}{n\widetilde{u}_{n,q}(x)\widehat{u}_{n,q}(x)^{-1/(1-a)}} \right), \label{def:hatu}
\end{split}
\end{align}
which allows us to conclude that
\begin{align}
u_{n,q}(x) = \widetilde{u}_{n,q}(x)\widehat{u}_{n,q}(x)^{-1/(1-a)}. \label{def:u2}
\end{align}
\begin{prop}\label{prop:limit1}
We have that $\lim_{x \to \infty} \widehat{u}_{n,q}(x) = 1$.
\end{prop}
\begin{proof}
From \eqref{def:u} we see that $\log(u_{n,q}(x)) \sim \frac{1}{1-a}\log\log(x)$ and we also note from \eqref{def:tildeu} that $\log(\widetilde{u}_{n,q}(x))\sim \frac{1}{1-a}\log\log(x)$. Hence, by \eqref{def:u2} we have
\begin{align*}
\lim_{x \to \infty} \frac{\log(\widehat{u}_{n,q}(x))}{\log\log(x)} = 0,
\end{align*}
which implies in \eqref{def:hatu} that $\lim_{x \to \infty} \widehat{u}_{n,q}(x) = 1$.
\end{proof}
To approximate $\widehat{u}_{n,q}(x)$ we define
\begin{align*}
\mathcal{H}(x,u)&:= \left( 1+ \frac{\log(n^{-1}(1-a)^{-b}(-a)c)}{\log\log(x)} + b\frac{\log\log\log(x)}{\log\log(x)}-\frac{\log(u)}{\log\log(x)}\right)^{-b} \\ &\times\left( 1+\frac{b}{a\log(\widetilde{u}_{n,q}(x)u^{-1/(1-a)})} \right)^{-1} \\ &\times\left( 1 - \frac{q}{n\widetilde{u}_{n,q}(x)u^{-1/(1-a)}} \right).
\end{align*}
In order to analyze $\mathcal{H}(x,u)$ we also denote its factors in the following way:
\begin{align*}
\mathcal{H}_1(x,u) &:=\left( 1+ \frac{\log(n^{-1}(1-a)^{-b}(-a)c)}{\log\log(x)} + b\frac{\log\log\log(x)}{\log\log(x)}-\frac{\log(u)}{\log\log(x)}\right)^{-b} \\
\mathcal{H}_2(x,u) &:= \left( 1+\frac{b}{a\log(\widetilde{u}_{n,q}(x)u^{-1/(1-a)})} \right)^{-1} =  1 - \frac{b}{b+a\log(\widetilde{u}_{n,q}(x)u^{-1/(1-a)})} \\
\mathcal{H}_3(x,u) &:= 1 - \frac{q}{n\widetilde{u}_{n,q}(x)u^{-1/(1-a)}}.
\end{align*}
\begin{Rem}\label{Rem:H}
There exists a positive constant $x_0$ so that the function $\mathcal{H}(x,u)$ is positive and it has no poles whenever $x > x_0$ and $u > \min\{u_0,\log\log(x)^{-|b|}\}$ where $0 < u_0 < u\leqslant\log\log(x)^{|b|}$ is small. Furthermore, the function $\mathcal{H}(x,u)$ is continuous in this domain. We also note that for any positive $x_1>x_0$ we have that $\mathcal{H}(x_1,u) = u$ holds if and only if $u = \widehat{u}_{n,q}(x_1)$.
\end{Rem}
\begin{prop}\label{prop:iteration2}
Let $u_0$ be some small positive constant, and let $w_{U}(x)$ and $w_{L}(x)$ be defined in such a way that $ u_0 < w_{L}(x) \leqslant \widehat{u}_{n,q}(x) \leqslant w_{U}(x) \leqslant \log\log(x)^{|b|}$ holds. Then we have eventually that
\begin{enumerate}
[label=(\roman*)]
\item if $b>0$, then
\begin{align*}
 \mathcal{H}(x,w_{L}(x)) \leqslant \widehat{u}_{n,q}(x) \leqslant \mathcal{H}(x,w_{U}(x)); \textrm{ and}
\end{align*}
\item if $b<0$, then
\begin{align*}
 \mathcal{H}(x,w_{U}(x)) \leqslant \widehat{u}_{n,q}(x) \leqslant \mathcal{H}(x,w_{L}(x)). 
\end{align*}
\end{enumerate}
with $ |\mathcal{H}(x,w_{U}(x))-\widehat{u}_{n,q}(x)| \leqslant |w_U(x)-\widehat{u}_{n,q}(x)|$ and $ |\mathcal{H}(x,w_{L}(x))-\widehat{u}_{n,q}(x)| \leqslant |w_L(x)-\widehat{u}_{n,q}(x)|$.
\end{prop}
\begin{proof}
We begin by noting that
\begin{align*}
\mathcal{H}(x,u) &= (1-a)^{-b}\frac{\log\log(x)^{b}\log(\widetilde{u}_{n,q}(x)u^{-1/(1-a)})^{1-b}}{\log(\widetilde{u}_{n,q}(x)u^{-1/(1-a)})+\frac{b}{a}} \\
&\qquad\times\biggl(1-\frac{q}{n}\exp\biggl(-\log(\widetilde{u}_{n,q}(x)u^{-1/(1-a)}) \biggr) \biggr)
\end{align*}
Now $\log(\widetilde{u}_{n,q}(x)u^{-1/(1-a)})$ is decreasing in $u$, which implies that
\begin{align*}
\log(\widetilde{u}_{n,q}(x)w_{U}(x)^{-1/(1-a)}) &\leqslant \log(\widetilde{u}_{n,q}(x)\widehat{u}_{n,q}(x)^{-1/(1-a)}) \\
&\leqslant \log(\widetilde{u}_{n,q}(x)w_{L}(x)^{-1/(1-a)}).
\end{align*}
Now if $b>0$, then $1-b < 1$ which implies that eventually for large enough $x$ we must have that $\mathcal{H}(x,u)$ is decreasing as $\log(\widetilde{u}_{n,q}(x)u^{-1/(1-a)})$ increases, which implies the case (i) and similarly if $b<0$ then $1-b > 1$ and $\mathcal{H}(x,u)$ will be increasing as $\log(\widetilde{u}_{n,q}(x)u^{-1/(1-a)})$ increases, which implies the case (ii).

Next, by setting $w(x)=D(x)\widehat{u}_{n,q}(x)$ where $\frac{1}{2}\log\log(x)^{-|b|}<D(x)<1$ if $w(x) = w_L(x)$ and $1<D(x)<2\log\log(x)^{|b|}$ if $w(x)=w_U(x)$ we note that
\begin{align*}
\frac{\mathcal{H}_1(x,w(x))}{\mathcal{H}_1(x,\widehat{u}_{n,q}(x))} &= \left(1 - \frac{\log(D(x))}{(1-a)\log(u_{n,q}(x))}\right)^{-b}, \\
\frac{\mathcal{H}_2(x,w(x))}{\mathcal{H}_2(x,\widehat{u}_{n,q}(x))} &= 1 \\
&\quad - \frac{b(1-a)/a}{(1-a)\log(u_{n,q}(x))} \\
&\qquad\quad \times\frac{\log(D(x))}{(1-a)\log(u_{n,q}(x))+b(1-a)/a-\log(D(x))}, \\
\frac{\mathcal{H}_3(x,w(x))}{\mathcal{H}_3(x,\widehat{u}_{n,q}(x))} &= 1- \frac{D(x)^{1/(1-a)}-1}{\frac{n}{q}u_{n,q}(x)-1},
\end{align*}
which will eventually imply that there exists a positive constant $K$ such that
\begin{align}
\left| \frac{\mathcal{H}(x,w(x))}{\widehat{u}_{n,q}(x)}  -1\right| \leqslant K\frac{\left|\log(D(x))\right|}{\log\log(x)} \label{closing-in}
\end{align}
holds and hence $ |\mathcal{H}(x,w_{U}(x))-\widehat{u}_{n,q}(x)| \leqslant |w_U(x)-\widehat{u}_{n,q}(x)|$ and $ |\mathcal{H}(x,w_{L}(x))-\widehat{u}_{n,q}(x)| \leqslant |w_L(x)-\widehat{u}_{n,q}(x)|$.
\end{proof}
\begin{prop}\label{Prop:iteration} Assume $b\not=0$. Let $f(x)$ be some function where for some small positive constant $u_0$ we have $u_0 < f(x) \leqslant \log\log(x)^{|b|}$. We denote $\mathcal{H}^k(x,f(x))$ as $k$-iterated $f(x)$ where as an example $\mathcal{H}^{1}(x,f(x))=\mathcal{H}(x,f(x))$, $\mathcal{H}^{2}(x,f(x))=\mathcal{H}(x,\mathcal{H}(x,f(x)))$, $\mathcal{H}^{3}(x,f(x))=\mathcal{H}(x,\mathcal{H}(x,\mathcal{H}(x,f(x))))$ and so on. Then there exists an $x'$ so that
\begin{align*}
\lim_{k \to \infty} \mathcal{H}^{k}(x,f(x)) = \widehat{u}_{n,q}(x)
\end{align*}
for every $x \geqslant x'$. 
\end{prop}
\begin{proof}
The proof is mainly obtained from Proposition \ref{prop:iteration2}. By looking at the proof and \eqref{closing-in} we see that by taking $w_k(x) = D_k(x)\widehat{u}_{n,q}(x)$ we have
\begin{align*}
|D_{k+1}(x)-1| = \left| \frac{\mathcal{H}(x,D_{k}(x)\widehat{u}_{n,q}(x))}{\widehat{u}_{n,q}(x)}  -1\right| \leqslant K\frac{|\log(D_{k}(x))|}{\log\log(x)},
\end{align*}
which by initial conditions implies that $\lim_{k\to\infty} D_k(x) = 1$.
\end{proof}
\begin{prop}\label{prop:Bounds}
Let $x_0$ be the same as in Remark \ref{Rem:H} and assume $\widetilde{u}_{n,q}(x) > \max\{e,q/n\}$. If $b <0$, then \begin{align*} \mathcal{H}(x,1) < \widehat{u}_{n,q}(x) < 1 \end{align*} holds whenever
\begin{align*}
\log\log(x) > -\log(n^{-1}(1-a)^{-b}(-a)c)-b\log\log\log(x)
\end{align*}
and
\begin{align*}
\log(n^{-1}(1-a)^{-b}(-a)c)+b\log\log\log(x) < 0
\end{align*}
hold.
\end{prop}
\begin{proof}
We note that the factors $\mathcal{H}_2(x,\widehat{u}_{n,q}(x))$ and $\mathcal{H}_3(x,\widehat{u}_{n,q}(x))$ are both always less than $1$. Also $\mathcal{H}_1(x,\widehat{u}_{n,q}(x))$ is less than $1$ if
\begin{align*}
\log(n^{-1}(1-a)^{-b}(-a)c)+b\log\log\log(x)-\log(\widehat{u}_{n,q}(x)) < 0,
\end{align*}
which always holds if $\widehat{u}_{n,q}(x) \geqslant 1$ and $\log(n^{-1}(1-a)^{-b}(-a)c)+b\log\log\log(x) < 0$. Hence, $\widehat{u}_{n,q}(x) < 1$ must hold. We then compute
\begin{align*}
\mathcal{H}(x,1)&:= \left( 1+ \frac{\log(n^{-1}(1-a)^{-b}(-a)c)}{\log\log(x)} + b\frac{\log\log\log(x)}{\log\log(x)}\right)^{-b} \\ &\times\left( 1-\frac{b}{b+a\log(\widetilde{u}_{n,q}(x))} \right) \\ &\times\left( 1 - \frac{q}{n\widetilde{u}_{n,q}(x)} \right). 
\end{align*}
and note that if $0 < u <1$, then $\mathcal{H}_1(x,1) < \mathcal{H}_1(x,u)$, $\mathcal{H}_2(x,1) < \mathcal{H}_2(x,u)$ and $\mathcal{H}_3(x,1) < \mathcal{H}_3(x,u)$ and since $0 < \widehat{u}_{n,q}(x) < 1$, we may conclude that $\mathcal{H}(x,1) < \widehat{u}_{n,q}(x)$. We then have to find values for $x$ so that $\mathcal{H}(x,1)$ is positive when $0<u<1$. We note that $\mathcal{H}_1(x,1) > 0$ if $\log(n^{-1}(1-a)^{-b}(-a)c)+b\log\log\log(x) > -\log\log(x)$, $\mathcal{H}_2(x,1)>0$ if $\widetilde{u}_{n,q}(x) > e$ and $\mathcal{H}_3(x,1) >0$ if $\widetilde{u}_{n,q}(x)>q/n$.
\end{proof}

\subsubsection[Comparison between the two saddle points]{Comparison between $u_{m-p(x),q}$ and $u_{m,0}$}\label{subsubsec:comparison}

In the following Proposition we will derive a useful relation between $u_{m-p(x),q}$ and $u_{m,0}$.
\begin{prop}\label{comparisonProp}
We have that 
\[ u_{m-p(x) ,q}(x) = \left(\frac{1}{\widehat{u}_{m-p(x),q}(x) (1-p(x)/m)}\right)^{1/(1-a)} \widetilde{u}_{m,0}(x) \]
and
\[ \log( u_{m-p(x) ,q}(x)  ) = \log(  \widetilde{u}_{m,0}(x) ) -  \frac{1}{1-a}\log(\widehat{u}_{m-p(x),q}(x) (1-p(x)/m)) .\]
\end{prop}
\begin{proof}
This is essentially by definition,
\[ u_{m ,0}(x) = \left(\frac{1}{\widehat{u}_{m,0}(x)}\right)^{1/(1-a)} \widetilde{u}_{m,0}(x) \]
and
\[ u_{m-p(x) ,q}(x) = \left(\frac{1}{\widehat{u}_{m-p(x) ,q}(x)}\right)^{1/(1-a)} \widetilde{u}_{m-p(x) ,q}(x) \]
using that
\[  \widetilde{u}_{m,0}(x) =  \left(1-p(x)/m\right)^{1/(1-a)}\widetilde{u}_{m-p(x) ,q}(x). \]
\end{proof}

\subsubsection{Explicit versions of the tail and main integral propositions}\label{subsubsec:explicitprop}

We now perform an explicit study of the terms appearing in Proposition \ref{prop:tailintegralabstract} and \ref{prop:mainintegralabstract}.%

First, by direct computation we have.
\begin{align*}
v'(u) &= c u^{a-1}\log(u)^{b}\left(a+\frac{b}{\log(u)}\right).\\
v''(u) &= c u^{a-2}\log(u)^b\left(a(a-1) + \frac{(2a-1)b}{\log(u)} + \frac{b^2-b}{\log(u)^2}\right).
\end{align*}

We now define, and by a direct calculation find:
\begin{align}
   \Phi_{1,\epsilon}(x) &= \frac{v'((1+\epsilon)u_{m-p,q}(x))}{(1+\epsilon)^{a-1}v'(  u_{m-p,q}(x) )}
   \\& =  \left(1 + \frac{\log(1+\epsilon)}{\log(u_{m-p,q})}\right)^b\left(1+\frac{b}{a\log((1+\epsilon)u_{m-p,q})}\right)\left(1+\frac{b}{a\log(u_{m-p,q})}\right)^{-1}\nonumber
\end{align}
and
\begin{align}
   \Phi_{2,\epsilon}(x) &=  \frac{(1+\epsilon)u_{m-p,q}(x)v''((1+\epsilon)u_{m-p,q}(x))}{-(1-a)v'((1+\epsilon) u_{m-p,q}(x) ) }
   \\&= 1 - \frac{b}{(1-a)\log((1+\epsilon)u_{m-p,q}(x))} \nonumber
   \\&\qquad + \frac{b}{(1-a)\log((1+\epsilon)u_{m-p,q}(x))(a\log((1+\epsilon)u_{m-p,q}(x)) + b)}\nonumber
\end{align}

\begin{lem}\label{lemma:Clemma}
Suppose $\epsilon > 0$ and denote by
\[  \Phi_{3,\epsilon}(x) =  \left(\left(1-\frac{q}{(m-p)u_{m-p,q}}\right)\Phi_{2,\epsilon}(x)\Phi_{1,\epsilon}(x)  +   \frac{(1+\epsilon)^{-a} q}{(m-p)(1-a)u_{m-p,q}} \right) \]
\begin{enumerate}
\item that
    \[ -f_x''((1+\epsilon)u_{m-p,q}(x)) = \frac{(m-p)(1-a)}{(1+\epsilon)^{2-a}u_{m-p,q}}\Phi_{3,\epsilon}(x) \]
\item that
\begin{align*}
\frac{(m-p)(1-a)\epsilon \Phi_{3,\epsilon}(x) }{(1+\epsilon)^{2-a}}
&\leqslant  -f_x'((1+\epsilon)u_{m-p,q}(x)) \\
&\leqslant  \epsilon u_{m-p,q}(x) (-f_x''(u_{m-p,q}(x)) )  = (m-p)(1-a)\epsilon \Phi_{3,0}(x)
\end{align*}
\item that
\begin{align*}
f_x((1+\epsilon)u_{m-p,q}(x)) -  f_x(u_{m-p,q}(x))
&\leqslant    - \frac{\epsilon u_{m-p,q}}{2}(-f_x'((1+\epsilon)u_{m-p,q}(x))) \\
&\leqslant - \frac{(m-p)(1-a)\epsilon^2 u_{m-p,q} \Phi_{3,\epsilon}(x) }{2(1+\epsilon)^{2-a}}
\end{align*}
\end{enumerate}
\end{lem}
\begin{proof}
A simple calculation gives
\begin{align*} -f_x''((1+\epsilon)u_{m-p,q}(x))
&= \log(x)v''((1+\epsilon)u_{m-p,q}(x)) + \frac{q}{((1+\epsilon)u_{m-p,q}(x))^2} \\
&= (1-a)\frac{-\log(x)v'((1+\epsilon)u_{m-p,q}(x))}{(1+\epsilon)u_{m-p,q}(x)}\Phi_{2,\epsilon}(x) \\
&\qquad + \frac{q}{((1+\epsilon)u_{m-p,q}(x))^2} \\
&= (1-a)\frac{-\log(x)v'(u_{m-p,q}(x))}{(1+\epsilon)^{2-a}u_{m-p,q}(x)}\Phi_{2,\epsilon}(x)  \Phi_{1,\epsilon}(x) \\
&\qquad + \frac{q}{((1+\epsilon)u_{m-p,q}(x))^2} \\
&=  \frac{(m-p)(1-a)\left(1-\frac{q}{(m-p)u_{m-p,q}}\right)}{(1+\epsilon)^{2-a}u_{m-p,q}(x)}\Phi_{2,\epsilon}(x)  \Phi_{1,\epsilon}(x) \\
&\qquad + \frac{q}{((1+\epsilon)u_{m-p,q}(x))^2}.
\end{align*}
Simplifying then gives the first claim.

 Taylor's theorem, applied to $f_x'(u)$, together with the previous bound, gives the second. Then convexity of $f_x'(u)$, together with the previous bound, gives the third.
\end{proof}
\begin{Rem}
All of $\Phi_{1,\epsilon}(x)$, $\Phi_{2,\epsilon}(x)$ and $\Phi_{3,\epsilon}(x)$ have limits of $1$ as $x\to\infty$.
However, if $b<0$, then $\Phi_{1,\epsilon}(x) < 1$ and $\Phi_{2,\epsilon}(x) > 1$, and the exact behavior of $\Phi_{3,\epsilon}(x)$ depends on $\epsilon$, though it is eventually monotonic.
\end{Rem}

\begin{prop}\label{prop:Cvalue}
Define 
\[ \eta(x) := \frac{\sqrt{2\log(u_{m-p,q}(x))}}{\sqrt{(m-p)(1-a)u_{m-p,q}(x)}} \]
Suppose $T(x)\ge  t_{m-p,q}(x)^{1+\eta(x)}$, with $T(x) = \infty$ admissible.
 Then
\[  \frac{\int_{t_m(x)}^{T(x)} \frac{x^{-v(\log(t))}}{t^{m+1-p}}\log(t)^q  \,dt}{ F_x( t_{m-p,q}(x))}
<  \Phi_{4,\eta}(x)  \frac{\sqrt{2\pi u_{m-p,q}(x)}}{\sqrt{(m-p)(1-a)}}\]
where
\begin{align*}  \Phi_{4,\eta}(x)  &= \frac{1}{2}\operatorname{erf}\left(\delta_{\rm low}(x) \right) \Phi_{3,0}(x)^{-1/2} + \frac{1}{2}(1+\eta(x))^{(2-a)/2}\Phi_{3,\eta(x)}(x)^{-1/2}   \\&\qquad +
\frac{(1+\eta(x))^{2-a}}{2\sqrt{\pi\log(u_{m-p,q}(x))}\;\Phi_{3,\eta}(x)}\;
 u_{m-p,q}(x)^{-\frac{\Phi_{3,\eta}(x)}{(1+\eta(x))^{2-a}}}.
\end{align*}
\end{prop}
\begin{proof}
We split the integral at $t_{m-p,q}$ and $t_{m-p,q}^{(1+\eta)}$ and apply respectively Propositions \ref{prop:mainintegralabstract} and \ref{prop:tailintegralabstract} to bound the parts.

Then, substituting the upper and lower bounds from Lemma \ref{lemma:Clemma}, and simplifying completes the proof.
\end{proof}
\begin{Rem}
In our applications we shall have that $\delta_{\rm low}(x)$  tends to $0$ and hence the entire lead term, including the $\operatorname{erf}$ term, tends to $0$ from above.
\end{Rem}

\subsubsection[Bounds on the three error components]{Bounds on $E_{i1}(x)$, $E_{i2}(x)$ and $E_{i3}(x)$}\label{subsubsec:Ebounds}
We will analyze how to bound $E_{i1}(x)$, $E_{i2}(x)$ and $E_{i3}(x)$ mentioned after Definition \ref{defn:um}. Our approach is to present the bounds in a form where we can isolate the main factor $\exp(-m(1-a^{-1})\widetilde{u}_{m,0}(x))$, which is
\begin{align*}
\exp\left( -m(1-a^{-1})(m^{-1}(1-a)^{-b}(-a)c)^{1/(1-a)} \log(x)^{1/(1-a)}\log\log(x)^{b/(1-a)} \right).
\end{align*}

\begin{prop}\label{Prop:Ei1}
Assume a zero-free region $v(u)$ as defined in this subsection and that $x$ is large enough so that the point $u_{m,0}(x)$ can be defined. Then $E_{i1}(x) $ is given by
\begin{align*}
&\exp(-m(1-a^{-1})\widetilde{u}_{m,0}(x))
\\&\quad\times \exp\biggl(-m(1-a^{-1})\biggl(\biggl(1  + \frac{b}{a(a-1)\log(u_{m,0}(x)) +b(a-1)}\\&\qquad\quad-\frac{p_i(\sigma(x))}{m(1-a^{-1})}\biggr)\left(\frac{1}{\widehat{u}_{m,0}(x)}\right)^{1/(1-a)}-1\biggr)\widetilde{u}_{m,0}(x)+ q_i(\sigma(x))\log(u_{m,0}(x))\biggr).
\end{align*}
\end{prop}
\begin{proof}
We can prove this by deriving
\begin{align*}
E_{i1}(x) 
&=  \frac{x^{-v(\log(t_m))}}{t_m^m} t_m^{p_i(\sigma(x))}\log(t_m)^{q_i(\sigma(x))} \\
&= \exp\biggl( u_{m,0}(x)\biggl( -m+p_i(\sigma(x))-\log(x)\frac{v(u_{m,0}(x))}{u_{m,0}(x)} \biggr) \\
&\qquad\qquad + q_i(\sigma(x))\log(u_{m,0}(x))\biggr). 
\end{align*}
Now, by the definition of $v(u)$ we have $ \frac{v(u)}{u} = \frac{v'(u)}{a+\frac{b}{\log(u)}}, $
and, since $f_x'(u_{m-p,q}(x))=0$, we have $\log(x)v'(u_{m-p,q}(x))=-m+p+\frac{q}{u_{m-p,q}(x)}$ as in the proof of Proposition \ref{prop:fx}, so that we can rewrite $E_{i1}(x) $
\begin{align*}
&=\exp\biggl(-mu_{m,0}(x)\biggl(1-a^{-1} + \frac{b/a}{a\log(u_{m,0}(x)) +b}-\frac{p_i(\sigma(x))}{m}\biggr) \\
&\qquad\qquad + q_i(\sigma(x))\log(u_{m,0}(x))\biggr) \\
&= \exp\biggl(-m(1-a^{-1})\biggl(1
 + \frac{b}{a(a-1)\log(u_{m,0}(x)) +b(a-1)} \\
&\qquad\qquad -\frac{p_i(\sigma(x))}{m(1-a^{-1})}\biggr)\left(\frac{1}{\widehat{u}_{m,0}(x)}\right)^{1/(1-a)} \widetilde{u}_{m,0}(x)+ q_i(\sigma(x))\log(u_{m,0}(x))\biggr)\\
 &=\exp(-m(1-a^{-1})\widetilde{u}_{m,0}(x))\\
 & \quad \times \exp\biggl(-m(1-a^{-1})\biggl(\biggl(1 + \frac{b}{a(a-1)\log(u_{m,0}(x)) +b(a-1)} \\
&\qquad\quad -\frac{p_i(\sigma(x))}{m(1-a^{-1})}\biggr)\left(\frac{1}{\widehat{u}_{m,0}(x)}\right)^{1/(1-a)}-1\biggr)\widetilde{u}_{m,0}(x)+ q_i(\sigma(x))\log(u_{m,0}(x))\biggr).
\end{align*}
\end{proof}
\begin{prop}\label{Prop:Ei2}
If $p_i(\sigma(x))=0$, then $E_{i2}(x)=0$. Assume $p_i(\sigma(x)) >0$ and $T(x)\geqslant t_{m-p,q}(x)^{1+\eta(x)}$. We have that $E_{i2}(x) $ is bounded by
\begin{align*}
& \exp(-m(1-a^{-1})\widetilde{u}_{m,0}(x)) \\
&\  \times \exp\biggl(-m(1-a^{-1})\Biggl(\biggl(1 + \frac{b}{a(a-1)\log(u_{m-p_i(\sigma(x)),q_i(\sigma(x))}(x))+b(a-1)} \\&\qquad - \frac{p_i(\sigma(x))}{m(1-a^{-1})}\biggl(1-\frac{\log(u_{m-p_i(\sigma(x)),q_i(\sigma(x))}(x))}{a\log(u_{m-p_i(\sigma(x)),q_i(\sigma(x))}(x))+b} \biggr) \biggr) \\&\qquad\times\biggl( \frac{1}{\widehat{u}_{m-p_i(\sigma(x)),q_i(\sigma(x))}(x)(1-p_i(\sigma(x))/m)} \biggr)^{1/(1-a)} -1\Biggr)\widetilde{u}_{m,0}(x) \\&\quad+q_i(\sigma(x))\biggl( -\frac{\log(u_{m-p_i(\sigma(x)),q_i(\sigma(x))}(x))}{a\log(u_{m-p_i(\sigma(x)),q_i(\sigma(x))}(x))+b} + \log(u_{m-p_i(\sigma(x)),q_i(\sigma(x))}(x)) \biggr) \\ &\quad +\log(p_i(\sigma(x))) +\tfrac{1}{2}\log(u_{m-p_i(\sigma(x)),q_i(\sigma(x))}(x))\biggr) \\
&\quad \times \Phi_{4,\eta}(x) \sqrt{\frac{2\pi}{(m-p_i(\sigma(x)))(1-a)}},
\end{align*}
where $\Phi_{4,\eta}(x)$ is from Proposition \ref{prop:Cvalue}.
\end{prop}
\begin{proof}
Using Proposition \ref{prop:Cvalue} we derive
\begin{align*}
E_{i2}(x)
&=  p_i(\sigma(x))\int_{t_m}^{T} \frac{x^{-v(\log(t))}}{t^{m+1-p_i(\sigma(x))}}\log(t)^{q_i(\sigma(x))}dt \\
&\leqslant \Phi_{4,\eta}(x) \sqrt{\frac{2\pi}{(m-p_i(\sigma(x)))(1-a)}} \\
&\quad\times p_i(\sigma(x)) \sqrt{u_{m-p_i(\sigma(x)),q_i(\sigma(x))}(x)} \\
&\quad\times\exp\bigl(f_{x,v,m,p_i(\sigma(x)),q_i(\sigma(x))}(u_{m-p_i(\sigma(x)),q_i(\sigma(x))}(x) )\bigr) \\
&= \Phi_{4,\eta}(x) \sqrt{\frac{2\pi}{(m-p_i(\sigma(x)))(1-a)}} \\
&\quad\times\exp\bigl(f_{x,v,m,p_i(\sigma(x)),q_i(\sigma(x))}(u_{m-p_i(\sigma(x)),q_i(\sigma(x))}(x) ) \\
&\qquad\qquad + \log(p_i(\sigma(x))) + \tfrac{1}{2}\log(u_{m-p_i(\sigma(x)),q_i(\sigma(x))})\bigr)
\end{align*}
and continue in the similar way as was done in the proof of Proposition \ref{Prop:Ei1}. Note that we applied Proposition \ref{comparisonProp}.
\end{proof}

\begin{prop}\label{Prop:Ei3}
Assume $q_i(\sigma(x)) \geqslant 1$ and $T(x)\geqslant t_{m-p,q-1}(x)^{1+\eta(x)}$. We have that $E_{i3}(x)$ is bounded by
\begin{align*}
 & \exp(-m(1-a^{-1})\widetilde{u}_{m,0}(x)) \\
& \times \exp\biggl(-m(1-a^{-1})\Biggl(\biggl(1 + \frac{b}{a(a-1)\log(u_{m-p_i(\sigma(x)),q_i(\sigma(x))-1}(x))+b(a-1)} \\&\qquad\qquad - \frac{p_i(\sigma(x))}{m(1-a^{-1})}\biggl(1-\frac{\log(u_{m-p_i(\sigma(x)),q_i(\sigma(x))-1}(x))}{a\log(u_{m-p_i(\sigma(x)),q_i(\sigma(x))-1}(x))+b} \biggr) \biggr) \\&\qquad\qquad\times\biggl( \frac{1}{\widehat{u}_{m-p_i(\sigma(x)),q_i(\sigma(x))-1}(x)(1-p_i(\sigma(x))/m)} \biggr)^{1/(1-a)} -1\Biggr)\widetilde{u}_{m,0}(x) \\&\qquad+(q_i(\sigma(x))-1)\biggl( -\frac{\log(u_{m-p_i(\sigma(x)),q_i(\sigma(x))-1}(x))}{a\log(u_{m-p_i(\sigma(x)),q_i(\sigma(x))-1}(x))+b} \\
&\qquad\qquad\qquad\qquad + \log(u_{m-p_i(\sigma(x)),q_i(\sigma(x))-1}(x)) \biggr) 
\\ &\qquad +\log(q_i(\sigma(x))) +\tfrac{1}{2}\log(u_{m-p_i(\sigma(x)),q_i(\sigma(x))-1}(x))\biggr) \\
&\quad \times \widetilde{\Phi}_{4,\eta}(x) \sqrt{\frac{2\pi}{(m-p_i(\sigma(x)))(1-a)}},
\end{align*}
where $\widetilde{\Phi}_{4,\eta}(x)$ is obtained by changing $q$ in $\Phi_{4,\eta}(x)$ from Proposition \ref{prop:Cvalue} to $q-1$.
\end{prop}
\begin{proof}
The derivation of this formula begins with noting that
\begin{align*}
E_{i3}(x)
&=  q_i(\sigma(x))\int_{t_m}^{T} \frac{x^{-v(\log(t))}}{t^{m+1-p_i(\sigma(x))}}\log(t)^{q_i(\sigma(x))-1}dt \\
&\leqslant \widetilde{\Phi}_{4,\eta}(x) \sqrt{\frac{2\pi}{(m-p_i(\sigma(x)))(1-a)}} \\
&\quad\times  q_i(\sigma(x)) \sqrt{u_{m-p_i(\sigma(x)),q_i(\sigma(x))-1}(x)} \\
&\quad\times\exp\bigl(f_{x,v,m,p_i(\sigma(x)),q_i(\sigma(x))-1}(u_{m-p_i(\sigma(x)),q_i(\sigma(x))-1}(x) )\bigr) \\
&= \widetilde{\Phi}_{4,\eta}(x) \sqrt{\frac{2\pi}{(m-p_i(\sigma(x)))(1-a)}} \\
&\quad\times\exp\bigl(f_{x,v,m,p_i(\sigma(x)),q_i(\sigma(x))-1}(u_{m-p_i(\sigma(x)),q_i(\sigma(x))-1}(x) ) \\
&\qquad\qquad + \log(q_i(\sigma(x))) + \tfrac{1}{2}\log(u_{m-p_i(\sigma(x)),q_i(\sigma(x))-1})\bigr)
\end{align*}
and we continue in a similar way as was done in the proof of Proposition \ref{Prop:Ei2} with minor modifications.
\end{proof}

\subsubsection[Bounds on the three error components for specific parameters]{Bounds for $E_{i1}(x)$, $E_{i2}(x)$ and $E_{i3}(x)$ for Specific Parameters}\label{subsubsec:Eboundsparameters}

Throughout this subsection we assume $m=1$, $v(u) = cu^{a}(\log(u))^{b}$ with $a=-2/3$, $b=-1/3$ and $c=1/53.989$ \cite{BellottiZFR}, $p_1(\sigma) =\frac{8(1-\sigma)}{3}$, $p_2(\sigma)=0$, $q_1(\sigma)=5-2\sigma$ and $q_2(\sigma)=2$. 
We also set $T(x) = t_1(x)^{1+\epsilon}$ with $\epsilon=2$ and $t_1(x)$ is as defined in Proposition \ref{prop:tm}.

Next, we choose $\sigma$ as
\begin{align*}
\sigma(x) = 1+\frac{-\widetilde{u}_{1,0}(x)-\log(x)v(\widetilde{u}_{1,0}(x))+4\log(\widetilde{u}_{1,0}(x))}{\log(x)}.
\end{align*}
\begin{lem}\label{lem:sigmabounds}
The following bounds related to $\sigma(x)$ hold:
\begin{align}
1-\sigma(x) &\leqslant \frac{1}{4\log(x)^{2/5}\log\log(x)^{1/5}}, \quad \textrm{ for all } \log(x) \geqslant 10000, \label{sigmaBound} \\
\sigma'(x) &\leqslant \frac{1}{10x\log(x)^{7/5}\log\log(x)^{1/5}}, \quad \textrm{ for all } \log(x) \geqslant 10000. \label{sigmaprimeBound}
\end{align}
\end{lem}
\begin{proof}
These bounds can be verified by assuming $\log(x) \geqslant 10000$ and approximating the expressions
\[ (1-\sigma(x))(\log(x)^{2/5}\log\log(x)^{1/5}) \qquad\text{and}\qquad \sigma'(x)(x\log(x)^{7/5}\log\log(x)^{1/5}).\qedhere \]
\end{proof}

\begin{cor}
For all $\log(x) \geqslant 10000$ we have the following bounds
\begin{align}
1-p_1(\sigma(x)) \geqslant 1-\frac{2}{3\log(x)^{2/5}\log\log(x)^{1/5}}, \label{p1Bound} \\
\frac{d}{dx}(1-p_1(\sigma(x))) \leqslant \frac{4}{15x\log(x)^{7/5}\log\log(x)^{1/5}}. \label{p1primeBound}
\end{align}
\end{cor}

\begin{lem}\label{lem:hathelp}
Assume $\log(x)\geqslant 10000$, that $0 \leqslant q(x) \leqslant 4$, and that $q(x)$ is decreasing in $x$. Then we have that
\begin{enumerate}
\item \begin{align*}
\frac{\frac{1}{3}\log\log\log(x)+\log(1-p_1(\sigma(x)))}{\log\log(x)}
\end{align*}
is strictly decreasing in $x$;
\item $\widetilde{u}_{1-p_1(\sigma(x)),q(x)}(x)$ is strictly increasing in $x$; and
\item we have
\begin{align*}
\widetilde{u}_{1-p_1(\sigma(x)),q(x)}(x) &> \max\{q(x)/(1-p_1(\sigma(x))),e\}, \\
\widetilde{u}_{1,q(x)}(x) &> \max\{q(x),e\}.
\end{align*}
\end{enumerate}
\end{lem}
\begin{proof}
We begin with (a) by noting that the numerator of
\begin{align*}
\frac{d}{dx}\left( \frac{\frac{1}{3}\log\log\log(x)+\log(1-p_1(\sigma(x)))}{\log\log(x)} \right)
\end{align*}
is by \eqref{p1Bound} and \eqref{p1primeBound} bounded from above by
\begin{align*}
\frac{1}{3}-\frac{1}{3}\log\log\log(x) + \frac{12\log\log(x)}{45\log(x)^{2/5}\log\log(x)^{1/5}-30} \\
-\log\biggl( 1 - \frac{2}{3\log(x)^{2/5}\log\log(x)^{1/5}} \biggr)
\end{align*}
and this is negative for all $\log(x) \geqslant 10000$ while the denominator is $\log\log(x)^2$, which is always positive.

For (b) we note that $\widetilde{u}_{1-p_1(\sigma(x)),q(x)}=(1-p_1(\sigma(x)))^{-3/5}\widetilde{u}_{1,q(x)}$ and that
\begin{align*}
\frac{d}{dx}(\widetilde{u}_{1,q(x)}(x)) = \frac{3 \, \left(\frac{5}{3}\right)^{\frac{1}{5}} \left(\frac{2000}{161967}\right)^{\frac{3}{5}}}{5 \, x \log\left(x\right)^{\frac{2}{5}} \log\left(\log\left(x\right)\right)^{\frac{1}{5}}} - \frac{\left(\frac{5}{3}\right)^{\frac{1}{5}} \left(\frac{2000}{161967}\right)^{\frac{3}{5}}}{5 \, x \log\left(x\right)^{\frac{2}{5}} \log\left(\log\left(x\right)\right)^{\frac{6}{5}}}, 
\end{align*}
which further implies that $\frac{d}{dx}((1-p_1(\sigma(x)))^{-3/5}\widetilde{u}_{1,q(x)})$ is bounded according to \eqref{p1Bound} and \eqref{p1primeBound} from below by
\begin{align*}
&\frac{3 \, \left(\frac{5}{3}\right)^{\frac{1}{5}} \left(\frac{2000}{161967}\right)^{\frac{3}{5}}}{5 \, x \log\left(x\right)^{\frac{2}{5}} \log\left(\log\left(x\right)\right)^{\frac{1}{5}}} - \frac{\left(\frac{5}{3}\right)^{\frac{1}{5}} \left(\frac{2000}{161967}\right)^{\frac{3}{5}}}{5 \, x \log\left(x\right)^{\frac{2}{5}} \log\left(\log\left(x\right)\right)^{\frac{6}{5}}} \\ &-\frac{4 \, \left(\frac{5}{3}\right)^{\frac{1}{5}} \left(\frac{2000}{161967}\right)^{\frac{3}{5}}}{25 \, x {\left(-\frac{2}{3 \, \log\left(x\right)^{\frac{2}{5}} \log\left(\log\left(x\right)\right)^{\frac{1}{5}}} + 1\right)}^{\frac{8}{5}} \log\left(x\right)^{\frac{4}{5}} \log\left(\log\left(x\right)\right)^{\frac{2}{5}}}
\end{align*}
and this is positive for $\log(x) \geqslant 10000$.

Finally, (c) can be verified by a direct computation as we make use of the fact that $\widetilde{u}_{1-p_1(\sigma(x)),q(x)}(x)$ and $\widetilde{u}_{1,q(x)}(x)$ are both strictly increasing and that $q(x)/n < 4.05$ under our assumptions.
\end{proof}
\begin{lem}\label{H-increase}
Assume that $a=-2/3$, $b=-1/3$, $c=1/53.989$, either $n(x)=1-p_1(\sigma(x))$ or $n(x)=1$ and $q=q(x)$ where $q(x)$ is decreasing and $0 \leqslant q(x) \leqslant 4$. We have that $\mathcal{H}(x,d)$ is strictly increasing when $\log(x)\geqslant 10000$ and $d$ is a constant satisfying $ 0.1 \leqslant d \leqslant 1$.
\end{lem}
\begin{proof}
We note that
\begin{align*}
\mathcal{H}(x,d)&:= \left( 1+ \frac{\log\left(\frac{2000}{161967} \, \left(\frac{5}{3}\right)^{\frac{1}{3}}\right)}{\log\log(x)}-\frac{\log(n(x))}{\log\log(x)} -\frac{\log\log\log(x)}{3\log\log(x)}-\frac{\log(d)}{\log\log(x)}\right)^{1/3} \\ &\times\left( 1+\frac{1}{2\log(\widetilde{u}_{n(x),q(x)}(x)d^{-1/(1-a)})} \right)^{-1} \\ &\times\left( 1 - \frac{q(x)}{n(x)\widetilde{u}_{n(x),q(x)}(x)d^{-1/(1-a)}} \right).
\end{align*}
and apply Lemma \ref{lem:hathelp}: (a) in the first row and (b) in the second and third row.
\end{proof}

\begin{lem}\label{uhat-increase}
Assume that $a=-2/3$, $b=-1/3$, $c=1/53.989$, either $n(x)=1-p_1(\sigma(x))$ or $n(x)=1$ and $q=q(x)$ where $q(x)$ is decreasing in $x$ and $0 \leqslant q(x) \leqslant 4$. Then we have that $\widehat{u}_{n(x),q(x)}(x)$ is increasing and $\widehat{u}_{n(x),q(x)}(x) < 1$ for all $\log(x) \geqslant 10000$.
\end{lem}
\begin{proof}
The condition $\widehat{u}_{n(x),q(x)}(x) < 1$ follows from Proposition \ref{prop:Bounds} as we note that all the required conditions hold. We also remark that $\widehat{u}_{n(x),q(x)} > 0.1$, since attempt to compute $\mathcal{H}(x,u)$ with $u \leqslant 0.1$ would give something larger than $0.72$. For the case that $u_{n(x),q(x)}(x) > \max\{q(x)/n(x),e\}$ we remark that since $\widetilde{u}_{n(x),q(x)}(x) > 12$ we would, according to \eqref{def:u2}, need at least $\widehat{u}_{n(x),q(x)}(x) > 6$ to disprove the bound, but as in the proof of Proposition \ref{prop:Bounds} was shown, attempt to compute $\mathcal{H}(x,\widehat{u}_{n(x),q(x)}(x))$ in this case would give something less than $1$, so $\widehat{u}_{n(x),q(x)}(x)$ can not be greater or equal to $1$.

The increasing property of $\widehat{u}_{n(x),q(x)}(x)$ follows from Lemma \ref{H-increase} as we remark that for any constant $d$ with $0.1 \leqslant d \leqslant 1$ we have that $\mathcal{H}(x,d)$ is strictly increasing and that according to Proposition \ref{prop:limit1} and what we just showed, $\widehat{u}_{n(x),q(x)}(x)$ has to be increasing somewhere. Since $\widehat{u}_{n(x),q(x)}(x)$ is continuous, if $\widehat{u}_{n(x),q(x)}(x)$ is somewhere decreasing before it will increase there exists constants $d_1$, $x_1$ and $x_2$ with $0.1 < d_1 < 1$ and $x_1 < x_2$ so that $d_1 = \widehat{u}_{n(x_1),q(x_1)}(x_1)=\widehat{u}_{n(x_2),q(x_2)}(x_2)$. However, this would according to \eqref{def:hatu} also imply that $\mathcal{H}(x_1,d_1) = \mathcal{H}(x_2,d_1)$, which would contradict Lemma \ref{H-increase}.
\end{proof}

\begin{lem}\label{hat-lemma}
Assume that $a=-2/3$, $b=-1/3$, $c=1/53.989$, either $n(x)=1-p_1(\sigma(x))$ or $n(x)=1$ and $q=q(x)$ where $q(x)$ is decreasing in $x$ and $0 \leqslant q(x) \leqslant 4$. Let $I \subseteq [\exp(10000),\infty)$ be an interval and let us assume that for constants $w_L$ and $w_U$ we have that $0.1 \leqslant w_L < \widehat{u}_{n(x),q(x)}(x) < w_U \leqslant 1$ holds for all $x \in I$. Then
\begin{align*}
 \mathcal{H}(x,w_{U}) &< \widehat{u}_{n(x),q(x)}(x) < \mathcal{H}(x,w_{L})
\end{align*}
holds for all $x \in I$.
\end{lem}
\begin{proof}
First we define
\begin{align*}
y(x,u) &:= \mathcal{H}_{1}(x,u)^{-1/b} \\
&= 1+ \frac{1}{3 \, \log\left(\log\left(x\right)\right)}\Bigl(10 \, \log\left(5\right) - 4 \, \log\left(3\right) + 3 \, \log\left(\tfrac{16}{53989}\right) \\
&\qquad\qquad - 3 \, \log\left(n(x)\right) - 3 \, \log\left(u\right) - \log\left(\log\left(\log\left(x\right)\right)\right)\Bigr).
\end{align*}
We then observe that with the given parameters we have
\begin{align}
\mathcal{H}(x,u) = \frac{6\log\log(x)y(x,u)^{4/3}}{5+6\log\log(x)y(x,u)}\biggl(1-\frac{q(x)}{n(x)}\exp\biggl(\frac{-3\log\log(x)y(x,u)}{5}\biggr) \biggr). \label{H-inc}
\end{align}
Now, for fixed $x$, we observe that $y(x,u)$ is decreasing in $u$, which implies that $\mathcal{H}(x,u)$ is also decreasing in $u$. By definition we have that $\mathcal{H}(x,\widehat{u}_{n(x),q(x)}(x))=\widehat{u}_{n(x),q(x)}(x)$, which implies that for fixed $x$ we have $\mathcal{H}(x,w_{U}) < \widehat{u}_{n(x),q(x)}(x) < \mathcal{H}(x,w_{L})$. In order to extend this uniformly for all $x \in I$ we recall that $\widehat{u}_{n(x),q(x)}(x)$ is continuous and according to Lemma \ref{uhat-increase} strictly increasing and according to Lemma \ref{H-increase} the same also applies for both $\mathcal{H}(x,w_{L})$ and $\mathcal{H}(x,w_{U})$. Hence, according to \eqref{def:hatu} the only possible point where $\mathcal{H}(x,w_{U}) = \widehat{u}_{n(x),q(x)}(x)$ is at the point of $x$ where $\widehat{u}_{n(x),q(x)}(x) = w_U$, and according to our assumption, such $x$ does not exist in $I$. Same argument also applies for $w_L$ and the crossing point of $\widehat{u}_{n(x),q(x)}(x)$ and $\mathcal{H}(x,w_L)$.
\end{proof}
\begin{Rem}\label{rem:hat-appro}
Assume that $a=-2/3$, $b=-1/3$, $c=1/53.989$, either $n(x)=1-p_1(\sigma(x))$ or $n(x)=1$ and $q=q(x)$ where $q(x)$ is decreasing in $x$ and $0 \leqslant q(x) \leqslant 4$. Let $I \subseteq [\exp(10000),\infty)$ be a finite closed interval $[x_l,x_u]$ and $J = [x_l',\infty)$ with $x_l' \geqslant \exp(10000)$ be an interval without an upper bound. We choose $j \in \mathbb{N}$ to describe the number of iterations. For $x \in I$ we approximate $\widehat{u}_{n(x),q(x)}(x)$ in the following way. For every $k =1,2,3,\dots,j$ we compute $\mathcal{H}^{k}(x_l,1)$ and $\mathcal{H}^{k}(x_u,1)$ (see notation in Proposition \ref{Prop:iteration}) and observe that if $k$ is odd, then the obtained number is a lower bound of $\widehat{u}_{n(x),q(x)}$ at the given point of $x$ and if $k$ is even, then the bound is an upper bound respectively (Lemmas \ref{uhat-increase} and \ref{hat-lemma}).

We assume that in the spirit of Proposition \ref{prop:iteration2} every iteration will improve the bound. This would imply that the last bound obtained is better than any of the previous bounds, but for the sake of this application, any computed bound is in any case valid, so we do not include the proof of this.

Now let $w_L$ be the lower bound constant of $\widehat{u}_{n(x_l),q(x_l)}(x_l)$ and $w_U$ be the upper bound constant of $\widehat{u}_{n(x_u),q(x_u)}(x_u)$. We will then declare that for all $x \in I$ we have that
\begin{align*}
\mathcal{H}(x,w_{U}) &< \widehat{u}_{n(x),q(x)}(x) < \mathcal{H}(x,w_{L}).
\end{align*}
For $x \in J$ the argument goes similarly for the lower bound $w_L'$ at $x=x_l'$, but we use $1$ as the upper bound constant and declare that for all $x \in J$ we have that
\begin{align*}
\mathcal{H}(x,1) &< \widehat{u}_{n(x),q(x)}(x) < \mathcal{H}(x,w_{L}')<1.
\end{align*}
\end{Rem}

Now we have everything we need for doing proper approximations. First we remark that with these choices of parameters the main factor
\begin{align}
\exp\left( -m(1-a^{-1})(m^{-1}(1-a)^{-b}(-a)c)^{1/(1-a)} \log(x)^{1/(1-a)}\log\log(x)^{b/(1-a)} \right). \label{MainConstantFormula}
\end{align}
is exactly
\begin{align*}
\exp\biggl(-\frac{5}{2} \, \left(\frac{5}{3}\right)^{\frac{1}{5}} \left(\frac{2000}{161967}\right)^{\frac{3}{5}}\log(x)^{3/5}\log\log(x)^{-1/5}\biggr).
\end{align*}
We thus define
\begin{equation}
    D := \frac{5}{2} \, \left(\frac{5}{3}\right)^{\frac{1}{5}} \left(\frac{2000}{161967}\right)^{\frac{3}{5}} \approx 0.1982767.
    \label{eq:D2}
\end{equation}

We will now move to estimating the bounds of $E_{11}(x), E_{12}(x)$ and $E_{13}(x)$.

\begin{defn}
We define
\begin{align*}
E_{1}(x) := \frac{E_{11}(x)+E_{12}(x)+E_{13}(x)}{\exp(-D\log(x)^{3/5}\log\log(x)^{-1/5})}
\end{align*}
and
\begin{align*}
E_{2}(x) := \frac{E_{21}(x)+E_{22}(x)+E_{23}(x)}{\exp(-D\log(x)^{3/5}\log\log(x)^{-1/5})}.
\end{align*}
\end{defn}

\begin{prop}
With notation as above. For $10000\leqslant \log(x) \leqslant 10^{15}$ the functions $E_1(x)$ and $E_2(x)$ can be bounded by a non-increasing step function. 
An example of such bounds is given by Table \ref{tableBoundE1E2}.
\end{prop}
\begin{proof}
This is checked computationally using interval arithmetic, by subdividing $[10^4,10^{15}]$ into sub-intervals.

Let $[x_k,x_{k+1}]$ denote a single subinterval of our chosen division of $[10^4,10^{15}]$, which is a subdivision in the variable $\log(x)$. In order to bound $E_{11}(x), E_{12}(x)$ and $E_{13}(x)$ from above by assuming $\log(x) \in [x_k,x_{k+1}]$ we have to approximate the respective $\widehat{u}$ for each of these bounds. We do it as described in Remark \ref{rem:hat-appro} with a note that for $\Phi_{4,\eta}(x)$ evaluations from Proposition \ref{prop:Cvalue} we use lower and upper bound constants of $\widehat{u}$. We give an example of how to do the bound with $E_{i1}(x)$:
\begin{align*}
E_{i1}(x) &\leqslant \exp(-m(1-a^{-1})\widetilde{u}_{m,0}(x))\\
 &\quad\times \exp\biggl( -m(1-a^{-1})\biggl(\biggl(1 \\
&\qquad + \frac{b}{a(a-1)\log(\widetilde{u}_{m,0}(x)\mathcal{H}(x,w_{L})^{-1/(1-a)}) +b(a-1)} \\&\qquad -\frac{p_i(\sigma(x))}{m(1-a^{-1})}\biggr)\left(\frac{1}{\mathcal{H}(x,w_{L})}\right)^{1/(1-a)} -1\biggr)\widetilde{u}_{m,0}(x) \\ &\qquad +q_i(\sigma(x))\log(\widetilde{u}_{m,0}(x)\mathcal{H}(x,w_{U})^{-1/(1-a)})\biggr).
\end{align*}
We then compute the upper bound for each interval using \texttt{RealIntervalField} and create a table of values that decrease in $x$. More details can be found in the related \texttt{Sage} code. The table of bounds for $E_1(x)$ and $E_2(x)$ is presented in Table \ref{tableBoundE1E2}.

\end{proof}

Next, we will show that given bound in Table \ref{tableBoundE1E2} at $\log(x) = 10^{15}$ is valid for all $\log(x) \geqslant 10^{15}$. This is done by generating a more trivial bound for which the decreasing property is easier to justify. One can use the obtained upper bound function for all $\log(x) \geqslant 10^{15}$, but more efficient method is to extended the list in Table \ref{tableBoundE1E2} by the method we generated its values initially.
\begin{lem}\label{hatcomparison}
We have that
\begin{align*}
\left(\frac{1}{(1-p_i(\sigma(x))/m)\widehat{u}_{m-p_i(\sigma(x)),q_i(\sigma(x))}(x)}\right)^{1/(1-a)} > \left(\frac{1}{\widehat{u}_{m,0}(x)}\right)^{1/(1-a)}.
\end{align*}
\end{lem}
\begin{proof}
This follows directly from  Propositions \ref{prop:fx} and \ref{comparisonProp}.
\end{proof}
\begin{defn}
We define
\begin{align*}
G(x) &:= \exp\biggl( -\frac{5}{2}\biggl(\left(1 - \frac{3}{10\log(\widetilde{u}_{1,0}(x)) +5}-p_1(\sigma(x))\right)\\&\quad\times\left(\frac{1}{\widehat{u}_{1,0}(x)}\right)^{1/(1-a)}-1\biggr)\widetilde{u}_{1,0}(x)+ \frac{3}{2}q_1(\sigma(x))\log(\widetilde{u}_{1,0}(x))\biggr).
\end{align*}
\end{defn}
\begin{Rem}\label{rem:constantonG}
We note that by Proposition \ref{comparisonProp} and Lemma \ref{hatcomparison} and the known parameters we have that both components $E_{1}(x)$ and $E_{2}(x)$ have an upper bound that can be expressed as $C_{i,x_0}G(x)$ for all $x \geqslant x_0$ where $C_{i,x_0} = \frac{E_{i}(x_0)}{G(x_0)}$ is a unique constant for $E_{i}(x)$ as the $G(x)$ was defined in a way that it will bound all the increasing parts of the given components from above.
\end{Rem}
Next, we will construct a way to estimate the structure of $G(x)$.
\begin{Rem}\label{RemBase2}
Recalling Propositions \ref{prop:iteration2}, \ref{Prop:iteration}, and \ref{prop:Bounds} and Lemma \ref{hat-lemma} we note that $\widehat{u}_{1,0}(x) \geqslant \mathcal{H}^{5}(\exp(10^{15}),1) \geqslant \frac{14367722903487}{15625000000000} =0.919534265823168 $ for all $x \geqslant \exp(10^{15})$ and 
\begin{align*}
\mathcal{H}_1\biggl(x,\frac{14367722903487}{15625000000000}\biggr) &= \left(1+\mathcal{A}_1(x)\right)^{1/3}, \\
\mathcal{H}_2\biggl(x,\frac{14367722903487}{15625000000000}\biggr) &= \left(1+\mathcal{A}_2(x)\right)^{-1},
\end{align*}
where
\begin{align*} \mathcal{A}_1(x) = -\frac{3 \, \log\left(\frac{14367722903487}{15625000000000}\right) - 3 \, \log\left(\frac{2000}{161967} \, \left(\frac{5}{3}\right)^{\frac{1}{3}}\right) + \log\left(\log\left(\log\left(x\right)\right)\right)}{3 \, \log\left(\log\left(x\right)\right)} \end{align*}
and
\begin{align*}
\mathcal{A}_2(x) = \frac{1}{2 \, \log\left(\frac{250000000 \, \left(\frac{14367722903487}{16}\right)^{\frac{2}{5}} \left(\frac{5}{3}\right)^{\frac{1}{5}} \left(\frac{2000}{161967}\right)^{\frac{3}{5}} \log\left(x\right)^{\frac{3}{5}}}{14367722903487 \, \log\left(\log\left(x\right)\right)^{\frac{1}{5}}}\right)}
\end{align*}
\end{Rem}
\begin{defn}\label{DefH1}
We define
\begin{align*}
H_1(x) = S_1(x)S_2(x), 
\end{align*}
where by setting $\mathcal{B}_1 = \frac{b}{1-a}=-1/5$ we have
\begin{align*}
S_1(x) &= 1+\mathcal{B}_1\mathcal{A}_1(x)+\frac{\mathcal{B}_1(\mathcal{B}_1-1)}{2!}\mathcal{A}_1(x)^2 \\&\quad+\frac{\mathcal{B}_1(\mathcal{B}_1-1)(\mathcal{B}_1-2)}{3!}\mathcal{A}_1(x)^3+\frac{\mathcal{B}_1(\mathcal{B}_1-1)(\mathcal{B}_1-2)(\mathcal{B}_1-3)}{4!}\mathcal{A}_1(x)^4,
\end{align*}
and by setting $\mathcal{B}_2 = \frac{1}{1-a}=3/5$ we have
\begin{align*}
S_2(x) = 1+\mathcal{B}_2\mathcal{A}_2(x)+\frac{\mathcal{B}_2(\mathcal{B}_2-1)}{2!}\mathcal{A}_2(x)^2.
\end{align*}
\end{defn}
\begin{Rem}\label{RemBase}
According to Remark \ref{RemBase2} and the known series expansion in Definition \ref{DefH1} we have that
\begin{align*}
H_1(x) \leqslant \left(\frac{1}{\widehat{u}_{1,0}(x)} \right)^{3/5}
\end{align*}
for all $\log(x) \geqslant 10^{15}$.
\end{Rem}
\begin{Rem}\label{defbase}
Let
\begin{align*}
\widetilde{G}(x) &:= -\frac{5}{2}\biggl(\left(1 - \frac{3}{10\log(\widetilde{u}_{1,0}(x)) +5}-p_1(\sigma(x))\right)H_1(x)-1\biggr)\widetilde{u}_{1,0}(x) \\
&\qquad + \frac{3}{2}q_1(\sigma(x))\log(\widetilde{u}_{1,0}(x)),
\end{align*}
which according to Remark \ref{RemBase} implies that
\begin{align*}
G(x) \leqslant \exp(\widetilde{G}(x)).
\end{align*}
\end{Rem}
\begin{Rem}
We have that
\begin{align*}
\lim_{x\to\infty} \frac{\widetilde{G}(x)}{\frac{\log\left(x\right)^{\frac{3}{5}} \log\left(\log\left(\log\left(x\right)\right)\right)}{\log\left(\log\left(x\right)\right)^{\frac{6}{5}}}} = -\frac{50}{485901} \cdot 4153^{\frac{2}{5}} 13^{\frac{2}{5}} 3^{\frac{1}{5}} 2^{\frac{2}{5}} = -\frac{D}{15}.
\end{align*}
\end{Rem}
\begin{prop}\label{G1decrease}
We have that for all $\log(x) \geqslant 10^{15}$
\begin{align}
\widetilde{G}(x) + \frac{D\log\left(x\right)^{\frac{3}{5}} \log\left(\log\left(\log\left(x\right)\right)\right)}{15\log\left(\log\left(x\right)\right)^{\frac{6}{5}}} \label{decrease-expr}
\end{align}
is always negative and decreasing.
\end{prop}
\begin{proof}
The strategy to prove the decreasing property by assuming $\log(x) \ge 10^{15}$ is the following. First we check that the negativity holds at $\log(x) = 10^{15}$. Let $g(x)$ be the derivative of \eqref{decrease-expr}. In \texttt{Sage} by taking the numerator of $g(x)$ we first observe that the sign of the numerator coincides with the sign of the derivative and that all the 4454 terms of the numerator consist of the following types of possible factors:
\begin{enumerate}
\item constant factor;
\item $\log(x)^{A_1}$ with $A_1 \in [1,13/5]$ or $A_1=0$;
\item $\log\log(x)^{B_1}$ with $B_1 \in [17/5, 134/15]$ or $B_1=0$;
\item the quantity
\begin{align*}
\biggl(&2 \, \log\left(14367722903487\right) - 3 \, \log\left(161967\right) + 3 \, \log\left(2000\right) \\
&+ 46 \, \log\left(5\right) - 11 \, \log\left(3\right) + 12 \, \log\left(2\right) + 5 \, \log\left(\tfrac{8}{1596413655943}\right) \\
&+ 3 \, \log\left(\log\left(x\right)\right) - \log\left(\log\left(\log\left(x\right)\right)\right)\biggr)^{B_2}
\end{align*}
with $B_2 \in [1, 3]$ or $B_2=0$;
\item $\biggl(-\frac{3}{5} \, \log\left(161967\right) + \frac{3}{5} \, \log\left(2000\right) + \frac{1}{5} \, \log\left(5\right) - \frac{1}{5} \, \log\left(3\right) + \frac{3}{5} \, \log\left(\log\left(x\right)\right) - \frac{1}{5} \, \log\left(\log\left(\log\left(x\right)\right)\right) \biggr)^{B_3}$ with $B_3 \in [1/3, 17/3]$ or $B_3=0$; and
\item $\log\log\log(x)^{C_1}$ with $C_1 \in [1, 4]$ or $C_1=0$.
\end{enumerate}
This implies that every term can be asymptotically expressed as 
\begin{align*} O\biggl(\log(x)^{A_1}\log\log(x)^{B_1+B_2+B_3}\log\log\log(x)^{C_1}\biggr). \end{align*}
For $\log(x) \geqslant 10^{15}$ terms in any of the mentioned categories can not change signs as $x$ grows. First we recognize that on top of the asymptotically ordered list there appears to be 4 terms that can be combined into 2 terms that will due to cancellations push them down in the list. It turns out that the surviving asymptotically dominant factor gives negative contribution (the value is always negative) and it has the following characteristic values $A_1 = 13/5, B_1+B_2+B_3 = 214/15$ and $C_1 = 0$. This term  also has a positively contributing counterpart with the same characteristic values but their combination is a negatively contributing term with the same characteristic values. For $\log(x) \geqslant 10^{10000}$ the sum of the dominant term and all the terms with positive contribution is always negative, which gives the decreasing property for all $\log(x) \geqslant 10^{10000}$. 

For values $10^{15} \leqslant \log(x) \leqslant 10^{10000}$ we observe that all the negative contribution terms of the numerator are decreasing and the positive contribution terms are increasing. Now based on computational verifications, there exists a division for the interval $[10^{15},10^{10000}]$ so that for each subinterval $\Delta_k \in [x_k,x_{k+1}]$ the sum of the negative contribution terms at $x=x_k$ and the sum of the positive contribution terms at $x=x_{k+1}$ combined is negative, which implies that the numerator is negative for all $x$ in the interval. It is worth mentioning that the maximum lengths of the subintervals grows quickly. A good strategy is to increase the length of the next subinterval by giving it a length which is the length of the current subinterval multiplied by some $10^d$ with positive $d$. Examples of suitable choices for $d$ for $\log(x)$ in a certain interval are given in the following table.
\begin{center}
\begin{tabular}{ |l|l| } 
\hline
 $\log(x)$ & $d$ \\
\hline
$[10^{15},10^{65}]$&$0.5$ \\
$[10^{65},10^{565}]$&$5$ \\
$[10^{565},10^{5565}]$&$50$\\
$[10^{5565},10^{10000}]$&$500$\\
\hline
\end{tabular}
\end{center}
\end{proof}
\begin{prop}\label{Prop:E1E2}
The upper bounds for $E_1(x)$ and $E_2(x)$ for $\log(x) \geqslant 10000$ listed in Table \ref{tableBoundE1E2} hold and we also have that for all $\log(x)\geqslant 10^{15}$
\begin{align}
E_{1}(x) &\leqslant 6.4\cdot 10^{-1156217}\exp\left(-D\frac{\log\left(x\right)^{\frac{3}{5}} \log\left(\log\left(\log\left(x\right)\right)\right)}{15\log\left(\log\left(x\right)\right)^{\frac{6}{5}}}\right), \label{E1Gbound}\\
E_{2}(x) &\leqslant 2.1\cdot 10^{-1156229}\exp\left(-D\frac{\log\left(x\right)^{\frac{3}{5}} \log\left(\log\left(\log\left(x\right)\right)\right)}{15\log\left(\log\left(x\right)\right)^{\frac{6}{5}}}\right). \label{E2Gbound}
\end{align}
\end{prop}
\begin{proof}
Bounds \ref{E1Gbound} and \ref{E2Gbound} is derived using Remarks \ref{rem:constantonG} and \ref{defbase} and Proposition \ref{G1decrease} and by assuming $\log(x) \geqslant 10^{15}$. We further remark that the last bound in Table \ref{tableBoundE1E2} is computed using interval assuming $\log(x) \in [10^{15},10^{15}+10^{13}]$. We then compute, according to \eqref{E1Gbound}, \eqref{E2Gbound}, Remarks \ref{rem:constantonG} and \ref{defbase} and Proposition \ref{G1decrease}, that for $\log(x) \geqslant 10^{15}+10^{13}$ we have
\begin{align*}
E_{1}(x) &\leqslant 1.8 \cdot 10^{-1447780}, \\
E_{2}(x) &\leqslant 5.7 \cdot 10^{-1447793},
\end{align*}
and both of these bounds validate the last bound in Table \ref{tableBoundE1E2}.
\end{proof}
\begin{defn}
We define
\begin{align*}
\mathcal{E}_1(x) := E_{1}(x)\exp\left(D\frac{\log\left(x\right)^{\frac{3}{5}} \log\left(\log\left(\log\left(x\right)\right)\right)}{15\log\left(\log\left(x\right)\right)^{\frac{6}{5}}}\right), \\
\mathcal{E}_2(x) := E_{2}(x)\exp\left(D\frac{\log\left(x\right)^{\frac{3}{5}} \log\left(\log\left(\log\left(x\right)\right)\right)}{15\log\left(\log\left(x\right)\right)^{\frac{6}{5}}}\right).
\end{align*}
\end{defn}
\begin{prop}\label{actualE1andE2bounds}
The upper bounds for $\mathcal{E}_1(x)$ and $\mathcal{E}_2(x)$ for $\log(x) \geqslant 10000$ listed in Tables \ref{tableBound} and \ref{tableBoundLarge} hold.
\end{prop}
\begin{proof}
The claim is essentially a corollary of Proposition \ref{Prop:E1E2}.
\end{proof}

\section{The Prime Number Theorem}\label{sec:pnt}

We now illustrate the application of our bounds on zero sums to obtain an explicit form of the prime number theorem.

We shall be assuming here the zero-density result of \eqref{eq:golnoush}, valid for $\sigma\in[0.75,1]$ and $T\ge 3\cdot 10^{12}$:
\[ N(\sigma,T) \leqslant \mathcal{U}T^{\frac{8}{3}(1-\sigma)}\log(T)^{3+2(1-\sigma)} + \mathcal{V}\log(T)^2 \]
and a zero-free region $1-v(u)$ where
\begin{align*}
v(u):=v_{a,b,c}(u)=cu^a\log(u)^b,
\end{align*}
with $a=-2/3$, $b=-1/3$ and $c=1/53.989$ \cite{BellottiZFR}.

One key tool shall be the explicit Perron's formula of \cite[Theorem 1.2]{CHJ-Perron}. 
\begin{thm}
For any $\alpha\in(0,1/2]$ and $\omega\in[0,1]$ there exist $M_{\omega,\alpha}$ and $x_{\omega,\alpha}$ such that for all $T$ with $\max\{102,2\log(x)^2\} < T/2 < (x^\alpha-2)/4$ we have  
 \[  \frac{|\psi(x) - x|}{x} < \sum_{0<|\gamma| \leqslant T} \frac{x^{\beta-1}}{|\gamma|} + 2M_{\omega,\alpha}\frac{\log(x)^{1-\omega}}{T}. \]
This holds for $x>x_{\omega,\alpha}$.
\end{thm}
\begin{proof}
This is immediate from \cite[Theorem 1.2]{CHJ-Perron}. 
We have simply taken absolute values and eliminated $T^\ast$ by taking the worst case in each bound.
\end{proof}
Using the values from \cite[Table 1]{CHJ-Perron} we take
\[ \alpha=1/100,\quad \omega=0.8,\quad M_{\omega,\alpha}=3.615,\quad \log(x_{\omega,\alpha}) = 10^{3}. \]

As before we set
\[ T(x) = t_1(x)^{1+\epsilon}
\]
with $\epsilon=2$, where $t_1(x)$ is as defined in Definition \ref{defn:um}.
We shall only consider $\log(x_0) \ge 10000$.
With these values we have
\[ \max\{102,2\log(x)^2\} < T(x)/2 < (x^\alpha-2)/4 \]
for $\log(x)>\log(x_0) \ge 10000$.

We may thus write
\begin{equation}\label{boundcombination}
\begin{split}
 \frac{|\psi(x)-x|}{x} < \
  &2\sum_{ \substack{0 < \beta \leqslant \sigma_0\\0<\gamma \leqslant T(x)} } \frac{x^{\beta-1}}{\gamma}
  +
  2\sum_{ \substack{\sigma_0 < \beta \leqslant \sigma_1(x)\\0<\gamma \leqslant T(x)} } \frac{x^{\beta-1}}{\gamma} \\
  &+ 2\sum_{ \substack{\sigma_1(x) < \beta \\0<\gamma \leqslant T(x)} } \frac{x^{\beta-1}}{\gamma}
  + 2M_{\omega,\alpha} \frac{\log(x)^{1-\omega}}{T(x)}
\end{split}
\end{equation}
for $\log(x)>\log(x_0)\ge 10000$ where we take
\begin{align*}
    \sigma_0 &= 0.75\\
    \sigma_1(x) &= 1+\frac{-\widetilde{u}_{1,0}(x)-\log(x)v(\widetilde{u}_{1,0}(x))+4\log(\widetilde{u}_{1,0}(x))}{\log(x)}. 
\end{align*}
 We can verify that with $\log(x)>\log(x_0) \ge 10000$ we shall have $\sigma_1(x) > 0.997 > 0.99$.

We may now bound 
\[     2M_{\omega,\alpha} \frac{\log(x)^{1-\omega}}{T(x)} \leqslant  7.23 \frac{\log(x)^{1/5} }{t_1(x)^{3}}
  \]
using the above.

Next, one may verify that 
\[ \mathcal{E}_{\rm per}(x) =  7.23 \frac{\log(x)^{1/5}\exp\left(D\left(1+\frac{\log\log\log(x)}{15\log\log(x)}\right)\log(x)^{3/5}\log\log(x)^{-1/5}\right) }{t_1(x)^{3}} 
\] 
is decreasing in $x$ for $\log(x) > 10000$ and hence we obtain
\begin{align*}
2M_{\omega,\alpha} \frac{\log(x)^{1-\omega}}{T(x)}
&\leqslant \mathcal{E}_{\rm per}(x_0) \\
&\quad\times\exp\left(-D\left(1+\frac{\log\log\log(x)}{15\log\log(x)}\right)\log(x)^{3/5}\log\log(x)^{-1/5}\right)
\end{align*}
We list upper bounds for $\mathcal{E}_{\rm per}(x)$ in Tables \ref{tableBound} and \ref{tableBoundLarge}. We also note that for $t_1(x)^3$ we apply the trivial bound $\widehat{u}_{1,0}(x)<1$ uniformly. This gives relatively much larger upper bounds for $\mathcal{E}_{\rm per}(x)$ compared to the actual values but we also remark that since this term has very small overall contribution in the application, we accept these bounds as they are.

Next, using Proposition \ref{prop:lowlying} we bound
\begin{align*}  2\sum_{ \substack{0 < \beta \leqslant \sigma_0\\0<\gamma \leqslant T(x)} } \frac{x^{\beta-1}}{\gamma} 
&<
  2( 57.5308569955611)x^{-1/2} + \\
&2\biggl(\frac{1}{2\pi}\log\left(\frac{T(x)}{H_0}\right)\log\left(\frac{\sqrt{T(x)H_0}}{2\pi}\right) - \frac{Q(H_0)}{H_0} + \frac{ R_{7/8}(T(x)) }{T(x)} \\ &\quad +  \frac{2A_0 + 2A_1\log(H_0) + A_1+A_2}{H_0^2} \biggr)\left( \frac{x^{\sigma_0-1} + x^{-\sigma_0}}{2} \right).\\
&\leqslant\mathcal{E}_{1/2}(x) \exp\left(-D\left(1+\frac{\log\log\log(x)}{15\log\log(x)}\right)\log(x)^{3/5}\log\log(x)^{-1/5}\right) 
\end{align*}

where 
\begin{align}
\label{eq:E1-2}\begin{split} \mathcal{E}_{1/2}(x)  & := 
2\biggl(  57.5308569955611 x^{-1/2} \\ &\quad +
\biggl(\frac{1}{2\pi}\log\left(\frac{T(x)}{H_0}\right)\log\left(\frac{\sqrt{T(x)H_0}}{2\pi}\right) - \frac{Q(H_0)}{H_0} + \frac{ R_{7/8}(T(x)) }{T(x)} \\ &\qquad +  \frac{2A_0 + 2A_1\log(H_0) + A_1+A_2}{H_0^2} \biggr)\left( \frac{x^{\sigma_0-1} + x^{-\sigma_0}}{2} \right) \biggr) \\&\times\exp\left(D\left(1+\frac{\log\log\log(x)}{15\log\log(x)}\right)\log(x)^{3/5}\log\log(x)^{-1/5}\right).
\end{split}
\end{align}
As above one verifies that for $\log(x)>10000$ this is decreasing in $x$ and hence
\begin{align*}
 2\sum_{ \substack{0 < \beta \leqslant \sigma_0\\0<\gamma \leqslant T(x)} } \frac{x^{\beta-1}}{\gamma}
 &< \mathcal{E}_{1/2}(x_0) \\
 &\quad\times\exp\left(-D\left(1+\frac{\log\log\log(x)}{15\log\log(x)}\right)\log(x)^{3/5}\log\log(x)^{-1/5}\right).
\end{align*}
We list upper bounds for $\mathcal{E}_{1/2}(x)$ in Tables \ref{tableBound} and \ref{tableBoundLarge}.

Next, using Corollary \ref{cor:intermediate} we bound
\begin{align*}   2\sum_{ \substack{\sigma_0 < \beta \leqslant \sigma_1(x)\\0<\gamma \leqslant T(x)} } \frac{x^{\beta-1}}{\gamma}
&\leqslant 2\left(
2.282680 \cdot 10^{-8}  \mathcal{U} +
2.949343 \cdot 10^{-10}  \mathcal{V}
\right)x^{\sigma_1(x)-1}\\
& = 2\left(
2.282680 \cdot 10^{-8}  \mathcal{U} +
2.949343 \cdot 10^{-10}  \mathcal{V}
\right) \\ & \quad \times \exp(-\widetilde{u}_{1,0}(x)-\log(x)v(\widetilde{u}_{1,0}(x))+4\log(\widetilde{u}_{1,0}(x))) \\
& \leqslant 2\left(
2.282680 \cdot 10^{-8}  \mathcal{U} +
2.949343 \cdot 10^{-10}  \mathcal{V}
\right)\mathcal{E}_{\rm mid}(x) \\ & \quad \times\exp\left(-D\left(1+\frac{\log\log\log(x)}{15\log\log(x)}\right)\log(x)^{3/5}\log\log(x)^{-1/5}\right)
\end{align*}
where
\begin{align*} \mathcal{E}_{\rm mid}(x)  &= \exp(-\widetilde{u}_{1,0}(x)-\log(x)v(\widetilde{u}_{1,0}(x))+4\log(\widetilde{u}_{1,0}(x))) \\ & \quad \times 
\exp\left(D\left(1+\frac{\log\log\log(x)}{15\log\log(x)}\right)\log(x)^{3/5}\log\log(x)^{-1/5}\right).
\end{align*}
We remark that $\mathcal{E}_{\rm mid}(x)$ is decreasing if we assume $\log(x) \geqslant 15000$. The proof of this can be achieved with similar strategy as in the proof of Proposition \ref{G1decrease} and is shown in \texttt{Sage} code.

Finally, using \eqref{eq:edge} and Proposition \ref{Prop:E1E2}  we bound
\begin{align*}
 2\sum_{ \substack{\sigma_1(x) < \beta \\0<\gamma \leqslant T(x)} } \frac{x^{\beta-1}}{\gamma}
 &\leqslant   2(\mathcal{U}\mathcal{E}_{1}(x)+\mathcal{V}\mathcal{E}_{2}(x)) \\
 &\quad\times\exp\left(-D\left(1+\frac{\log\log\log(x)}{15\log\log(x)}\right)\log(x)^{3/5}\log\log(x)^{-1/5}\right).
\end{align*}
Since for each of the four terms above the coefficient of
\[ \exp(-D\log(x)^{3/5}\log\log(x)^{-1/5}) \]
is decreasing in $x$ for $\log(x)>15000$, we obtain
\begin{thm}\label{thm:assymp}
For $\log(x) > 10000$ we have
\[ \frac{|\psi(x) - x|}{x} < \mathcal{E}(x)\exp\left(-D\left(1+\frac{\log\log\log(x)}{15\log\log(x)}\right)\log(x)^{3/5}\log\log(x)^{-1/5}\right) \]
where
\begin{align*} \mathcal{E}(x) &= \mathcal{E}_{\rm per}(x) + \mathcal{E}_{1/2}(x) \\
&\quad + 2\left(
2.282680 \cdot 10^{-8}  \cdot 12.45321 +
2.949343 \cdot 10^{-10}  \cdot 3.86894
\right)\mathcal{E}_{\rm mid}(x) \\
&\quad + 2(12.45321\mathcal{E}_{1}(x)+3.86894\mathcal{E}_{2}(x)) \end{align*}
and that we can give upper bounds for $\mathcal{E}(x)$ that decrease in $x$ as given in Tables \ref{tableBound} and \ref{tableBoundLarge}.
\end{thm}
\begin{proof}
We combine the above bounds for the components in \eqref{boundcombination} and set $\mathcal{U} = 12.45321$ and $\mathcal{V} = 3.86894$ according to \eqref{eq:golnoush}. This gives the result with the caveat that for $10000 \leqslant \log(x) \leqslant 15000$ we must use an upper bound for $\mathcal{E}_{\rm mid}(x)$ valid on this whole interval.
\end{proof}

If one is willing to assume a quasi-Riemann hypothesis, the contribution of the zeros near the $1$-line disappears and the middle term is governed by Corollary~\ref{cor:intermediatelow}. For the zero density we use the maxima over the relevant rows of \cite[Table~8]{FKS1} and \eqref{eq:golnoush}.
We have
\begin{align*}
(\mathcal{U},\mathcal{V})&=(8.5373,\,4.1006)\ \text{ on } [3/4,13/16],\\
 (\mathcal{U},\mathcal{V})&=(11.0755,\,4.1006)\ \text{ on } [3/4,7/8],\text{ and}\\
 (\mathcal{U},\mathcal{V})&=(12.45321,\,3.86894)\ \text{ on } [3/4,1].
\end{align*}
\begin{thm}\label{thm:assympquasi}
Let $3/4\le\sigma\le 1$ and suppose that every zero $\rho=\beta+i\gamma$ of $\zeta$ with $\beta>0$ has $\beta\le\sigma$. Let $L_0\geq 84$. Then for $\log(x)>L_0$
\[ \frac{|\psi(x) - x|}{x} < 0.0198944\log(x)^2\,x^{-1/4} + C(\sigma,L_0)\, x^{\sigma-1}, \]
where $C(\sigma,L_0)=2\big(C_{\mathcal{U}}(\sigma,L_0)\,\mathcal{U}+2.949343\cdot10^{-10}\,\mathcal{V}\big)$.
Upper bounds for $C(\sigma,L_0)$, computed from the table of Corollary~\ref{cor:intermediatelow}, are:
\begin{center}
\begin{tabular}{|r|c|c|c|}
\hline
$L_0$ & $\sigma\in[3/4,13/16]$ & $\sigma\in[13/16,7/8]$ & $\sigma\in[7/8,1]$ \\
\hline
$100$ & $6764.518$ & $36.0125$ & $0.1912794$ \\
$120$ & $2867.29$ & $13.80316$ & $0.06839662$ \\
$140$ & $1981.837$ & $9.122821$ & $0.04365015$ \\
$150$ & $1759.13$ & $7.983365$ & $0.03775338$ \\
$157$ & $1642.925$ & $7.396193$ & $0.03474053$ \\
$191$ & $1309.112$ & $5.741562$ & $0.02636425$ \\
$200$ & $1255.613$ & $5.481296$ & $0.02506448$ \\
$300$ & $972.0567$ & $4.128316$ & $0.01840514$ \\
$500$ & $822.0193$ & $3.432807$ & $0.01505842$ \\
$1000$ & $736.3152$ & $3.042801$ & $0.01320991$ \\
$2000$ & $699.7371$ & $2.878095$ & $0.01243612$ \\
$5000$ & $679.4592$ & $2.787256$ & $0.01201123$ \\
$10^4$ & $672.9546$ & $2.75819$ & $0.01187557$ \\
$10^5$ & $667.2046$ & $2.732526$ & $0.0117559$ \\
\hline
\end{tabular}
\end{center}
\end{thm}
\begin{proof}
We use \eqref{boundcombination} with $T(x)=\frac{1}{4}x^{1/2}$ and $\sigma_1(x)=\sigma$, and the explicit Perron formula above with $\alpha=1/2$, $\omega=0$, $M_{\omega,\alpha}=2.894$, $\log(x_{\omega,\alpha})=40$ from \cite[Table~1]{CHJ-Perron}. The first sum is bounded by Proposition~\ref{prop:lowlying} with $\sigma=3/4$ and we obtain the bound $\log(x)^2/(16\pi)+10^{-10}$. The Perron term is bounded by $23.152\log(x)x^{-1/2}$. These combine to give at most $0.0198944\log(x)^2x^{-1/4}$ for $\log(x)\geq 84$. The remaining sum over zeros is bounded by Corollary~\ref{cor:intermediatelow} with $\sigma_1=\sigma$.
\end{proof}

\subsection{Proof of Theorem \ref{OurMainTheorem}}\label{sec:prooftheorem1}

We shall define $L(x_0)$ be the supremum over $x\geqslant x_0$ of the best of the following bounds which is admissible at $x$.

For $x\geqslant \exp(10^8)$ we may apply Theorem \ref{thm:assymp} with $\log(x_0)=10^8$.

For $3 \leqslant \log(x) < 4$ we note that the bound is trivial and for $4 \leqslant \log(x) \leqslant 43$ we apply the bound from \cite{Buthe}, which states that
\begin{align*}
\frac{|\psi(x)-x|}{x} &\leqslant 0.94x^{-1/2}, 11 < x \leqslant 10^{19}
\end{align*}
and for $43 \leqslant \log(x) \leqslant 1300$ our bound will come from Table 2 in \cite{FKS2} by noting that 
\begin{align*}
\frac{|\psi(x)-x|}{x} &\leqslant 1.9220 \cdot 10^{-8}, 43 \leqslant \log(x) 
\end{align*}

For $1300\leqslant \log(x)$ we may apply \eqref{eq:fksresult}, i.e.\ the main theorem of \cite{FKS1}. The bound in Theorem \ref{thm:assymp} becomes stronger than \eqref{eq:fksresult} for $\log(x)$ approximately $1.238\cdot 10^{8}$. Thus beyond this point we will simply take $L(x_0) = \mathcal{E}(x)$ by Theorem \ref{thm:assymp}.

With $L(x_0)$ as defined above it will be a decreasing function which satisfies \begin{align*}
\frac{|\psi(x)-x|}{x} \leqslant L(x)\exp\left(-D\left(1+\frac{\log\log\log(x)}{15\log\log(x)}\right)\log(x)^{3/5}\log\log(x)^{-1/5}\right).
\end{align*} This completes the proof.

\subsection{Proof of Theorem \ref{OurQuasiTheorem}}\label{sec:prooftheoremquasi}

Let $L_1=140$, $157$, $191$ according as $\sigma$ lies in $[3/4,13/16]$, $[13/16,7/8]$, $[7/8,1]$.
For $x\le 10^{19}$ we apply \cite{Buthe} and for $10^{19}<x\le e^{L_1}$ we apply \cite[Corollary~1.3]{ChirreHelfgott},
which gives \[ E_\psi(x)\le \pi/(3\cdot10^{12})+113.67\,x^{-1/2}.\]
Finally, for $x>e^{L_1}$ we apply Theorem~\ref{thm:assympquasi} with $L_0=\max(L_1,\log(x_0))$.

\appendix

\section{Tables of upper bounds}\label{app:tables}

The two tables below are referred to from Sections \ref{sec:zeros1} and
\ref{sec:pnt}; they are collected here because of their length.

\begin{table}
\caption{Upper bounds for $E_1(x)$ and $E_2(x)$, valid for all $x$ with $\log(x)$ at least the value in the first column}
\label{tableBoundE1E2}
\begin{tabular}{|l|l|l|}
\hline
$\log(x)$ & $E_{1}(x)$ & $E_{2}(x)$ \\
\hline
$ 10000 $&$ 123.61 $&$ 4.4675 $\\
$ 11000 $&$ 113.71 $&$ 3.9377 $\\
$ 12000 $&$ 104.62 $&$ 3.4824 $\\
$ 13000 $&$ 96.278 $&$ 3.0933 $\\
$ 14000 $&$ 88.641 $&$ 2.7669 $\\
$ 15000 $&$ 81.648 $&$ 2.4798 $\\
$ 16000 $&$ 75.248 $&$ 2.2268 $\\
$ 17000 $&$ 69.388 $&$ 2.0032 $\\
$ 18000 $&$ 64.023 $&$ 1.8051 $\\
$ 19000 $&$ 59.109 $&$ 1.6293 $\\
$ 20000 $&$ 54.607 $&$ 1.4729 $\\
$ 30000 $&$ 30.995 $&$ 0.64565 $\\
$ 40000 $&$ 14.882 $&$ 0.27342 $\\
$ 50000 $&$ 7.6377 $&$ 0.12560 $\\
$ 60000 $&$ 4.0883 $&$ 0.061255 $\\
$ 70000 $&$ 2.2653 $&$ 0.031310 $\\
$ 80000 $&$ 1.2911 $&$ 0.016624 $\\
$ 90000 $&$ 0.75377 $&$ 0.0091116 $\\
$ 100000 $&$ 0.44937 $&$ 0.0051313 $\\
$ 110000 $&$ 0.27287 $&$ 0.0029585 $\\
$ 120000 $&$ 0.16844 $&$ 0.0017413 $\\
$ 130000 $&$ 0.10551 $&$ 0.0010439 $\\
$ 140000 $&$ 0.066983 $&$ 0.00063611 $\\
$ 150000 $&$ 0.043044 $&$ 0.00039344 $\\
$ 160000 $&$ 0.027971 $&$ 0.00024665 $\\
$ 170000 $&$ 0.018364 $&$ 0.00015656 $\\
$ 180000 $&$ 0.012172 $&$ 0.00010050 $\\
$ 190000 $&$ 0.0081395 $&$ 0.000065199 $\\
$ 200000 $&$ 0.0054881 $&$ 0.000042711 $\\
$ 300000 $&$ 0.00019328 $&$ 1.1384 \times 10^{-6} $\\
$ 400000 $&$ 8.1961 \times 10^{-6} $&$ 4.1435 \times 10^{-8} $\\
$ 500000 $&$ 4.8391 \times 10^{-7} $&$ 2.1657 \times 10^{-9} $\\
$ 600000 $&$ 3.6292 \times 10^{-8} $&$ 1.4676 \times 10^{-10} $\\
$ 700000 $&$ 3.2717 \times 10^{-9} $&$ 1.2129 \times 10^{-11} $\\
$ 800000 $&$ 3.4182 \times 10^{-10} $&$ 1.1744 \times 10^{-12} $\\
$ 900000 $&$ 4.0341 \times 10^{-11} $&$ 1.2955 \times 10^{-13} $\\
$ 1000000 $&$ 5.2776 \times 10^{-12} $&$ 1.5948 \times 10^{-14} $\\
$ 10000000 $&$ 7.7284 \times 10^{-52} $&$ 5.6992 \times 10^{-55} $\\
$ 100000000 $&$ 7.6430 \times 10^{-191} $&$ 1.2530 \times 10^{-194} $\\
$ 1000000000 $&$ 2.8388 \times 10^{-675} $&$ 9.1663 \times 10^{-680} $\\
$ 10000000000 $&$ 8.3863 \times 10^{-2379} $&$ 4.5025 \times 10^{-2384} $\\
$ 100000000000 $&$ 1.1247 \times 10^{-8432} $&$ 7.8272 \times 10^{-8439} $\\
$ 1000000000000 $&$ 1.3275 \times 10^{-30155} $&$ 8.2271 \times 10^{-30163} $\\
$ 10000000000000 $&$ 5.6661 \times 10^{-108762} $&$ 1.7614 \times 10^{-108770} $\\
$ 100000000000000 $&$ 3.9456 \times 10^{-395273} $&$ 2.5444 \times 10^{-395283} $\\
$ 1000000000000000 $&$ 1.2385 \times 10^{-1446121} $&$ 4.2396 \times 10^{-1446134} $\\
\hline
\end{tabular}
\end{table}

\begin{table}
\caption{Upper bounds for the normalised error terms $\mathcal{E}_{\rm per}$, $\mathcal{E}_{1/2}$, $\mathcal{E}_{\rm mid}$, $\mathcal{E}_1$, $\mathcal{E}_2$ and $\mathcal{E}$, for $10^4 \leqslant \log(x) < 10^6$, valid for all $x$ with $\log(x)$ at least the value in the first column}
\label{tableBound}
\setlength{\tabcolsep}{2pt}
\resizebox{\textwidth}{!}{%
\begin{tabular}{ |l|l|l|l|l|l|l| }
\hline
 $\log(x)$ & $\mathcal{E}_{\rm per}(x)$ & $\mathcal{E}_{1/2}(x)$ & $\mathcal{E}_{\rm mid}(x)$ &$\mathcal{E}_{1}(x)$ & $\mathcal{E}_{2}(x)$ & $\mathcal{E}(x)$ \\
\hline
$ 10000 $&$ 0.129 $&$ 2.89 \,\mathord{\cdot}\, 10^{-1070} $&$ 164.9 $&$ 207.1 $&$ 7.485 $&$ 5216 $\\
$ 11000 $&$ 0.0933 $&$ 5.62 \,\mathord{\cdot}\, 10^{-1178} $&$ 164.9 $&$ 195.6 $&$ 6.771 $&$ 4922 $\\
$ 12000 $&$ 0.0687 $&$ 1.01 \,\mathord{\cdot}\, 10^{-1285} $&$ 164.9 $&$ 184.4 $&$ 6.139 $&$ 4641 $\\
$ 13000 $&$ 0.0511 $&$ 1.66 \,\mathord{\cdot}\, 10^{-1393} $&$ 164.9 $&$ 173.9 $&$ 5.585 $&$ 4373 $\\
$ 14000 $&$ 0.0383 $&$ 2.58 \,\mathord{\cdot}\, 10^{-1501} $&$ 164.9 $&$ 163.8 $&$ 5.112 $&$ 4119 $\\
$ 15000 $&$ 0.0290 $&$ 3.77 \,\mathord{\cdot}\, 10^{-1609} $&$ 163.1 $&$ 154.3 $&$ 4.685 $&$ 3878 $\\
$ 16000 $&$ 0.0220 $&$ 5.24 \,\mathord{\cdot}\, 10^{-1717} $&$ 161.2 $&$ 145.3 $&$ 4.299 $&$ 3652 $\\
$ 17000 $&$ 0.0169 $&$ 6.94 \,\mathord{\cdot}\, 10^{-1825} $&$ 158.7 $&$ 136.8 $&$ 3.949 $&$ 3438 $\\
$ 18000 $&$ 0.0130 $&$ 8.81 \,\mathord{\cdot}\, 10^{-1933} $&$ 155.7 $&$ 128.9 $&$ 3.632 $&$ 3237 $\\
$ 19000 $&$ 0.0101 $&$ 1.08 \,\mathord{\cdot}\, 10^{-2040} $&$ 152.3 $&$ 121.4 $&$ 3.345 $&$ 3048 $\\
$ 20000 $&$ 7.84 \,\mathord{\cdot}\, 10^{-3} $&$ 1.27 \,\mathord{\cdot}\, 10^{-2148} $&$ 148.7 $&$ 114.3 $&$ 3.083 $&$ 2871 $\\
$ 30000 $&$ 8.05 \,\mathord{\cdot}\, 10^{-4} $&$ 1.46 \,\mathord{\cdot}\, 10^{-3228} $&$ 107.2 $&$ 78.31 $&$ 1.632 $&$ 1963 $\\
$ 40000 $&$ 1.10 \,\mathord{\cdot}\, 10^{-4} $&$ 2.59 \,\mathord{\cdot}\, 10^{-4309} $&$ 72.13 $&$ 43.55 $&$ 0.8001 $&$ 1091 $\\
$ 50000 $&$ 1.82 \,\mathord{\cdot}\, 10^{-5} $&$ 1.39 \,\mathord{\cdot}\, 10^{-5390} $&$ 47.58 $&$ 25.48 $&$ 0.4190 $&$ 637.8 $\\
$ 60000 $&$ 3.46 \,\mathord{\cdot}\, 10^{-6} $&$ 3.17 \,\mathord{\cdot}\, 10^{-6472} $&$ 31.28 $&$ 15.37 $&$ 0.2303 $&$ 384.6 $\\
$ 70000 $&$ 7.31 \,\mathord{\cdot}\, 10^{-7} $&$ 3.74 \,\mathord{\cdot}\, 10^{-7554} $&$ 20.65 $&$ 9.514 $&$ 0.1315 $&$ 238.0 $\\
$ 80000 $&$ 1.69 \,\mathord{\cdot}\, 10^{-7} $&$ 2.63 \,\mathord{\cdot}\, 10^{-8636} $&$ 13.72 $&$ 6.016 $&$ 0.07746 $&$ 150.5 $\\
$ 90000 $&$ 4.18 \,\mathord{\cdot}\, 10^{-8} $&$ 1.20 \,\mathord{\cdot}\, 10^{-9718} $&$ 9.182 $&$ 3.874 $&$ 0.04683 $&$ 96.85 $\\
$ 100000 $&$ 1.11 \,\mathord{\cdot}\, 10^{-8} $&$ 3.81 \,\mathord{\cdot}\, 10^{-10801} $&$ 6.199 $&$ 2.535 $&$ 0.02895 $&$ 63.36 $\\
$ 110000 $&$ 3.07 \,\mathord{\cdot}\, 10^{-9} $&$ 8.87 \,\mathord{\cdot}\, 10^{-11884} $&$ 4.219 $&$ 1.683 $&$ 0.01825 $&$ 42.05 $\\
$ 120000 $&$ 8.91 \,\mathord{\cdot}\, 10^{-10} $&$ 1.58 \,\mathord{\cdot}\, 10^{-12966} $&$ 2.895 $&$ 1.132 $&$ 0.01170 $&$ 28.27 $\\
$ 130000 $&$ 2.70 \,\mathord{\cdot}\, 10^{-10} $&$ 2.21 \,\mathord{\cdot}\, 10^{-14049} $&$ 2.002 $&$ 0.7694 $&$ 7.612 \,\mathord{\cdot}\, 10^{-3} $&$ 19.23 $\\
$ 140000 $&$ 8.49 \,\mathord{\cdot}\, 10^{-11} $&$ 2.50 \,\mathord{\cdot}\, 10^{-15132} $&$ 1.395 $&$ 0.5288 $&$ 5.022 \,\mathord{\cdot}\, 10^{-3} $&$ 13.21 $\\
$ 150000 $&$ 2.76 \,\mathord{\cdot}\, 10^{-11} $&$ 2.33 \,\mathord{\cdot}\, 10^{-16215} $&$ 0.9780 $&$ 0.3670 $&$ 3.355 \,\mathord{\cdot}\, 10^{-3} $&$ 9.166 $\\
$ 160000 $&$ 9.24 \,\mathord{\cdot}\, 10^{-12} $&$ 1.83 \,\mathord{\cdot}\, 10^{-17298} $&$ 0.6907 $&$ 0.2570 $&$ 2.266 \,\mathord{\cdot}\, 10^{-3} $&$ 6.417 $\\
$ 170000 $&$ 3.18 \,\mathord{\cdot}\, 10^{-12} $&$ 1.23 \,\mathord{\cdot}\, 10^{-18381} $&$ 0.4909 $&$ 0.1814 $&$ 1.547 \,\mathord{\cdot}\, 10^{-3} $&$ 4.529 $\\
$ 180000 $&$ 1.13 \,\mathord{\cdot}\, 10^{-12} $&$ 7.12 \,\mathord{\cdot}\, 10^{-19465} $&$ 0.3510 $&$ 0.1290 $&$ 1.066 \,\mathord{\cdot}\, 10^{-3} $&$ 3.221 $\\
$ 190000 $&$ 4.05 \,\mathord{\cdot}\, 10^{-13} $&$ 3.62 \,\mathord{\cdot}\, 10^{-20548} $&$ 0.2524 $&$ 0.09240 $&$ 7.402 \,\mathord{\cdot}\, 10^{-4} $&$ 2.308 $\\
$ 2\,\mathord{\cdot}\,10^5 $&$ 1.50 \,\mathord{\cdot}\, 10^{-13} $&$ 1.63 \,\mathord{\cdot}\, 10^{-21631} $&$ 0.1825 $&$ 0.06663 $&$ 5.186 \,\mathord{\cdot}\, 10^{-4} $&$ 1.664 $\\
$ 3\,\mathord{\cdot}\,10^5 $&$ 1.76 \,\mathord{\cdot}\, 10^{-17} $&$ 2.83 \,\mathord{\cdot}\, 10^{-32467} $&$ 9.140 \,\mathord{\cdot}\, 10^{-3} $&$ 4.503 \,\mathord{\cdot}\, 10^{-3} $&$ 2.653 \,\mathord{\cdot}\, 10^{-5} $&$ 0.1124 $\\
$ 4\,\mathord{\cdot}\,10^5 $&$ 6.64 \,\mathord{\cdot}\, 10^{-21} $&$ 6.47 \,\mathord{\cdot}\, 10^{-43306} $&$ 6.456 \,\mathord{\cdot}\, 10^{-4} $&$ 3.180 \,\mathord{\cdot}\, 10^{-4} $&$ 1.608 \,\mathord{\cdot}\, 10^{-6} $&$ 7.931 \,\mathord{\cdot}\, 10^{-3} $\\
$ 5\,\mathord{\cdot}\,10^5 $&$ 5.44 \,\mathord{\cdot}\, 10^{-24} $&$ 1.95 \,\mathord{\cdot}\, 10^{-54146} $&$ 5.813 \,\mathord{\cdot}\, 10^{-5} $&$ 2.961 \,\mathord{\cdot}\, 10^{-5} $&$ 1.325 \,\mathord{\cdot}\, 10^{-7} $&$ 7.384 \,\mathord{\cdot}\, 10^{-4} $\\
$ 6\,\mathord{\cdot}\,10^5 $&$ 7.80 \,\mathord{\cdot}\, 10^{-27} $&$ 2.57 \,\mathord{\cdot}\, 10^{-64988} $&$ 6.290 \,\mathord{\cdot}\, 10^{-6} $&$ 3.368 \,\mathord{\cdot}\, 10^{-6} $&$ 1.362 \,\mathord{\cdot}\, 10^{-8} $&$ 8.397 \,\mathord{\cdot}\, 10^{-5} $\\
$ 7\,\mathord{\cdot}\,10^5 $&$ 1.73 \,\mathord{\cdot}\, 10^{-29} $&$ 3.06 \,\mathord{\cdot}\, 10^{-75831} $&$ 7.869 \,\mathord{\cdot}\, 10^{-7} $&$ 4.468 \,\mathord{\cdot}\, 10^{-7} $&$ 1.656 \,\mathord{\cdot}\, 10^{-9} $&$ 1.114 \,\mathord{\cdot}\, 10^{-5} $\\
$ 8\,\mathord{\cdot}\,10^5$&$ 5.39 \,\mathord{\cdot}\, 10^{-32} $&$ 5.28 \,\mathord{\cdot}\, 10^{-86675} $&$ 1.109 \,\mathord{\cdot}\, 10^{-7} $&$ 6.706 \,\mathord{\cdot}\, 10^{-8} $&$ 2.304 \,\mathord{\cdot}\, 10^{-10} $&$ 1.672 \,\mathord{\cdot}\, 10^{-6} $\\
$ 9\,\mathord{\cdot}\,10^5 $&$ 2.25 \,\mathord{\cdot}\, 10^{-34} $&$ 1.86 \,\mathord{\cdot}\, 10^{-97519} $&$ 1.726 \,\mathord{\cdot}\, 10^{-8} $&$ 1.115 \,\mathord{\cdot}\, 10^{-8} $&$ 3.580 \,\mathord{\cdot}\, 10^{-11} $&$ 2.780 \,\mathord{\cdot}\, 10^{-7} $\\
\hline
\end{tabular}
}
\end{table}

\begin{table}
\caption{Upper bounds for the normalised error terms $\mathcal{E}_{\rm per}$, $\mathcal{E}_{1/2}$, $\mathcal{E}_{\rm mid}$, $\mathcal{E}_1$, $\mathcal{E}_2$ and $\mathcal{E}$, for $\log(x) \geqslant 10^6$, valid for all $x$ with $\log(x)$ at least the value in the first column}
\label{tableBoundLarge}
\setlength{\tabcolsep}{2pt}
\resizebox{\textwidth}{!}{%
\begin{tabular}{ |l|l|l|l|l|l|l| }
\hline
 $\log(x)$ & $\mathcal{E}_{\rm per}(x)$ & $\mathcal{E}_{1/2}(x)$ & $\mathcal{E}_{\rm mid}(x)$ &$\mathcal{E}_{1}(x)$ & $\mathcal{E}_{2}(x)$ & $\mathcal{E}(x)$ \\
\hline
$ 10^{6} $&$ <10^{-10^{1}} $&$ <10^{-10^{5}} $&$ 2.923 \,\mathord{\cdot}\, 10^{-9} $&$ 2.021 \,\mathord{\cdot}\, 10^{-9} $&$ 6.105 \,\mathord{\cdot}\, 10^{-12} $&$ 5.037 \,\mathord{\cdot}\, 10^{-8} $\\
$ 10^{7} $&$ <10^{-10^{2}} $&$ <10^{-10^{6}} $&$ 1.331 \,\mathord{\cdot}\, 10^{-44} $&$ 8.648 \,\mathord{\cdot}\, 10^{-43} $&$ 6.378 \,\mathord{\cdot}\, 10^{-46} $&$ 2.155 \,\mathord{\cdot}\, 10^{-41} $\\
$ 10^{8} $&$ <10^{-10^{2}} $&$ <10^{-10^{7}} $&$ 1.644 \,\mathord{\cdot}\, 10^{-167} $&$ 1.121 \,\mathord{\cdot}\, 10^{-158} $&$ 1.838 \,\mathord{\cdot}\, 10^{-162} $&$ 2.792 \,\mathord{\cdot}\, 10^{-157} $\\
$ 10^{9} $&$ <10^{-10^{3}} $&$ <10^{-10^{8}} $&$ 7.149 \,\mathord{\cdot}\, 10^{-591} $&$ 1.352 \,\mathord{\cdot}\, 10^{-559} $&$ 4.363 \,\mathord{\cdot}\, 10^{-564} $&$ 3.366 \,\mathord{\cdot}\, 10^{-558} $\\
$ 10^{10} $&$ <10^{-10^{3}} $&$ <10^{-10^{9}} $&$ 3.085 \,\mathord{\cdot}\, 10^{-2061} $&$ 7.434 \,\mathord{\cdot}\, 10^{-1959} $&$ 3.991 \,\mathord{\cdot}\, 10^{-1964} $&$ 1.852 \,\mathord{\cdot}\, 10^{-1957} $\\
$ 10^{11} $&$ <10^{-10^{4}} $&$ <10^{-10^{10}} $&$ 7.336 \,\mathord{\cdot}\, 10^{-7227} $&$ 3.699 \,\mathord{\cdot}\, 10^{-6896} $&$ 2.575 \,\mathord{\cdot}\, 10^{-6902} $&$ 9.213 \,\mathord{\cdot}\, 10^{-6895} $\\
$ 10^{12} $&$ <10^{-10^{5}} $&$ <10^{-10^{11}} $&$ 2.983 \,\mathord{\cdot}\, 10^{-25573} $&$ 1.346 \,\mathord{\cdot}\, 10^{-24496} $&$ 8.339 \,\mathord{\cdot}\, 10^{-24504} $&$ 3.351 \,\mathord{\cdot}\, 10^{-24495} $\\
$ 10^{13} $&$ <10^{-10^{5}} $&$ <10^{-10^{12}} $&$ 2.240 \,\mathord{\cdot}\, 10^{-91356} $&$ 2.806 \,\mathord{\cdot}\, 10^{-87802} $&$ 8.723 \,\mathord{\cdot}\, 10^{-87811} $&$ 6.989 \,\mathord{\cdot}\, 10^{-87801} $\\
$ 10^{14} $&$ <10^{-10^{6}} $&$ <10^{-10^{13}} $&$ 3.336 \,\mathord{\cdot}\, 10^{-329167} $&$ 3.720 \,\mathord{\cdot}\, 10^{-317265} $&$ 2.399 \,\mathord{\cdot}\, 10^{-317275} $&$ 9.264 \,\mathord{\cdot}\, 10^{-317264} $\\
$ 10^{15} $&$ <10^{-10^{6}} $&$ <10^{-10^{14}} $&$ 2.421 \,\mathord{\cdot}\, 10^{-1194944} $&$ 4.612 \,\mathord{\cdot}\, 10^{-1154558} $&$ 1.579 \,\mathord{\cdot}\, 10^{-1154570} $&$ 1.149 \,\mathord{\cdot}\, 10^{-1154556} $\\
\hline
\end{tabular}
}
\end{table}

\end{document}